\documentclass[11pt]{article}
\usepackage{mathptmx}
\usepackage{authblk}
\usepackage[T1]{fontenc}
\usepackage{inputenc}
\usepackage{babel}
\usepackage{array}
\usepackage{float}
\usepackage{textcomp}
\usepackage{multirow}
\usepackage{graphicx}
\usepackage{natbib}
\usepackage{hyperref}
\usepackage{graphicx} 
\usepackage{amsmath}
\usepackage{amssymb}
\usepackage{setspace}
\usepackage[linesnumbered,ruled,vlined]{algorithm2e}
\usepackage[most]{tcolorbox}
\usepackage{geometry}
\usepackage{soul}
\usepackage{empheq}
\usepackage{moresize}
\allowdisplaybreaks

\usepackage{amsthm}
\usepackage{authblk}
 \hypersetup{
    colorlinks=true,
    linkcolor=blue,
    filecolor=blue,      
    urlcolor=cyan,
    citecolor=blue,
    pdftitle={Overleaf Example},
    pdfpagemode=FullScreen,
    }
\newtheoremstyle{exampstyle1}
  {8pt} 
  {8pt} 
  {} 
  {} 
  {\bfseries} 
  {.} 
  {3pt} 
  {} 
\newtheorem{theorem}{Theorem}[section]
\newtheorem{lemma}[theorem]{Lemma}

\theoremstyle{exampstyle1}\newtheorem{assumption}{Assumption}
\newtheorem{definition}{Definition}[subsection]
\newtheorem{Remark}{Remark}
\title{Statistical inference using debiased group graphical lasso for
multiple sparse precision matrices}
\vspace{1cm}
\author{Sayan Ranjan Bhowal\thanks{Corresponding author: Theoretical Statistics and Mathematics Unit, Indian Statistical Institute, Kolkata} \;  Debashis Paul\thanks{Applied Statistics Division, Indian Statistical Institute, Kolkata} \; Gopal K Basak\thanks{Theoretical Statistics and Mathematics Unit, Indian Statistical Institute, Kolkata} \; Samarjit Das\thanks{Economic-Research Unit, Indian Statistical Institute, Kolkata}}

\date{}

\begin{document}
\maketitle
\begin{abstract}
    Debiasing group graphical lasso estimates enables statistical inference when multiple Gaussian graphical models share a common sparsity pattern. We analyze the estimation properties of group graphical lasso, establishing convergence rates and model selection consistency under irrepresentability conditions. Based on these results, we construct debiased estimators that are asymptotically Gaussian, allowing hypothesis testing for linear combinations of precision matrix entries across populations. We also investigate regimes where irrepresentibility conditions does not hold, showing that consistency can still be attained in moderately high-dimensional settings. Simulation studies confirm the theoretical results, and applications to real datasets demonstrate the practical utility of the method.    
\end{abstract}
\vspace{0.5cm}
\newpage
\section{\large Introduction}
The Gaussian graphical models are the most common graphical models that have gone through intensive research for a long time. Suppose $\boldsymbol{x}=(x_{1},x_{2},\ldots,x_{p})\sim N(\boldsymbol{0},\boldsymbol{\Sigma})$ is a $p$-variate random vector having multivariate Gaussian distribution. Here $\boldsymbol{\Sigma}$ is a positive definite matrix, and we have the precision matrix $\boldsymbol{\Omega}=\boldsymbol{\Sigma}^{-1}$. This precision matrix tells us about the connectivity of the coordinates. Specifically, any two variables, say $x_{i}$ and $x_{j}$ are independent given the other variables if $\boldsymbol{\Omega}_{ij}=0$. In the graphical model, this corresponds to the interpretation that if we assume the coordinates as the vertices of a graphical model, then the vertices $x_{i}$ and $x_{j}$ do not share an edge. If we have $n$ observations, then we have the sample covariance matrix as $\hat{\Sigma}=\frac{1}{n}\boldsymbol{X}^{T}\boldsymbol{X}$, where $\boldsymbol{X}_{p\times n}=(\boldsymbol{x}_{1},\ldots,\boldsymbol{x}_{n})$ is the data matrix. However, the sample covariance matrix often leads to poor estimation of $\boldsymbol{\Sigma}$ when the sample size $n$ is comparable to the dimension $p$. Moreover, in the high-dimensional regime, the sample covariance matrix is not even invertible, and hence, cannot provide a reliable estimate of the precision matrix depending on just the sample covariance matrix.

To solve this problem, the estimation is done via penalized maximum log likelihood optimization. The problem can be written as
\begin{eqnarray}\label{as 1}
\notag
\boldsymbol{\hat{\Omega}} & = & \underset{\boldsymbol{\Omega}\in S_{++}}{arg\ max}\ \{log\ det\ \boldsymbol{\Omega}-trace(\boldsymbol{\Omega}\hat{\Sigma})-\lambda||\boldsymbol{\Omega}^{-}||_{1}\}\\
 & = & \underset{\boldsymbol{\Omega}\in S_{++}}{arg\ min}\ \{-log\ det\ \boldsymbol{\Omega}+trace(\boldsymbol{\Omega}\hat{\Sigma})+\lambda||\boldsymbol{\Omega}^{-}||_{1}\}.
\end{eqnarray}

The idea of using the penalty term is to penalize the off-diagonal elements of $\boldsymbol{\Omega}$, making the solution sparse and positive definite. This optimization is what is known as graphical lasso. The techniques to solve the optimization problem have been discussed in \cite{hastie2015statistical}. Theoretical details of the solution have been studied thoroughly; see e.g.,\cite{fan2009network}, \cite{10.1214/11-EJS631}, \cite{10.1214/12-AOS1041}, \cite{yuan2019constrained}. However, the introduction of the penalty term makes the estimators biased. To do statistical inference, one needs to debias the estimates. \cite{10.1214/15-EJS1031} studied the theoretical properties after debiasing the estimators by inverting the Karush-Kuhn-Tucker (KKT) conditions.

Recently, the study of multiple Gaussian graphical models has attracted the attention of researchers. Estimating the precision matrices of multiple models sharing common characteristics has continued to gain interest. As a result, several attractive estimation methods have been developed. The common characteristic mostly deals with the fact that the zero pattern of the precision matrices remains constant throughout the models. Incorporating the additional information about the commonality structure improves the joint estimation of the precision matrices rather than separately estimating them. \cite{guo2011joint} proposed decomposing each precision matrix into the Schur-Hadamard product of two matrices, where one matrix remains constant throughout the models, while the other remains variable. The objective is to estimate the matrices so that each of the matrices is as sparse as possible. \cite{danaher2014joint} introduced the fused graphical lasso and group graphical lasso methods. The fused graphical lasso deals with the estimation of the precision matrices that should not be too different from each other, whereas the group graphical lasso estimates the precision matrices in a way that the sparsity remains intact within the models. \cite{lee2015joint} decomposed the precision matrices into the sum of two matrices, where one remains constant throughout the models and the other changes. They used CLIME estimation and hence can also be used for non-Gaussian data. \cite{ma2016joint} proposed jointly estimating the precision matrices from the neighborhood selection method, introduced by \cite{10.1214/009053606000000281}. Here, each coordinate can be represented as a regression model of the other coordinates, and the regression coefficients represent the dependence of the coordinate on the others. Several other methods have been developed over recent years in jointly estimating the precision matrices over multiple models. We refer to the survey paper by \cite{tsai2022joint} for an overview of the methods that have been proposed till date.

Though estimation of sparse precision matrix has been studied, research on statistical inference remains limited. \cite{10.1214/13-AOS1169} developed a method by using a debiased estimate of the partial correlation coefficient to test whether there is significant conditional dependence on two coordinates for a single model. They used FDR control techniques from \cite{efron2007correlation} to select the cut-off for the multiple testing of dependence of all pairs of coordinates. Statistical inference of two networks was introduced by \cite{xia2015testing}, and later, multiple networks were studied by \cite{10.1214/17-AOS1539} based on the same idea of using the debiased estimate of the partial correlation coefficient. However, studies of statistical inference based on estimates obtained from penalized maximum likelihood equations are much fewer. Recently \cite{zhang2024application} addressed the problem of applying fused graphical lasso to infer multiple sparse precision matrices. However, their work is not truly high-dimensional; rather, their work is moderately high-dimensional, where the dimension increases but at a slower rate than the sample size. In this work, we develop debiased estimates of true precision matrices from group graphical lasso, and then we obtain a test statistic based on the debiased estimates for statistical inference. We focus on the regime where we assume our precision matrix of each population satisfies the irrepresentability condition and the between-group irrepresentability condition, and also exhibits a similar sparsity pattern across the populations. Violations of the shared sparsity pattern complicate the strict dual feasibility conditions. However, we shall also address the situation without the assumption of irrepresentability and between-group irrepresentability condition, and without a similar sparsity pattern across the populations. Without these assumptions, we demonstrate consistency only in the moderately high-dimensional regime, where the dimension grows at a rate much slower than the sample size. To our knowledge, debiased estimates of group graphical lasso have not yet been developed or worked with. 

The rest of the paper is organized as follows. Subsection \ref{Subsection 1.2} introduces the model setup along with the group graphical lasso estimators and the assumptions that shall be worked with throughout the paper. Convergence of the precision matrices shall be examined in Section \ref{Section 2}. Section \ref{Section 3} illustrates the methods and theories for statistical inference using the group graphical lasso estimators. Section \ref{Section 4} is dedicated to the analysis of the group graphical lasso estimators without assuming the irrepresentability conditions, the between-group irrepresentability conditions, and a similar sparsity pattern of the population precision matrices. Section \ref{Section 5} verifies the theoretical properties using synthetic and real-world datasets.
\subsection{\normalsize Notations}\label{Subsection 1.1}
Denote $\boldsymbol{A}=(A_{ij})_{i,j=1}^{p}$ as a matrix, where $A_{ij}$ denote the $(i,j)^{th}$ entry of $\boldsymbol{A}$. We write $trace(\boldsymbol{A})$ to define the trace of a matrix$\boldsymbol{A}$, while $det(\boldsymbol{A})$ denotes the determinant of the matrix $\boldsymbol{A}$. We also denote $\boldsymbol{A}^{+}$ as $diag(\boldsymbol{A})$, whereas $\boldsymbol{A}^{-}=\boldsymbol{A}-\boldsymbol{A}^{+}$, the matrix containing only the off-diagonal elements of $\boldsymbol{A}$. For any vector $\boldsymbol{x}\in\mathbb{R}^{p}$, we denote the $d$-norm of $\boldsymbol{x}$, by $||\boldsymbol{x}||_{d}$, $d\in(0,\infty]$. We use $||\boldsymbol{A}||_{F}^{2}=\sum_{i,j}A_{ij}^{2}$, $||\boldsymbol{A}||_{\infty}=\underset{i,j}{max}|A_{ij}|$, $|||\boldsymbol{A}|||_{\infty}=\underset{i}{max}\sum_{j}|A_{ij}|$, and $|||\boldsymbol{A}|||_{1}=|||\boldsymbol{A}^{T}|||_{\infty}$ as the Frobenius norm, supremum norm, $l_{\infty}$-operator norm, and $l_{1}$-operator norm of a matrix $\boldsymbol{A}$. We also denote $||\boldsymbol{A}||_{1}=\sum_{i,j}|A_{ij}|$.

For sequences, we use $f_{n}=O(g_{n})$, if $f_{n}\leq cg_{n}$, for some $c<\infty$, and $f_{n}=\psi(g_{n})$, if $f_{n}>dg_{n}$, for some $d>0$. Then $f_{n}\asymp g_{n}$, if $f_{n}=O(g_{n})$, and $f_{n}=\psi(g_{n})$. The dimension $p$ is comparable to the sample size. We denote the set $\mathcal{S}_{++}^{p}=\{\boldsymbol{\Omega}:\boldsymbol{\Omega}=\boldsymbol{\Omega}^{T},\boldsymbol{\Omega}\succ0\}$as the set of all symmetric positive definite matrices. Moreover, we denote $S_{k}=\{(i,j):\boldsymbol{\Omega}_{ij}^{k}\neq0,i\neq j\}$ as the set of all pairs of coordinates sharing edges between them for the $k^{th}$ model and $s_{k}=|S_{k}|$ as the cardinality of $S_{k}$, representing the number of edges within the graphical model $k$. Here $\boldsymbol{\Omega}^{k}$ denotes the precision matrix for model $k$. Under the assumption that the edge set among the models remains constant, we have $S_k=S$, and $s_k=s$, for all $k=1,2,\ldots, K$. Define $d_{k}=\underset{j=1,2,\ldots,p}{max}|D_{j_{1}}^{k}|$, where $|D_{j_{1}}^{k}|=\{(i,j_{1}):\Omega_{ij_{1}}^{k}\neq0,i\neq j_{1}\}$. If we consider the graphical model of the corresponding Gaussian distribution, then $d_{k}$ represents the highest number of edges any vertex shares with the others, excluding the self-loop of each vertex. In other words, $d_{k}$ is the maximum number of coordinates a particular coordinate shares its conditional dependency with. Since the edge set among the models are same, we have $d_{k}=d$, for all $k=1,2,\ldots, K$.

\subsection{\normalsize Model setup}\label{Subsection 1.2}
In this subsection, we shall define the optimization problem that we should be working with, along with some definitions.
\begin{definition}\label{Definition 1.2.1}
    A random variable $X$ with $\mathbb{E}(X)=0$ is sub-Gaussian if
    \begin{equation}\label{as 2}
        \mathbb{E}(e^{X^2/K_1^2})\leq 2,
    \end{equation}
     for some $K_1>0$.
\end{definition}
 We know what the sub-Gaussian conditions for random vectors corresponding to a single population model are. We now generalize the conditions for multiple population models.
\begin{definition}\label{Definition 1.2.2}
    (Sub-Gaussian condition) Consider $\boldsymbol{x}^{k}=(x_1^k,x_2^k,\ldots,x_p^k)$, with $\mathbb{E}(\boldsymbol{x}^k)=0$, and covariance matrix $\boldsymbol{\Sigma}_0^{k}$, as random vectors from the population model $k,k=1,2,\ldots,K$. Then, $x_i^k/\sqrt{\boldsymbol{\Sigma}_{0ii}^{k}}$,$i=1,2,\ldots,p$, $k=1,2,\ldots,K$ are sub-Gaussian random variables with a common $K_1>0$, for all $k=1,2,\ldots,K$.
\end{definition}
\begin{definition}\label{Definition 1.2.3}
    (Sub-Gaussian vector condition) $\boldsymbol{x}^{k}$, with $\mathbb{E}(\boldsymbol{x}^k)=0$, and covariance matrix $\boldsymbol{\Sigma}_0^{k}$, from the population model $k$ satisfies the sub-Gaussianity vector condition if we have a constant $K_1>0$ such that
    \begin{equation}\label{as 3}
        \underset{\boldsymbol{a}\in \mathbb{R}^p:||\boldsymbol{a}||_2\leq 1}{sup}\mathbb{E}(e^{|\boldsymbol{a}'\boldsymbol{x}^{k}|^2/K_1^2})\leq2,
    \end{equation}
     for all $k=1,2,\ldots,K$.
\end{definition}
We define the Tail conditions for multiple population models as follows:
\begin{definition}\label{Definition 1.2.4}
    (Tail conditions) We say $\boldsymbol{x}^{k},k=1,2,\ldots,K$ satisfies the tail condition $\mathcal{T}(f,v_{k0})$, if we have a $v_{k0}\in (0,\infty],k=1,2,\ldots,K$, and a function $f:\mathbb{N}\times(0,\infty)\rightarrow (0,\infty)$, such that $(i,j)\in V\times V$,
    \begin{equation}\label{as 4}
        \mathbb{P}(|\hat{\boldsymbol{\Sigma}}_{ij}^{k}-\boldsymbol{\Sigma}_{0ij}^{k}|>\delta)\leq \frac{1}{f(n_k,\delta)},
    \end{equation}
     for all $\delta\in (0,\underset{k=1,2,\ldots,K}{min}\frac{1}{v_{k0}}]$, and $k=1,2,\ldots,K$.Here, $V$ is the set of coordinates corresponding to the random vector and $\hat{\boldsymbol{\Sigma}}^{k}$ and $\boldsymbol{\Sigma}_{0}^{k}$ are the sample covariance matrix and true population covariance matrix respectively for each population $k=1,2,\ldots, K$.
\end{definition}
For single population, Definition \ref{Definition 1.2.4} simplifies to the fact that $\boldsymbol{x}$ satisfies the tail condition $\mathcal{T}(f,v_0)$, if we have a $v_0\in (0,\infty]$, and a function $f:\mathbb{N}\times(0,\infty)\rightarrow (0,\infty)$, such that $(i,j)\in V\times V$, 
$$
\mathbb{P}(|\hat{\boldsymbol{\Sigma}}_{ij}^{k}-\boldsymbol{\Sigma}_{0ij}^{k}|>\delta)\leq \frac{1}{f(n,\delta)},$$
for all $\delta \in (0,\frac{1}{v_0}]$. Two important definitions as a result of this interpretation for the single population model are of the inverse functions as
$$
n_{f}(\delta,r)=\underset{n}{arg\ max}\{f(n,\delta)\leq r\},
$$
and $$
\delta_{f}(n,r)=\underset{\delta}{arg\ max}\{f(n,\delta)\leq r\}.$$ A simple understanding from the definitions of the inverse functions is
$$
n>n_{f}(\delta,r)\text{ for some }\delta>0\implies\delta_{f}(n,r)\leq \delta.$$ We refer to \cite{10.1214/11-EJS631} for a deeper discussion on the Tail condition and inverse functions for the single population model. We state the following lemma regarding Tail conditions corresponding to sub-Gaussian conditions (Definition \ref{Definition 1.2.2}).
\begin{lemma}\label{Lemma 1.1}
    Suppose $\boldsymbol{x}^{k},k=1,2,\ldots,K$, are each zero mean random vector with covariance matrix $\boldsymbol{\Sigma}_{0}^{k}$, satisfying the sub-Gaussian conditions (Definition \ref{Definition 1.2.2}), with a common $K_1>0$. If we have $n_k$ samples each from population model $k, k=1,2,\ldots,K$, the sample covariance matrices $\hat{\boldsymbol{\Sigma}}^{k}$, satisfies the tail conditions
    \begin{equation}\label{as 5}
        \mathbb{P}(|\hat{\boldsymbol{\Sigma}}_{ij}^{k}-\boldsymbol{\Sigma}_{0ij}^{k}|>\delta)\leq 4\ exp\bigg\{ -\frac{n_k \delta^2}{128(1+12K_{1}^{2})^{2}max\ (\Sigma_{0ii}^{k})^{2}}\bigg\},
    \end{equation}
    for all $\delta\in (0,\underset{k_1}{min}\ 8(1+12K_1^{2})\ max\ (\Sigma_{0ii}^{k_1}))$.
\end{lemma}
For the single population model, Lemma \ref{Lemma 1.1} results in satisfying the Tail conditions
$$
 \mathbb{P}(|\hat{\boldsymbol{\Sigma}}_{ij}^{k}-\boldsymbol{\Sigma}_{0ij}^{k}|>\delta)\leq 4\ exp\bigg\{ -\frac{n \delta^2}{128(1+12K_{1}^{2})^{2}max\ (\Sigma_{0ii}^{k})^{2}}\bigg\},
$$ for all $\delta\in (0,8(1+12K_1^{2})\ max\ (\Sigma_{0ii}))$. From the definition of an inverse function,
$$
\delta_{f}(n,r)=8(1+12K_1^{2})\ max\ (\Sigma_{0ii})\sqrt{2\frac{log\ 4r}{n}},$$
and
$$
n_{f}(\delta,r)=128(1+12K_{1}^{2})^{2}\ max\ (\Sigma_{0ii})^{2}\bigg(\frac{log\ 4r}{\delta^2}\bigg).$$
As a consequence of the definition of Tail condition (Definition \ref{Definition 1.2.4}) and Lemma \ref{Lemma 1.1}, if
\begin{equation}\label{as 6}
    \delta_{k}(r)=\delta_{f}(n_k,r)=8(1+12K_1^{2})\ max\ \Sigma_{0ii}^{k}\sqrt{2\frac{log\ 4r}{n_k}},
\end{equation}
and $n_k$ is such that $\delta_{k}=\delta_{f}(n_k,p^{\gamma})<\underset{k}{min}\ 8(1+12K_1^{2})\ max\ \Sigma_{0ii}^{k}$, for $k=1,2,\ldots,K$, then for every $\gamma>2$, we have the tail conditions as,
\begin{equation}\label{as 7}
    \mathbb{P}(||\hat{\boldsymbol{\Sigma}}^{k}-\boldsymbol{\Sigma}_{0}^{k}||_{\infty}>\delta)\leq \frac{1}{p^{\gamma-2}},
\end{equation}
where $\delta=max\ \delta_k$.

We now briefly discuss group graphical lasso.
The group graphical lasso is associated with the problem of jointly estimating the precision matrices from multiple Gaussian graphical models. Suppose we have $K$ Gaussian graphical models and $n_{k}$
observations from each of the models. Let us denote $\boldsymbol{x}_{i}^{k}$
as the $i^{th}$ observation from $k^{th}$ model, $i=1,2,\ldots,n_{k}$,
and $k=1,2,\ldots,K$. The observations $\boldsymbol{x}_{i}^{k}\sim N(\boldsymbol{0},\boldsymbol{\Sigma}^{k})$,
where $\boldsymbol{\Sigma}^{k}$ denotes the covariance matrix and
$\boldsymbol{\Omega}^{k}=\{\boldsymbol{\Sigma}^{k}\}^{-1}$ denotes
the precision matrix from model $k$. Define $\hat{\boldsymbol{\Sigma}}^{k}=\frac{1}{n_{k}}(\boldsymbol{X}^{k})^{T}\boldsymbol{X}^{k}$
as the sample covariance matrix from the model $k$, where $\boldsymbol{X}_{p\times n_{k}}^{k}=(\boldsymbol{x}_{1}^{k},\ldots,\boldsymbol{x}_{n_{k}}^{k})$
as the data matrix. The group graphical lasso by \cite{danaher2014joint}
is defined to be the solution to 
\begin{equation}\label{as 8}
    \underset{\boldsymbol{\Omega}=\{\boldsymbol{\Omega}^{1},\ldots,\boldsymbol{\Omega}^{K}\}}{max}\sum_{k=1}^{K}(log\ det(\boldsymbol{\Omega}^{k})-trace(\hat{\boldsymbol{\Sigma}}\boldsymbol{\Omega}^{k}))-P(\boldsymbol{\Omega}),
\end{equation}
 where $P(\boldsymbol{\Omega})=\lambda\sum_{k}||(\boldsymbol{\Omega}^{k})^{-}||_{1}+\rho\sum_{i\neq j}\sqrt{\sum_{k}(\Omega_{ij}^{k})^{2}}$.
In the penalty term $P(\boldsymbol{\Omega})$, the first term is used
to impose sparsity within the models, and the group lasso penalty
is applied so that it encourages similar sparsity patterns across the
models; see \cite{10.1111/j.1467-9868.2005.00532.x}. Here $\lambda$ and $\rho$
are non-negative regularization parameters that are mainly selected
using BIC or e-BIC for proper model fitting. However, theoretical analysis
of the method still needs to be done.
Some definitions of functions of the true precision matrix $\boldsymbol{\Omega}^{k}$ for each model $k$ shall now be given. We refer to \cite{10.1214/15-EJS1031} for a deeper analysis of the definitions.
Let $\kappa_{\boldsymbol{\Sigma}_{0}^{k}}$ be defined as $$\kappa_{\boldsymbol{\Sigma}_{0}^{k}}=|||\boldsymbol{\Sigma}_{0}^{k}|||_{\infty},$$ for all $k$. This $\kappa_{\boldsymbol{\Sigma}_{0}^{k}}$ measures the size of the entries of the covariance matrix of each model $k$. Note that the Hessian $\boldsymbol{\Gamma}^{k}$ of the negative likelihood function $l(\Omega)=\sum_{k=1}^{K}(trace(\hat{\boldsymbol{\Sigma}}^{k}\boldsymbol{\Omega}^{k})-log\ det(\boldsymbol{\Omega}^{k}))$ in matrix form is given as $$\boldsymbol{\Gamma}(\boldsymbol{\Omega}^{k})=\boldsymbol{\Sigma}^{k}\otimes \boldsymbol{\Sigma}^{k},$$
 where $\otimes$ represents the Kronecker product of two matrices. The Hessian at the true parameter $\boldsymbol{\Omega}_{0}^{k}$ is represented as $\boldsymbol{\Gamma}_{0}^{k}$. Let us define $\boldsymbol{\Gamma}^{k}_{0T_{k}T_{k}'}$ with rows and columns of $\boldsymbol{\Gamma}_{0}^{k}$ indexed by $T_{k}$ and $T_{k}'$, where $T_{k}$ and $T_{k}'$ are subsets of $V$. Define  $$\boldsymbol{\Gamma}^{k}_{0SS}=[\boldsymbol{\Sigma}_{0}^{k}\otimes \boldsymbol{\Sigma}_{0}^{k}]_{SS},$$ and $\kappa_{\boldsymbol{\Gamma}_{0}^{k}}=|||(\boldsymbol{\Gamma}^{k}_{0SS})^{-1}|||_{\infty}$. In light of the above definition of $\kappa_{\boldsymbol{\Gamma}^{k}_{0}}$, we define the irrepresentability condition and between-group irrepresentability condition for multiple populations as follows:
\begin{assumption}\label{Assumption 1}
(Irrepresentability condition) We have an $\alpha\in(0,1]$, such that
\begin{equation}\label{as 9}
	\underset{e \in S^{c}}{max}||\boldsymbol{\Gamma}_{0eS}^{k}(\boldsymbol{\Gamma}_{0SS}^{k})^{-1}||_{1}\leq 1-\alpha,
\end{equation}
 for every $k=1,2,\ldots,K$.
\end{assumption}
The interpretation of this assumption is that no edge variable that is not included in the edge set $S$ is highly correlated with variables in the edge set $S$.

\begin{assumption}\label{Assumption 2}
(Between-group irrepresentability condition) For penalty parameters $\lambda$ and $\rho$, with $K$ as the number of population models, $\alpha$ from (\ref{as 9}), and $\psi\in(0,1)$, we have
\begin{equation}\label{as 10}
\underset{e\in S^{c}}{max}\sqrt{\sum_{k=1}^{K}(\boldsymbol{\Gamma}_{0eS}^{k}(\boldsymbol{\Gamma}_{0SS}^{k})^{-1}\boldsymbol{1})^2}<\frac{\frac{\rho}{(\lambda+\rho)(1-\psi)}-\frac{\alpha\sqrt{K}}{4}}{1+(\alpha/4)}.
\end{equation}
\end{assumption}

 This assumption limits the overall correlation of the $K$ population models among the variables in the edge set $S$ with those not in the edge set $S$. Note that for some $\alpha$, Assumption  \ref{Assumption 1} shall be enough. Specifically, if we have an $\alpha$ such that $$1-\frac{\alpha}{2}-\frac{\alpha^2}{4}\leq \frac{\rho}{\sqrt{K}(\lambda+\rho)(1-\psi)},$$ the irrepresentability condition implies the between-group irrepresentability condition.
 Along with the irrepresentability condition and the between-group irrepresentability condition, let us assume the boundedness of the eigenvalues of the population precision matrices. The assumption stated is as follows:
\begin{assumption}\label{Assumption 3}
     (Bounded eigenvalues) There exists a constant $L$
such that 
\begin{equation*}
    0<L<\Lambda_{min}(\boldsymbol{\Omega}_{0}^{k})\leq\Lambda_{max}(\boldsymbol{\Omega}_{0}^{k})<1/L<\infty,
\end{equation*}
 where $\Lambda_{min}$ and $\Lambda_{max}$ denotes the minimum and maximum eigenvalues, for all $k=1,2,\ldots,K$.
\end{assumption}

\begin{assumption}\label{Assumption 4}
    (Sample size ratio) Suppose $n_k$ is the number of samples from model $k, k=1,2,\ldots,K$, and let $n=\underset{k=1,2,\ldots,K}{min}n_k$. Then the sample sizes are of the same order as $n$, i.e. $n/n_{k}\rightarrow c\in (0,1)$, for all $k=1,2,\ldots,K$.
\end{assumption}
Our analysis will be based on the above assumptions.
\section{\large Convergence of precision matrices}\label{Section 2}
The optimization problem
(\ref{as 8}) can be equivalently represented as 
\begin{equation}\label{as 11}
\underset{\{\boldsymbol{\Omega}^{k}\in\mathcal{S}_{++}^{p}\}}{min}\sum_{k=1}^{K}(trace(\hat{\boldsymbol{\Sigma}}^{k}\boldsymbol{\Omega}^{k})-log\ det(\boldsymbol{\Omega}^{k}))+\lambda\sum_{k}||(\boldsymbol{\Omega}^{k})^{-}||_{1}+\rho\sum_{i\neq j}\sqrt{\sum_{k}(\Omega_{ij}^{k})^{2}}.
\end{equation}
In this section, we derive the rate of convergence in terms of the supremum norm for each population model. We derive the results in terms of the tail function $f$. The choices of the penalty parameters $\lambda$ and $\rho$ are stated in terms of $\gamma$. The value of $\gamma$ shall be chosen so that the theorem will hold for $\gamma>2$. We need to keep in mind that higher values of $\gamma$ shall lead to faster convergence, but will need more samples from each class. The following theorem limits the maximum element-wise difference between the predicted and true precision matrices when all the population models are considered together.

\begin{theorem}\label{Theorem 2.1}
	Consider $\boldsymbol{x}_i^{k}\in \mathbb{R}^{p}$ are independent and identically distributed samples from $\boldsymbol{x}^{k}$, such that $\mathbb{E}(\boldsymbol{x}^{k})=0$ and covariance matrix $\boldsymbol{\Sigma}_{0}^{k},k=1,2,\ldots,K$. Suppose $\boldsymbol{\Omega}_{0}^{k}=(\boldsymbol{\Sigma}_{0}^{k})^{-1}$ is our precision matrix for population model $k$, and the distributions satisfies Assumption \ref{Assumption 1} and Assumption \ref{Assumption 2}, for an $\alpha\in(0,1]$ and $\psi\in (0,1)$, and the tail condition (Definition \ref{Definition 1.2.4}) holds for all $k=1,2,\ldots,K$. Let $\hat{\boldsymbol{\Omega}}=\{\hat{\boldsymbol{\Omega}}^{k},k=1,2,\ldots,K\}$ be the unique solution to the optimization problem (\ref{as 11}), with $\lambda>0,\rho \leq \frac{\lambda}{\psi(2-\alpha)}(2-2\psi+\alpha\psi)$, and $\lambda+\rho=(8/\alpha)\delta$, where $\delta=\underset{k=1,2,\ldots,K}{max}\delta_{f}(n_k,p^{\gamma})$, for some $\gamma>2$. Then, if for each population $k$, the sample size $n_k$ satisfies
	\begin{equation}\label{as 12}
		n_k>n_{f}\bigg(min\left\{\frac{1}{\kappa_{\boldsymbol{\Gamma}_{0}^{k}}6(1+(8/\alpha))d}\ \underset{k_1=1,2,\ldots,K}{min}\left\{\frac{1}{max\{\kappa_{\boldsymbol{\Sigma}_{0}^{k_1}},\kappa^{3}_{\boldsymbol{\Sigma}_{0}^{k_1}}\kappa_{\boldsymbol{\Gamma}_{0}^{k_1}}\}}\right\},\underset{k=1,2,\ldots,K}{min}\left\{\frac{1}{v_{k0}}\right\}\right\},p^{\gamma}\bigg),
	\end{equation}
	 then, with probability greater than $1-\frac{K}{p^{\gamma-2}}\rightarrow 1$, for finite $K$,
	\begin{enumerate}
	\item The estimate of the precision matrix from population model $k$, $\hat{\boldsymbol{\Omega}}^{k}$ satisfies the supremum norm bound
	\begin{equation}\label{as 13}
	||\hat{\boldsymbol{\Omega}}^{k}-\boldsymbol{\Omega}_{0}^{k}||_{\infty}\leq 2 \kappa_{\boldsymbol{\Gamma}_{0}^{k}}(1+\frac{8}{\alpha})\delta,
	\end{equation}
	 for all $k=1,2,\ldots,K$.
	\item Denote $E(\boldsymbol{A})=\{(i,j):\boldsymbol{A}_{ij}\neq 0,i\neq j)$ as the edge set corresponding to $\boldsymbol{A}$. Then, $E(\hat{\boldsymbol{\Omega}}^{k})\subset E(\boldsymbol{\Omega}_{0}^{k})$, for all $k=1,2,\ldots,K$ including all edges $(i,j)$ such that $|\Omega_{0ij}^{k}|>2 \kappa_{\boldsymbol{\Gamma}_{0}^{k}}(1+\frac{8}{\alpha})\delta$.
	\end{enumerate}
\end{theorem}
\begin{proof}
Define the sub-gradient of $\sum_{k}||(\boldsymbol{\Omega}^{k})^{-}||_{1}$ and $\sum_{i\neq j}\sqrt{\sum_{k}(\Omega_{ij}^{k})^{2}}$ by $\boldsymbol{Z}^{k}$, and $\boldsymbol{M}^{k}$, respectively, where
\begin{equation}\label{as 14}
Z_{ij}^{k}=\begin{cases}
0 &, \text{if }i=j\\
sign(\Omega_{ij}^{k})&, \text{if } i\neq j\text{, and }|\Omega_{ij}^{k}|>0\\
[-1,1]&, \text{if } i\neq j\text{, and }|\Omega_{ij}^{k}|=0
\end{cases},
\end{equation}
 while 
\begin{equation}\label{as 15}
\{M_{ij}^{k},k=1,2,\ldots,K\}=\begin{cases}
\frac{\Omega_{ij}^{k}}{\sqrt{\sum_{l=1}^{K}(\Omega_{ij}^{l})^2}} &, \text{if }\{\Omega_{ij}^{k}\neq 0,k=1,2,\ldots,K\}\\
\boldsymbol{a}\text{,with }||\boldsymbol{a}||_2\leq 1&, \text{if } \{\Omega_{ij}^{k}= 0,k=1,2,\ldots,K\}
\end{cases}.
\end{equation}
\begin{lemma}\label{Lemma 2.2}
For any $\lambda>0,\rho>0$, and sample covariance matrix $\hat{\boldsymbol{\Sigma}}^{k}$ with strictly positive diagonal elements, $k=1,2,\ldots,K$, the problem
\begin{equation}\label{as 16}
\hat{\boldsymbol{\Sigma}}^k-(\hat{\boldsymbol{\Omega}}^k)^{-1}+\lambda \hat{\boldsymbol{Z}}^k+\rho \hat{\boldsymbol{M}}^k=0,
\end{equation}
 has a unique solution $\hat{\boldsymbol{\Omega}}=\{\hat{\boldsymbol{\Omega}}^{k},k=1,2,\ldots,K\}$, where $\hat{\boldsymbol{Z}}^k$, and $\hat{\boldsymbol{M}}^k$ are subgradients at $\hat{\boldsymbol{\Omega}}^k$.
\end{lemma}
\begin{proof}
Note that the problem in (\ref{as 11}) can also be written as 
$$
\underset{\{\boldsymbol{\Omega}^{k}\in\mathcal{S}_{++}^{p},\sum_{k}||(\boldsymbol{\Omega}^{k})^{-}||_{1}\leq a_{1}(\lambda),\sum_{i\neq j}\sqrt{\sum_{k}(\Omega_{ij}^{k})^{2}}\leq a_2(\rho)\}}{min}\sum_{k=1}^{K}(trace(\hat{\boldsymbol{\Sigma}}^{k}\boldsymbol{\Omega}^{k})-log\ det(\boldsymbol{\Omega}^{k})),
$$
 for some $a_1(\lambda)<+\infty, a_2(\rho)<+\infty$. These two constraints, together, bound the off-diagonal elements. However, using the Hadamard inequality, it can be shown that
$$
\sum_{k=1}^{K}(trace(\hat{\boldsymbol{\Sigma}}^{k}\boldsymbol{\Omega}^{k})-log\ det(\boldsymbol{\Omega}^{k}))\geq \sum_{k=1}^{K}(\sum_{i=1}^{p}(\hat{\boldsymbol{\Sigma}}_{ii}^{k}\boldsymbol{\Omega}^{k}_{ii}-log\ \boldsymbol{\Omega}^{k}_{ii})),
$$
 and the right hand side approaches to $\infty$, if $||(\boldsymbol{\Omega}^{k}_{11},\boldsymbol{\Omega}^{k}_{22},\ldots,\boldsymbol{\Omega}^{k}_{pp})||_{2}\rightarrow \infty$, for atleast one $k$, since $\hat{\boldsymbol{\Sigma}}^{k}$ has strictly positive diagonal elements, $k=1,2,\ldots,K$. Hence, the minimum is attained and will be unique, as the optimization function is strictly convex. Lastly, (\ref{as 16}) is the sub-differential equation, and any solution should belong to the sub-gradients of $\sum_{k}||(\boldsymbol{\Omega}^{k})^{-}||_{1}$, and $\sum_{i\neq j}\sqrt{\sum_{k}(\Omega_{ij}^{k})^{2}}$, evaluated at $\hat{\boldsymbol{\Omega}}=\{\hat{\boldsymbol{\Omega}}^{k},k=1,2,\ldots, K\}$.
\end{proof}
We shall prove the theorem using the primal-dual witness method, which involves a particular order of steps (see \cite{wainwright2006sharp}). We shall follow the steps to construct $(\tilde{\boldsymbol{\Omega}}^{k},\tilde{\boldsymbol{Z}}^{k},\tilde{\boldsymbol{M}}^{k};k=1,2,\ldots, K)$. The solutions are expected to satisfy the optimal conditions associated with the optimization problem (\ref{as 11}) with high probability. If the method succeeds, then $\tilde{\boldsymbol{\Omega}}^{k}$ shall be equal to the unique solution $\hat{\boldsymbol{\Omega}}^{k}$. The primal-dual witness solution shall be constructed as follows:
 
\begin{enumerate}
\item Obtain $\tilde{\boldsymbol{\Omega}}=\{\tilde{\boldsymbol{\Omega}}^{k},k=1,2,\ldots,K\}$ by solving the restricted problem 
\begin{equation}\label{as 17}
\underset{\{\boldsymbol{\Omega}^{k}\in\mathcal{S}_{++}^{p},\boldsymbol{\Omega}^{k}_{S^{c}}=0,k=1,2,\ldots,K\}}{min}\sum_{k=1}^{K}(trace(\hat{\boldsymbol{\Sigma}}^{k}\boldsymbol{\Omega}^{k})-log\ det(\boldsymbol{\Omega}^{k}))+\lambda\sum_{k}||(\boldsymbol{\Omega}^{k})^{-}||_{1}+\rho\sum_{i\neq j}\sqrt{\sum_{k}(\Omega_{ij}^{k})^{2}}.
\end{equation}
\item $\tilde{\boldsymbol{Z}}^{k}$ and $\tilde{\boldsymbol{M}}^{k}$ are the sub-gradients at $\tilde{\boldsymbol{\Omega}}^{k}$, for $(i,j)\in S$.
\item For all $k=1,2,\ldots,K$, and $(i,j) \in S^{c}$, we select
\begin{enumerate}
\item
\begin{equation}\label{as 18}
 \tilde{Z}_{ij}^{k}=\frac{\psi}{\lambda}\left\{[(\tilde{\boldsymbol{\Omega}}^{k})^{-1}]_{ij}-[\hat{\boldsymbol{\Sigma}}^{k}]_{ij}\right\}
\end{equation}
\item 
\begin{equation}\label{as 19}
 \tilde{M}_{ij}^{k}=\frac{1-\psi}{\rho}\left\{[(\tilde{\boldsymbol{\Omega}}^{k})^{-1}]_{ij}-[\hat{\boldsymbol{\Sigma}}^{k}]_{ij}\right\}
\end{equation}
\end{enumerate}

\item Verify the strict dual feasibility conditions
\begin{enumerate}
\item \begin{equation}\label{as 20}
|\tilde{Z}_{ij}^{k}|<1,
\end{equation}for all $k=1,2,\ldots,K$, and $(i,j)\in S^{c}$.
\item\begin{equation}\label{as 21}
\sqrt{\sum_{k=1}^{K}(\tilde{M}_{e}^{k})^{2}}<1,
\end{equation} for any $e \in S^{c}$
\end{enumerate}

\end{enumerate}
Steps 1 to 3 are obtained so that the triplet $(\tilde{\boldsymbol{\Omega}}^{k},\tilde{\boldsymbol{Z}}^{k},\tilde{\boldsymbol{M}}^{k})$ satisfies the optimality conditions (\ref{as 16}) of the problem. However, the solutions still not guarantee whether $(\tilde{\boldsymbol{Z}}^{k},\tilde{\boldsymbol{M}}^{k})$ truly are an elements of the sub-differential at $\tilde{\boldsymbol{\Omega}}^{k},k=1,2,\ldots,K$. Step 2 ensures whether $(\tilde{\boldsymbol{Z}}^{k},\tilde{\boldsymbol{M}}^{k})$ are truly sub-differentials at $S,k=1,2,\ldots,K$. Step 4 verifies whether $(\tilde{\boldsymbol{Z}}^{k},\tilde{\boldsymbol{M}}^{k})$ are members of sub-differentials at $S^{c},k=1,2,\ldots,K$. If the method is successful, then $\tilde{\boldsymbol{\Omega}}^{k}$, which is the solution to the restricted problem (\ref{as 17}), is equal to the solution to the unrestricted problem (\ref{as 11}). We show that the solution to the restricted problem indeed satisfies the strict dual feasibility condition, and hence the method succeeds with high probability. Let $\boldsymbol{W}^{k}=\hat{\boldsymbol{\Sigma}}^{k}-\boldsymbol{\Sigma}^{k}_{0}=\hat{\boldsymbol{\Sigma}}^{k}-(\boldsymbol{\Omega}^{k}_{0})^{-1}$, and $\boldsymbol{\Delta}^{k}=\tilde{\boldsymbol{\Omega}}^{k}-\boldsymbol{\Omega}_{0}^{k}$. Note that $\boldsymbol{\Delta}^{k}_{S^{c}}=0$, for all $k=1,2,\ldots,K$. Also define $R(\boldsymbol{\Delta}^{k})=\boldsymbol{R}^{k}=(\tilde{\boldsymbol{\Omega}}^{k})^{-1}-(\boldsymbol{\Omega}_{0}^{k})^{-1}+(\boldsymbol{\Omega}_{0}^{k})^{-1}\boldsymbol{\Delta}^{k}(\boldsymbol{\Omega}_{0}^{k})^{-1}$ as the difference between the gradient of $log\ det(\boldsymbol{A})$ at $\boldsymbol{A}=\tilde{\boldsymbol{\Omega}}^{k}$ and its first order Taylor expansion around $\boldsymbol{\Omega}_{0}^{k},k=1,2,\ldots,K$.

\begin{lemma}\label{Lemma 2.3}
If $\rho \leq \frac{\lambda}{\psi(2-\alpha)}(2-2\psi+\alpha\psi)$, and
\begin{equation}\label{as 22}
max\{||\boldsymbol{W}^{k}||_{\infty},||\boldsymbol{R}^{k}||_{\infty}\}\leq \frac{\alpha(\lambda+\rho)}{8},
\end{equation}
 for all $k=1,2,\ldots,K$, then
\begin{enumerate}
\item \begin{equation}\label{as 23}
|\tilde{Z}_{ij}^{k}|<1,
\end{equation}for all $k=1,2,\ldots,K$, and $(i,j)\in S^{c}$.
\item\begin{equation}\label{as 24}
\sqrt{\sum_{k=1}^{K}(\tilde{M}_{e}^{k})^{2}}<1,
\end{equation} for any $e \in S^{c}$
\end{enumerate}
\end{lemma}

\begin{proof}
The KKT conditions for (\ref{as 11}) are
\begin{equation}\label{as 25}
\hat{\boldsymbol{\Sigma}}^k-(\tilde{\boldsymbol{\Omega}}^k)^{-1}+\lambda \tilde{\boldsymbol{Z}}^k+\rho \tilde{\boldsymbol{M}}^k=0.
\end{equation}
  From the definitions of $\boldsymbol{W}^{k}$ and $\boldsymbol{R}^{k}$, we have (\ref{as 25}) as
 \begin{equation}\label{as 26}
 (\boldsymbol{\Omega}_{0}^{k})^{-1}\boldsymbol{\Delta}^{k}(\boldsymbol{\Omega}_{0}^{k})^{-1}+\boldsymbol{W}^{k}-\boldsymbol{R}^{k}+\lambda \tilde{\boldsymbol{Z}}^k+\rho \tilde{\boldsymbol{M}}^k=0.
 \end{equation}
From the properties of vectorizing a matrix, we know,
\begin{equation}\label{as 27}
vec((\boldsymbol{\Omega}_{0}^{k})^{-1}\boldsymbol{\Delta}^{k}(\boldsymbol{\Omega}_{0}^{k}
)^{-1})=((\boldsymbol{\Omega}_{0}^{k})^{-1}\otimes (\boldsymbol{\Omega}_{0}^{k})^{-1})\overline{\boldsymbol{\Delta}}^{k}=\boldsymbol{\Gamma}_{0}^{k}\overline{\boldsymbol{\Delta}}^{k},
\end{equation}
 where $\overline{\boldsymbol{A}}$ is the notation for vectorizing a matrix $\boldsymbol{A}$. Hence, vectorizing (\ref{as 26}), we have
\begin{eqnarray}
\label{as 28}
\boldsymbol{\Gamma}_{0SS}^{k}\overline{\boldsymbol{\Delta}}_{S}^{\ k}+\overline{\boldsymbol{W}}_{S}^{\ k}-\overline{\boldsymbol{R}}_{S}^{\ k}+\lambda \overline{\tilde{\boldsymbol{Z}}}_{S}^{\ k}+\rho \overline{\tilde{\boldsymbol{M}}}_{S}^{\ k}&=&0\\
\label{as 29}
\boldsymbol{\Gamma}_{0S^{c}S}^{k}\overline{\boldsymbol{\Delta}}_{S}^{\ k}+\overline{\boldsymbol{W}}_{S^{c}}^{\ k}-\overline{\boldsymbol{R}}_{S^{c}}^{\ k}+\lambda \overline{\tilde{\boldsymbol{Z}}}_{S^{c}}^{\ k}+\rho \overline{\tilde{\boldsymbol{M}}}_{S^{c}}^{\ k}&=&0.
\end{eqnarray}
There are many possible solutions of (\ref{as 29}) for $\overline{\tilde{\boldsymbol{Z}}}_{S^{c}}^{\ k}$, and $\overline{\tilde{\boldsymbol{M}}}_{S^{c}}^{\ k}$. However, we take solutions of the following type: 
\begin{eqnarray}
\label{as 30}
\overline{\tilde{\boldsymbol{Z}}}_{S^{c}}^{\ k}&=& \frac{\psi}{\lambda}[-\boldsymbol{\Gamma}_{0S^{c}S}^{k}\overline{\boldsymbol{\Delta}}_{S}^{\ k}-\overline{\boldsymbol{W}}_{S^{c}}^{\ k}+\overline{\boldsymbol{R}}_{S^{c}}^{\ k}]\\
\label{as 31}
\overline{\tilde{\boldsymbol{M}}}_{S^{c}}^{\ k}&=& \frac{1-\psi}{\rho}[-\boldsymbol{\Gamma}_{0S^{c}S}^{k}\overline{\boldsymbol{\Delta}}_{S}^{\ k}-\overline{\boldsymbol{W}}_{S^{c}}^{\ k}+\overline{\boldsymbol{R}}_{S^{c}}^{\ k}].
\end{eqnarray}
We show that the solutions of $\overline{\tilde{\boldsymbol{Z}}}_{S^{c}}^{\ k}$ and $\overline{\tilde{\boldsymbol{M}}}_{S^{c}}^{\ k}$ as in (\ref{as 30}) and (\ref{as 31}) satisfies the strict dual feasibility conditions (\ref{as 23}) and (\ref{as 24}). From (\ref{as 28}), we have
\begin{equation}\label{as 32}
\overline{\boldsymbol{\Delta}}_{S}^{\ k}=(\boldsymbol{\Gamma}_{0SS}^{k})^{-1}[-\overline{\boldsymbol{W}}_{S}^{\ k}+\overline{\boldsymbol{R}}_{S}^{\ k}-\lambda \overline{\tilde{\boldsymbol{Z}}}_{S}^{\ k}-\rho \overline{\tilde{\boldsymbol{M}}}_{S}^{\ k}].
\end{equation}
 Substituting (\ref{as 32}) in (\ref{as 30}) and (\ref{as 31}), we get
\begin{equation}\label{as 33}
\overline{\tilde{\boldsymbol{Z}}}_{S^{c}}^{\ k}= \frac{\psi}{\lambda}\left( \boldsymbol{\Gamma}_{0S^{c}S}^{k}(\boldsymbol{\Gamma}_{0SS}^{k})^{-1}(\overline{\boldsymbol{W}}_{S}^{\ k}-\overline{\boldsymbol{R}}_{S}^{\ k})+\boldsymbol{\Gamma}_{0S^{c}S}^{k}(\boldsymbol{\Gamma}_{0SS}^{k})^{-1}(\lambda \overline{\tilde{\boldsymbol{Z}}}_{S}^{\ k}+\rho \overline{\tilde{\boldsymbol{M}}}_{S}^{\ k})-(\overline{\boldsymbol{W}}_{S^{c}}^{\ k}-\overline{\boldsymbol{R}}_{S^{c}}^{\ k}) \right)
\end{equation}

\begin{eqnarray}
\notag
\overline{\tilde{\boldsymbol{M}}}_{S^{c}}^{\ k}&=& \frac{1-\psi}{\rho}\bigg( \boldsymbol{\Gamma}_{0S^{c}S}^{k}(\boldsymbol{\Gamma}_{0SS}^{k})^{-1}(\overline{\boldsymbol{W}}_{S}^{\ k}-\overline{\boldsymbol{R}}_{S}^{\ k})+\boldsymbol{\Gamma}_{0S^{c}S}^{k}(\boldsymbol{\Gamma}_{0SS}^{k})^{-1}(\lambda \overline{\tilde{\boldsymbol{Z}}}_{S}^{\ k}+\rho \overline{\tilde{\boldsymbol{M}}}_{S}^{\ k})\\
\label{as 34}
& &-(\overline{\boldsymbol{W}}_{S^{c}}^{\ k}-\overline{\boldsymbol{R}}_{S^{c}}^{\ k})\bigg).
\end{eqnarray}
Taking $l_{\infty}$ norm on both sides of (\ref{as 33}), we get
\begin{eqnarray}
\notag
||\overline{\tilde{\boldsymbol{Z}}}_{S^{c}}^{\ k}||_{\infty}&\leq & \frac{\psi}{\lambda}\bigg[\ |||\boldsymbol{\Gamma}_{0S^{c}S}^{k}(\boldsymbol{\Gamma}_{0SS}^{k})^{-1}|||_{\infty}(||\overline{\boldsymbol{W}}^{\ k}||_{\infty}+||\overline{\boldsymbol{R}}^{\ k}||_{\infty})+ |||\boldsymbol{\Gamma}_{0S^{c}S}^{k}(\boldsymbol{\Gamma}_{0SS}^{k})^{-1}|||_{\infty}(\lambda ||\overline{\tilde{\boldsymbol{Z}}}_{S}^{\ k}||_{\infty}\\
\label{as 35}
& &+\rho ||\overline{\tilde{\boldsymbol{M}}}_{S}^{\ k}||_{\infty})+(||\overline{\boldsymbol{W}}^{\ k}||_{\infty}+||\overline{\boldsymbol{R}}^{\ k}||_{\infty})\ \bigg].
\end{eqnarray}
From (\ref{as 22}), and Assumption (\ref{Assumption 1}), (\ref{as 35}) yields,
\begin{eqnarray}
\notag
||\overline{\tilde{\boldsymbol{Z}}}_{S^{c}}^{\ k}||_{\infty} &\leq & \frac{\psi}{\lambda}\left((1-\alpha)\frac{\alpha(\lambda+\rho)}{4}+(1-\alpha)(\lambda+\rho)+\frac{\alpha (\lambda+\rho)}{4} \right)\\
\notag
&=&  \frac{\psi}{\lambda}(\lambda+\rho)\left(\frac{\alpha}{2}-\frac{\alpha^2}{4}+1-\alpha\right)\\
\label{as 36}
&<&\psi \frac{\lambda+\rho}{\lambda}\left(\frac{2-\alpha}{2}\right).
\end{eqnarray}
From restriction on $\rho$, we have $||\overline{\tilde{\boldsymbol{Z}}}_{S^{c}}^{\ k}||_{\infty}<1$. Hence, we have $|\tilde{Z}_{ij}^{k}|<1$, for all $k=1,2,\ldots,K$, and $(i,j)\in S^{c}$. Elementwise, for any $e \in S^{c}$, (\ref{as 34}) yields,
\begin{eqnarray}
\notag
\overline{\tilde{\boldsymbol{M}}}_{e}^{\ k}&=& \frac{1-\psi}{\rho}\bigg( \boldsymbol{\Gamma}_{0eS}^{k}(\boldsymbol{\Gamma}_{0SS}^{k})^{-1}(\overline{\boldsymbol{W}}_{S}^{\ k}-\overline{\boldsymbol{R}}_{S}^{\ k})+\boldsymbol{\Gamma}_{0eS}^{k}(\boldsymbol{\Gamma}_{0SS}^{k})^{-1}(\lambda \overline{\tilde{\boldsymbol{Z}}}_{S}^{\ k}+\rho \overline{\tilde{\boldsymbol{M}}}_{S}^{\ k})\\
\label{as 37}
& &-(\overline{\boldsymbol{W}}_{e}^{\ k}-\overline{\boldsymbol{R}}_{e}^{\ k})\bigg)=\frac{1-\psi}{\rho}(a_{1e}^{k}+a_{2e}^{k}-a_{3e}^{k}).
\end{eqnarray}
Using $(a-b)^{2}\leq a^2+b^2$, we have
\begin{eqnarray*}
\bigg(\overline{\tilde{\boldsymbol{M}}}_{e}^{\ k}\bigg)^{2} &\leq & \left( \frac{1-\psi}{\rho} \right)^{2}((a_{1e}^{k}+a_{2e}^{k})^{2}+(a_{3e}^{k})^{2})\\
& = & \left( \frac{1-\psi}{\rho} \right)^{2}\left\{\left(\boldsymbol{\Gamma}_{0eS}^{k}(\boldsymbol{\Gamma}_{0SS}^{k})^{-1}(\overline{\boldsymbol{W}}_{S}^{\ k}-\overline{\boldsymbol{R}}_{S}^{\ k}+\lambda \overline{\tilde{\boldsymbol{Z}}}_{S}^{\ k}+\rho \overline{\tilde{\boldsymbol{M}}}_{S}^{\ k})\right)^{2}+(\overline{\boldsymbol{W}}_{e}^{\ k}-\overline{\boldsymbol{R}}_{e}^{\ k})^{2}\right\}.
\end{eqnarray*}
Summing over all $K$ models, we have
\begin{eqnarray*}
\sum_{k=1}^{K}\bigg(\overline{\tilde{\boldsymbol{M}}}_{e}^{\ k}\bigg)^{2}&\leq & \left( \frac{1-\psi}{\rho} \right)^{2}\sum_{k=1}^{K}\{(\boldsymbol{\Gamma}_{0eS}^{k}(\boldsymbol{\Gamma}_{0SS}^{k})^{-1}(\overline{\boldsymbol{W}}_{S}^{\ k}-\overline{\boldsymbol{R}}_{S}^{\ k}+\lambda \overline{\tilde{\boldsymbol{Z}}}_{S}^{\ k}+\rho \overline{\tilde{\boldsymbol{M}}}_{S}^{\ k}))^{2}+(\overline{\boldsymbol{W}}_{e}^{\ k}-\overline{\boldsymbol{R}}_{e}^{\ k})^{2}\}\\
& \leq & \left( \frac{1-\psi}{\rho} \right)^{2}\sum_{k=1}^{K} \bigg\{(||\overline{\boldsymbol{W}}^{\ k}||_{\infty}+||\overline{\boldsymbol{R}}^{\ k}||_{\infty}+\lambda+\rho )^{2}(\boldsymbol{\Gamma}_{0eS}^{k}(\boldsymbol{\Gamma}_{0SS}^{k})^{-1}\boldsymbol{1})^{2}+(||\overline{\boldsymbol{W}}^{\ k}||_{\infty}\\
& &+||\overline{\boldsymbol{R}}^{\ k}||_{\infty})^{2}\bigg\}\\
& \leq & \left( \frac{1-\psi}{\rho} \right)^{2}\left(((\lambda+\rho)(1+(\alpha/4)))^{2}\sum_{k=1}^{K}(\boldsymbol{\Gamma}_{0eS}^{k}(\boldsymbol{\Gamma}_{0SS}^{k})^{-1}\boldsymbol{1})^{2}+K(\lambda+\rho)^{2}(\alpha/4)^{2}\right).
\end{eqnarray*}
Since, $\sqrt{\sum_{k}z_{k}}\leq \sum_{k}\sqrt{z_{k}}$, we have
\begin{equation}\label{as 38}
\sqrt{\sum_{k=1}^{K}(\overline{\tilde{\boldsymbol{M}}}_{e}^{\ k})^{2}}\leq (1-\psi)\left(\frac{\lambda+\rho}{\rho}\right)\left((1+\frac{\alpha}{4})\sqrt{\sum_{k=1}^{K}(\boldsymbol{\Gamma}_{0eS}^{k}(\boldsymbol{\Gamma}_{0SS}^{k})^{-1}\boldsymbol{1})^{2}}+\frac{\alpha\sqrt{K}}{4}\right).
\end{equation}
From Assumption (\ref{Assumption 2}), we have $\sqrt{\sum_{k=1}^{K}(\tilde{M}_{e}^{k})^{2}}<1$.
\end{proof}

We now state some lemmas to control the sampling noise $\boldsymbol{W}^{k}$, $\boldsymbol{\Delta}^{k}$ and the remainder term $\boldsymbol{R}^{k}$.
\begin{lemma}\label{Lemma 2.4}
For any $\gamma>2$ and sampling size $n_k$, such that $\delta_{f}(n_k,p^{\gamma})\leq \underset{k=1,2,\ldots,K}{min}\frac{1}{v_{k0}}$, for all $k=1,2,\ldots,K$, we have
\begin{equation}\label{as 39}
\mathbb{P}(||\boldsymbol{W}^{k}||_{\infty}\geq \delta_{f}(n_k,p^{\gamma}))\leq \frac{1}{p^{\gamma-2}}\rightarrow 0,
\end{equation}
 for all $k=1,2,\ldots,K$.
\end{lemma}
\begin{proof}
From the definition of tail condition (Definition \ref{Definition 1.2.4}), we have
$$\mathbb{P}\bigg(\underset{i,j}{max}|W_{ij}^{k}|\geq \delta\bigg)\leq \frac{p^2}{f(n_k,\delta)},$$ for all $\delta \in (0,\underset{k=1,2,\ldots,K}{min}\frac{1}{v_{k0}}]$. Select $\delta=\delta_f(n_k,p^{\gamma})$. Then, we have the inequality provided $\delta_f(n_k,p^{\gamma})<\underset{k_1}{min}\frac{1}{v_{k_10}}, k=1,2,\ldots,K$.
\end{proof}
\begin{lemma}\label{Lemma 2.5}(Lemma 5 \citep{10.1214/11-EJS631})
Suppose
\begin{equation}\label{as 40}
||\boldsymbol{\Delta}^{k}||_{\infty}\leq \frac{1}{3d\kappa_{\boldsymbol{\Sigma}_0^{k}}}, k=1,2,\ldots,K.
\end{equation}
Then $\boldsymbol{J}^k=\sum_{i=0}^{\infty}(-1)^{i}((\boldsymbol{\Omega}^{k}_{0})^{-1}\boldsymbol{\Delta}^{k})^{i}$ satisfies
 \begin{equation}\label{as 41}
 |||(\boldsymbol{J}^k)^{T}|||_{\infty}\leq \frac{3}{2},
 \end{equation}
and $\boldsymbol{R}^k=(\boldsymbol{\Omega}^{k}_{0})^{-1}\boldsymbol{\Delta}^{k}(\boldsymbol{\Omega}^{k}_{0})^{-1}\boldsymbol{\Delta}^{k}\boldsymbol{J}^{k}(\boldsymbol{\Omega}^{k}_{0})^{-1}$ satisfies
  \begin{equation}\label{as 42}
  ||\boldsymbol{R}^k||_{\infty}\leq \frac{3}{2}d ||\boldsymbol{\Delta}^{k}||_{\infty}^{2}\kappa^{3}_{\boldsymbol{\Sigma}_{0}^{k}}.
  \end{equation}
\end{lemma}
\begin{lemma}\label{Lemma 2.6}
Suppose
\begin{equation}\label{as 43}
r_{k}=2\kappa_{\boldsymbol{\Gamma}_{0}^{k}}(||\boldsymbol{W}^{k}||_{\infty}+(\lambda+\rho))\leq \frac{1}{d}\ \underset{k_1}{min}\left\{ min\left\{\frac{1}{3\kappa_{\boldsymbol{\Sigma}_{0}^{k_1}}},\frac{1}{3\kappa^{3}_{\boldsymbol{\Sigma}_{0}^{k_1}}\kappa_{\boldsymbol{\Gamma}_{0}^{k_1}}}\right\}\right\},
\end{equation}
for all $k=1,2,\ldots,K$. Then
\begin{equation}\label{as 44}
||\boldsymbol{\Delta}^{k}||_{\infty}=||\tilde{\boldsymbol{\Omega}}^{k}-\boldsymbol{\Omega}_{0}^{k}||_{\infty}\leq r_k,
\end{equation}
for all $k=1,2,\ldots,K$.
\end{lemma}
\begin{proof}
Define the $l_{\infty}$ ball as $$\mathcal{B}(r_k)=\{\boldsymbol{B}_{S}:||\boldsymbol{B}_{S}||_{\infty}\leq r_k\},$$ with $r_k=2\kappa_{\boldsymbol{\Gamma}_{0}^{k}}(||\boldsymbol{W}^{k}||_{\infty}+(\lambda+\rho))$. Also define $$G(\boldsymbol{\Omega}_{S}^{k})=-[(\boldsymbol{\Omega}^{k})^{-1}]_{S}+\hat{\boldsymbol{\Sigma}_{S}^{k}}+\lambda \tilde{\boldsymbol{Z}}_{S}^k+\rho \tilde{\boldsymbol{M}}_{S}^k.$$
Note that $G(\boldsymbol{\Omega}_{S}^{k})$ is the gradient function of the restricted problem (\ref{as 17}). The gradient function will vanish at the optimum $\boldsymbol{\Omega}_{S}^{k}$, i.e., $G(\boldsymbol{\Omega}_{S}^{k})=0$. By the same arguement as Lemma $\ref{Lemma 2.2}$, the equation $$G(\boldsymbol{\Omega}_{S}^{k})=0,$$ has a unique solution, $\tilde{\boldsymbol{\Omega}}_{S}^{k},k=1,2,\ldots,K$. The idea is to define a continuous map $F$, such that any value from the convex set $\mathcal{B}(r_k)$ shall be mapped onto the set $\mathcal{B}(r_k)$ itself. Thus, we have $F(\mathcal{B}(r_k))\subset \mathcal{B}(r_k)$, and hence, from Brouwer's fixed point theorem and uniqueness of $G(\boldsymbol{\Omega}^{k}_{S})=0$, we conclude $||\boldsymbol{\Delta}^{k}||_{\infty}\leq r_k$.
Define the map $F:\mathbb{R}^{|S|}\rightarrow \mathbb{R}^{|S|}$ as $$F(\overline{\boldsymbol{\Delta}}^{\ k}_{S})=-(\boldsymbol{\Gamma}^{k}_{0S S})^{-1}\overline{G}(\boldsymbol{\Omega}^{k}_{0S}+\boldsymbol{\Delta}^{k}_{S})+\overline{\boldsymbol{\Delta}}^{k}_{S}.$$ By construction, $F(\overline{\boldsymbol{\Delta}}^{\ k}_{S})=\overline{\boldsymbol{\Delta}}^{\ k}_{S}$ iff $G(\boldsymbol{\Omega}^{k}_{0S}+\overline{\boldsymbol{\Delta}}^{\ k}_{S})=G(\tilde{\boldsymbol{\Omega}}^{k}_{S})=0$. Let $\boldsymbol{\Delta}^{k}\in \mathbb{R}^{p \times p}$ be a matrix such that $\boldsymbol{\Delta}^{k}_{S}\in \mathcal{B}(r_k)$, and $\boldsymbol{\Delta}^{k}_{S^{c}}=0$. Thus following the steps of the proof of Lemma 6 \citep{10.1214/11-EJS631}, we have $$F(\overline{\boldsymbol{\Delta}}^{\ k}_{S})=(\boldsymbol{\Gamma}^{k}_{0S S})^{-1}vec[((\boldsymbol{\Omega}_{0}^{k})^{-1}\boldsymbol{\Delta}^{k})^2\boldsymbol{J}^{k}(\boldsymbol{\Omega}_{0}^{k})^{-1}]_{S}-(\boldsymbol{\Gamma}^{k}_{0S S})^{-1}(\overline{\boldsymbol{W}}^{\ k}_{S}+\lambda \overline{\tilde{\boldsymbol{Z}}}_{S}^{\ k}+\rho \overline{\tilde{\boldsymbol{M}}}_{S}^{\ k})=T_1+T_2.$$ Now, $T_2$ is easy to control since $$||T_2||_{\infty}\leq \kappa_{\boldsymbol{\Gamma}^{k}_{0}}(||\boldsymbol{W}^{k}||_{\infty}+(\lambda+\rho))=\frac{r_k}{2}.$$
For $T_1$, we have $$||T_{1}||_{\infty}\leq \kappa_{\boldsymbol{\Gamma}^{k}_{0}} ||vec[((\boldsymbol{\Omega}_{0}^{k})^{-1}\boldsymbol{\Delta}^{k})^2\boldsymbol{J}^{k}(\boldsymbol{\Omega}_{0}^{k})^{-1}]_{S}||_{\infty}\leq \kappa_{\boldsymbol{\Gamma}^{k}_{0}} ||\boldsymbol{R}^{k}||_{\infty}.$$ 
Now, since $\boldsymbol{\Delta}^{k}$ is such that $||\boldsymbol{\Delta}^{k}||_{\infty}\leq r_{k}\leq \frac{1}{d}\ \underset{k_1}{min}\ \frac{1}{3\kappa_{\boldsymbol{\Sigma}_{0}^{k_1}}}\leq \frac{1}{3d\kappa_{\boldsymbol{\Sigma}_{0}^{k}}}$, we have $||\boldsymbol{R}^k||_{\infty}\leq \frac{3}{2}d ||\boldsymbol{\Delta}^{k}||_{\infty}^{2}\kappa^{3}_{\boldsymbol{\Sigma}_{0}^{k}}$. Also $r_{k}\leq \frac{1}{d}\ \underset{k_1}{min}\frac{1}{3\kappa_{\boldsymbol{\Sigma}_{0}^{k_{1}}}^{3}\kappa_{\boldsymbol{\Gamma}_{0}^{k_1}}}<\frac{1}{3d\kappa_{\boldsymbol{\Sigma}_{0}^{k}}^{3}\kappa_{\boldsymbol{\Gamma}_{0}^{k}}}$, we have
$$||T_1||_{\infty}\leq \frac{r_k}{2}.$$ Thus, we prove that $F(\overline{\boldsymbol{\Delta}}^{\ k}_{S})\in \mathcal{B}(r_k)$ if $\boldsymbol{\Delta}^{\ k}_{S} \in \mathcal{B}(r_k)$. Hence, $F(\mathcal{B}(r_k))\subset \mathcal{B}(r_k)$ and we have (\ref{as 44}).
\end{proof}
\begin{lemma}\label{Lemma 2.7}(Lemma 7 \citep{10.1214/11-EJS631}) Suppose (\ref{as 43}) holds for all $k=1,2,\ldots,K$, and $\omega_{min}^{k}$ is the minimum absolute value of the non-zero entries from $\boldsymbol{\Omega}^{k}_{0}$, and is lower bounded by 
\begin{equation}\label{as 45}
\omega_{min}^{k}\geq 4\kappa_{\boldsymbol{\Gamma}_{0}^{k}}(||\boldsymbol{W}^{k}||_{\infty}+(\lambda+\rho)).
\end{equation}
Then, $sign(\tilde{\boldsymbol{\Omega}}^{\ k}_{S})=sign(\boldsymbol{\Omega}^{\ k}_{0S})$, for all $k=1,2,\ldots,K$.
\end{lemma}
We now return to the proof of Theorem \ref{Theorem 2.1}. Define $\mathcal{A}_{k}=\{||\boldsymbol{W}^{k}||_{\infty}\leq \delta_{f}(n_k,p^{\gamma})\}$, for $k=1,2,\ldots,K$ and some $\gamma>2$. Then, if $\delta_{f}(n_k,p^{\gamma})\leq \underset{k}{min}(1/v_{k0})$, for all $k=1,2,\ldots,K$, we have $\mathbb{P}(\mathcal{A}_{k})\geq 1-\frac{1}{p^{\gamma-2}}\rightarrow 1$. Also, $\delta_{f}(n_k,p^{\gamma})\leq \underset{k}{min}(1/v_{k0})$, for all $k=1,2,\ldots,K$ implies $n_k>n_f(\underset{k}{min}(1/v_{k0}),p^{\gamma})$, for all $k$. Let $(\lambda+\rho)=\frac{8}{\alpha}\delta$, where $\delta=\underset{k}{max}\ \delta_{f}(n_k,p^{\gamma})$. Under event $\mathcal{A}_{k}$, we have $||\boldsymbol{W}^{k}||_{\infty}\leq \delta =\frac{\alpha}{8}(\lambda+\rho)$, for all $k=1,2,\ldots,K$. It remains to show that $||\boldsymbol{R}^{k}||_{\infty}\leq \frac{\alpha}{8}(\lambda+\rho)$.

Note that (\ref{as 12}) implies
\begin{equation}\label{as 46}
\delta<\frac{1}{6d(1+(8/ \alpha))\kappa_{\boldsymbol{\Gamma}_{0}^{k}}}\underset{k_1}{min}\left \{min\left \{\frac{1}{\kappa_{\boldsymbol{\Sigma}_{0}^{k_1}}},\frac{1}{\kappa^{3}_{\boldsymbol{\Sigma}_{0}^{k_1}}\kappa_{\boldsymbol{\Gamma}_{0}^{k_1}}}\right \}\right\}.
\end{equation}
Thus,
\begin{eqnarray}\label{as 47}
\notag
2\kappa_{\boldsymbol{\Gamma}_{0}^{k}}(||\boldsymbol{W}^{k}||_{\infty}+(\lambda+\rho))&\leq& 2\kappa_{\boldsymbol{\Gamma}_{0}^{k}}(1+(8/ \alpha))\delta\\
&\leq & \frac{1}{d} \underset{k_1}{min}\left \{min\left \{\frac{1}{3\kappa_{\boldsymbol{\Sigma}_{0}^{k_1}}},\frac{1}{3\kappa^{3}_{\boldsymbol{\Sigma}_{0}^{k_1}}\kappa_{\boldsymbol{\Gamma}_{0}^{k_1}}}\right \}\right\}.
\end{eqnarray}
 Hence, from Lemma \ref{Lemma 2.6}, we have 
 \begin{equation}\label{as 48}
 ||\boldsymbol{\Delta}^{k}||_{\infty}\leq r_k\leq 2\kappa_{\boldsymbol{\Gamma}_{0}^{k}}(1+(8/ \alpha))\delta,
 \end{equation}
 for all $k=1,2,\ldots,K$. Thus, from Lemma \ref{Lemma 2.5}, we have  $||\boldsymbol{R}^k||_{\infty}\leq \frac{3}{2}d ||\boldsymbol{\Delta}^{k}||_{\infty}^{2}\kappa^{3}_{\boldsymbol{\Sigma}_{0}^{k}}$. Substituting $||\boldsymbol{\Delta}^{k}||_{\infty}\leq r_k\leq \frac{1}{d}\underset{k_1}{min}\left \{\frac{1}{3\kappa^{3}_{\boldsymbol{\Sigma}_{0}^{k_1}}\kappa_{\boldsymbol{\Gamma}_{0}^{k_1}}}\right \}$, which implies $||\boldsymbol{\Delta}^{k}||_{\infty}\leq \frac{1}{3d\kappa^{3}_{\boldsymbol{\Sigma}_{0}^{k}}\kappa_{\boldsymbol{\Gamma}_{0}^{k}}}$, for all $k=1,2,\ldots,K$, we have $||\boldsymbol{R}^{k}||_{\infty}\leq \frac{\alpha}{8}(\lambda+\rho)$. Hence, the conditions for Lemma \ref{Lemma 2.2} are satisfied. Thus, the strict dual feasibility condition holds, allowing $\tilde{\boldsymbol{\Omega}}^{k}=\hat{\boldsymbol{\Omega}}^{k}$, for all $k=1,2,\ldots,K$. Hence, in $\{\mathcal{A}_{k},k=1,2,\ldots,K\}$, we have $\hat{\boldsymbol{\Omega}}^{k}$, satisfying (\ref{as 48}), for all $k=1,2,\ldots,K$. Also, we have $\tilde{\boldsymbol{\Omega}}^{k}_{S^{c}}=\hat{\boldsymbol{\Omega}}^{k}_{S^{c}}=0$, for all $k=1,2,\ldots,K$. However, the above statements hold if $\mathcal{A}_{k}$ holds for all $k=1,2,\ldots,K$. Now, we know $\mathbb{P}(\mathcal{A}_{k})\geq 1-\frac{1}{p^{\gamma-2}}$, for all $k=1,2,\ldots,K$. Hence, by Boole's inequality, we have $\mathbb{P}(\bigcap_{k=1}^{K}\mathcal{A}_{k})\geq 1-\frac{K}{p^{\gamma-2}}\rightarrow 1$, for finite $K$. Hence, we have the theorem.
\end{proof}
\begin{Remark}\label{Remark 1}
The penalty parameters $\lambda$ and $\rho$ are taken to be of the same order. In that case, $\rho=m\lambda$. Then, for the above theorem to hold, the conditions on $\lambda$, $\rho$ and the restriction of $\lambda+\rho=(8/\alpha)\delta$ can be further simplified as $\lambda>0,0<m \leq \frac{2-2\psi+\alpha\psi}{(2-\alpha)\psi}$, and $(1+m)\lambda=(8/\alpha)\delta$.
\end{Remark}

\begin{Remark}\label{Remark 2}
    If the samples satisfy the sub-Gaussian condition (Definition \ref{Definition 1.2.2}), $\delta_{f}(n_{k},p^{\gamma})$ satisfies (\ref{as 6}). In that case, from Theorem \ref{Theorem 2.1}, we have,
    \begin{equation}\label{as 49}
        \mathbb{P}(||\hat{\boldsymbol{\Omega}}^{k}-\boldsymbol{\Omega}_{0}^{k}||_{\infty}\leq 2 \kappa_{\boldsymbol{\Gamma}_{0}^{k}}(1+\frac{8}{\alpha})\delta)\geq 1-\frac{K}{p^{\gamma-2}},
    \end{equation}
    where
    \begin{equation}\label{as 50}
        \delta=\underset{k=1,2,\ldots,K}{max}\ 8\sqrt{2}(1+12K_{1}^{2})\underset{i}{max}\ \Sigma_{0ii}^{k}\sqrt{\frac{log\ 4p^{\gamma}}{n_k}},
    \end{equation}
    and the sample sizes $n_k$ should satisfy
    \begin{equation}\label{as 51}
        n_{k}>c_{1k}d^{2}(1+(8/\alpha))^{2}log(4p^{\gamma}),
    \end{equation}
    for all $k=1,2,\ldots,K$, where 
    \begin{equation*}
        c_{1k}= \bigg(48\sqrt{2}(1+12K_{1}^{2})\kappa_{\boldsymbol{\Gamma}_{0}^{k}}\ \underset{i}{max}\ \Sigma_{0ii}^{k}\underset{k_{1}=1,2,\ldots,K}{max}\bigg(max\bigg\{ \kappa_{\boldsymbol{\Sigma}_{0}^{k_1}},\kappa_{\boldsymbol{\Sigma}_{0}^{k_1}}^{3}\kappa_{\boldsymbol{\Gamma}_{0}^{k_1}}\bigg\}\bigg)\bigg)^{2}.
    \end{equation*}
    From Assumption \ref{Assumption 3}, with probability greater than $1-\frac{K}{p^{\gamma-2}}$, we have,
    \begin{equation}\label{as 52}
        ||\hat{\boldsymbol{\Omega}}^{k}-\boldsymbol{\Omega}_{0}^{k}||_{\infty}\leq 16\sqrt{2}\kappa_{\boldsymbol{\Gamma}_{0}^{k}}(1+(8/\alpha))(1+12K_{1}^{2})\underset{k_{1}=1,2,\ldots,K}{max}\ \bigg(\underset{i}{max}\ \Sigma_{0ii}^{k_1}\sqrt{\frac{log(4p^{\gamma})}{n_{k_1}}}\bigg),
    \end{equation}
    for all $k=1,2,\ldots,K$. Hence, when $\kappa_{\boldsymbol{\Sigma}_{0}^{k_1}},\kappa_{\boldsymbol{\Gamma}_{0}^{k}},\alpha$ remains constant as a function of $n,p,d$, for all $k=1,2,\ldots,K$, we have $||\hat{\boldsymbol{\Omega}}^{k}-\boldsymbol{\Omega}_{0}^{k}||_{\infty}=O\bigg(\sqrt{\frac{log\ p}{n}}\bigg)$ with high probability by taking each sample size $n_{k}=\psi(d^{2}log(4p^{\gamma}))$.
\end{Remark}

It is stated by part 2 of Theorem \ref{Theorem 2.1} that the edge set of the estimated precision matrices, $E(\hat{\boldsymbol{\Omega}}^{k})$ for each population $k$ is a subset of the actual precision matrices $E(\boldsymbol{\Omega}^{k}_{0})$. The edge sets $E(\boldsymbol{\Omega}^{k}_{0})$ are equal. Here $E(\hat{\boldsymbol{\Omega}}^{k})$ contains all the edges $(i,j)$, whenever $\Omega_{0ij}^{k}>2 \kappa_{\boldsymbol{\Gamma}_{0}^{k}}(1+\frac{8}{\alpha})\delta$. However, it can be proved that a finer adjustment of the sample size $n_{k}$ for all $k=1,2,\ldots, K$ will be sufficient for model selection consistency for all population models. Define the event for population $k=1,2,\ldots,K$ as
$$\mathcal{M}(\hat{\boldsymbol{\Omega}}^{k},\boldsymbol{\Omega}^{k}_{0})=\{ sign(\hat{\Omega}_{ij}^{k})=sign(\Omega_{0,ij}^{k})),(i,j)\in E(\boldsymbol{\Omega}^{k}_{0})\}.$$
Then we have the following theorem:
\begin{theorem}\label{Theorem 2.8}
    Under similar conditions of Theorem \ref{Theorem 2.1}, if for each population $k=1,2,\ldots,K$, the sample size $n_k$ satisfies
    \begin{eqnarray}\label{as 53}
    \notag
       n_k&>&n_{f}\bigg(min\bigg\{\frac{1}{\kappa_{\boldsymbol{\Gamma}_{0}^{k}}6d(1+(8/\alpha))}\ \underset{k_1=1,2,\ldots,K}{min}\bigg\{\frac{1}{max\{\kappa_{\boldsymbol{\Sigma}_{0}^{k_1}},\kappa^{3}_{\boldsymbol{\Sigma}_{0}^{k_1}}\kappa_{\boldsymbol{\Gamma}_{0}^{k_1}}\}}\bigg\}\\
       & &,\underset{k_1=1,2,\ldots,K}{min}\bigg\{ \frac{1}{v_{k_10}}\bigg\},\frac{1}{4(1+(8/\alpha))}\underset{k_1=1,2,\ldots,K}{min}\bigg\{ \frac{\omega_{min}^{k_1}}{\kappa_{\boldsymbol{\Gamma}_{0}^{k_1}}}\bigg\}\bigg\},p^{\gamma}\bigg)
    \end{eqnarray}
    , the estimators $\hat{\boldsymbol{\Omega}}=\{\hat{\boldsymbol{\Omega}}^{k},k=1,2,\ldots,K\}$ satisfies 
    \begin{equation}\label{as 54}
        \mathbb{P}\bigg(\mathcal{M}(\hat{\boldsymbol{\Omega}}^{k},\boldsymbol{\Omega}^{k}_{0}) ,\forall k=1,2,\ldots,K\bigg)\geq 1-\frac{K}{p^{\gamma-2}}\rightarrow 1,\text{for finite }K.
    \end{equation}

\end{theorem}

The proof follows the same line of proof in Theorem 2 of \cite{10.1214/11-EJS631}. Hence, we omit the proof here.

\begin{Remark}\label{Remark 2.1}
    If the samples satisfy the sub-Gaussian condition (Definition \ref{Definition 1.2.2}), from (\ref{as 53}), we have 
    $$n_{k}=\psi \bigg(\bigg(d^{2}+\bigg(\underset{k_1=1,2,\ldots,K}{max}\ \omega^{k_1}_{min}\bigg)^{-2}\bigg)log(4p^{\gamma})\bigg),$$
    for all $k=1,2,\ldots,K$. Thus, if $\underset{k_1=1,2,\ldots,K}{max}\ \omega^{k_1}_{min}=\psi\bigg(\sqrt{\frac{log\ p}{n}}\bigg)$ is satisfied, then sample sizes of $n_{k}=\psi(d^2log(4p^{\gamma}))$ is sufficient for ensuring model consistency with high probability. 
\end{Remark}

Theorem \ref{Theorem 2.1} lower bounds the number of samples from each class. The next theorem inverts the above theorem to give an upper bound on the degree of vertices for each graph when sample sizes $n_k$ are given.
\begin{theorem}\label{Theorem 2.9}
Suppose $\boldsymbol{x}_{i}^{k}\in \mathbb{R}^{p}$ are samples from population $k$, with mean $0$ and covariance matrix $\boldsymbol{\Sigma}_{0}^{k}$, $k=1,2,\ldots,K$. Suppose the distributions satisfy the irrepresentability condition (Assumption \ref{Assumption 1}) and between-group irrepresentability condition (Assumption \ref{Assumption 2}) with an $\alpha\in (0,1]$ and a $\psi\in (0,1)$. Let $\hat{\boldsymbol{\Omega}}=\{\hat{\boldsymbol{\Omega}}^{k},k=1,2,\ldots,K\}$ be the unique solutions to the optimizing problem (\ref{as 11}) with tuning parameter $\lambda+\rho=\frac{8}{\alpha}\delta$, where $\lambda>0,\rho \leq \frac{\lambda}{\psi(2-\alpha)}(2-2\psi+\alpha\psi)$, and $\delta=\underset{k}{max}\ \delta_k$, with each $\delta_{k}>0$. Suppose that
\begin{equation}\label{as 55}
d\leq \frac{1}{\kappa_{\boldsymbol{\Sigma}_{0}^{k_1}}}\ min\left \{1,\frac{1}{\kappa^{2}_{\boldsymbol{\Sigma}_{0}^{k_1}}\kappa_{\boldsymbol{\Gamma}_{0}^{k_1}}}\right \}\frac{1}{6(1+(8/\alpha))\underset{k}{max}\{\delta_k \kappa_{\boldsymbol{\Gamma}_{0}^{k}}\}},
\end{equation}
for all $k_1=1,2,\ldots,K$ is satisfied. Then, on the set $\{ ||\hat{\boldsymbol{\Sigma}}^{k}-\boldsymbol{\Sigma}_{0}^{k}||_{\infty}\leq \delta,k=1,2,\ldots,K\}$, it holds
\begin{enumerate}
\item $\hat{\boldsymbol{\Omega}}^{k}_{S^c}=\boldsymbol{\Omega}^{k}_{0S^c}$.
\item $||\hat{\boldsymbol{\Omega}}^{k}-\boldsymbol{\Omega}_{0}^{k}||_{\infty}\leq 2 \kappa_{\boldsymbol{\Gamma}_{0}^{k}}(1+\frac{8}{\alpha})\delta$, for all $k=1,2,\ldots,K$.
\end{enumerate}
\end{theorem}
\begin{Remark}\label{Remark 3}
 When $\rho=m\lambda$, then Theorem \ref{Theorem 2.9} holds if we simplify our restrictions on $\lambda$, $\rho$, and $\lambda+\rho=\frac{8}{\alpha}\delta$ to $\lambda>0,0<m \leq \frac{2-2\psi+\alpha\psi}{(2-\alpha)\psi}$, and $(1+m)\lambda=(8/\alpha)\delta$.
\end{Remark}
\section{\large Results related to debiased estimates of group graphical lasso}\label{Section 3}
The joint estimators are subjected to bias due to optimization, including the penalty terms. To do any sort of statistical inference involving the estimators, it is first necessary to de-bias the estimators. For that purpose, we need to invert the KKT conditions corresponding to the optimization problem (\ref{as 11}). The KKT conditions corresponding to (\ref{as 11}) are 
\begin{equation}\label{as 56}
\hat{\boldsymbol{\Sigma}}^{k}-(\hat{\boldsymbol{\Omega}}^{k})^{-1}+\lambda\hat{\boldsymbol{Z}}^{k}+\rho\hat{\boldsymbol{M}}^{k}=0,
\end{equation}

for each $k=1,2,\ldots,K$, where $\hat{\boldsymbol{Z}}^{k}$ and $\hat{\boldsymbol{M}}^{k}$ are the
sub-differential corresponding to $||(\hat{\boldsymbol{\Omega}}^{k})^{-}||_{1}$,
and $\sum_{i\neq j}\sqrt{\sum_{k}(\hat{\Omega_{ij}}^{k})^{2}}$. The sub-differentials can be obtained from (\ref{as 14}) and (\ref{as 15}).
From Lemma \ref{Lemma 2.2}, there exist an unique solution of $\hat{\boldsymbol{\Omega}}^{k}\in\mathcal{S}_{++}^{p}$
satisfying (9) for $\lambda>0,\rho>0$, and $\hat{\boldsymbol{\Sigma}}^{k}$
having strictly positive diagonal elements. We then invert the KKT
conditions (9), by multiplying $\hat{\boldsymbol{\Omega}}^{k}$ on
both sides, and we have 
\begin{eqnarray}\label{as 57}
\notag
\hat{\boldsymbol{\Omega}}^{k}\hat{\boldsymbol{\Sigma}}^{k}\hat{\boldsymbol{\Omega}}^{k}-\hat{\boldsymbol{\Omega}}^{k}+\hat{\boldsymbol{\Omega}}^{k}(\lambda\hat{\boldsymbol{Z}}^{k}+\rho\hat{\boldsymbol{M}}^{k})\hat{\boldsymbol{\Omega}}^{k} & = & 0\\
\implies\hat{\boldsymbol{\Omega}}^{k}+\hat{\boldsymbol{\Omega}}^{k}(\lambda\hat{\boldsymbol{Z}}^{k}+\rho\hat{\boldsymbol{M}}^{k})\hat{\boldsymbol{\Omega}}^{k}-\boldsymbol{\Omega}_{0}^{k} & = & \boldsymbol{\zeta}^{k}+\boldsymbol{re}\boldsymbol{m}^{k},
\end{eqnarray}

where $\boldsymbol{W}^{k}=\hat{\boldsymbol{\Sigma}}^{k}-\boldsymbol{\Sigma}_{0}^{k}$
,$\boldsymbol{\zeta}^{k}=-\boldsymbol{\Omega}_{0}^{k}\boldsymbol{W}^{k}\boldsymbol{\Omega}_{0}^{k}$,
and $\boldsymbol{re}\boldsymbol{m}^{k}=-(\hat{\boldsymbol{\Omega}}^{k}-\boldsymbol{\Omega}_{0}^{k})\boldsymbol{W}^{k}\boldsymbol{\Omega}_{0}^{k}-(\hat{\boldsymbol{\Omega}}^{k}\hat{\boldsymbol{\Sigma}}^{k}-I)(\hat{\boldsymbol{\Omega}}^{k}-\boldsymbol{\Omega}_{0}^{k})$. Here $\hat{\boldsymbol{\Omega}}^{k}(\lambda\hat{\boldsymbol{Z}}^{k}+\rho\hat{\boldsymbol{M}}^{k})\hat{\boldsymbol{\Omega}}^{k}$
eliminates the bias. Using the KKT conditions, we can represent $\lambda\hat{\boldsymbol{Z}}^{k}+\rho\hat{\boldsymbol{M}}^{k}=(\hat{\boldsymbol{\Omega}}^{k})^{-1}-\hat{\boldsymbol{\Sigma}}^{k}$.
Hence, we have the de-biased group graphical lasso estimators for
each model $k$ as $\hat{\boldsymbol{\Omega}}_{d}^{k}=\hat{\boldsymbol{\Omega}}^{k}+\hat{\boldsymbol{\Omega}}^{k}\left((\hat{\boldsymbol{\Omega}}^{k})^{-1}-\hat{\boldsymbol{\Sigma}}^{k}\right)\hat{\boldsymbol{\Omega}}^{k}=2\hat{\boldsymbol{\Omega}}^{k}-\hat{\boldsymbol{\Omega}}^{k}\hat{\boldsymbol{\Sigma}}^{k}\hat{\boldsymbol{\Omega}}^{k}$.

The following lemma shall give a bound for the remainder terms $\boldsymbol{rem}^{k},k=1,2,\ldots, K$ under the sub-Gaussian condition (Definition \ref{Definition 1.2.2}) and the sub-Gaussian vector condition (Definition \ref{Definition 1.2.3}).
\begin{lemma}\label{Lemma 3.1}
Consider samples from population $k$, $\boldsymbol{x}_{i}^{k}$ with mean $0$ and covariance matrix $\boldsymbol{\Sigma}_{0}^{k}$, $k=1,2,\ldots,K$. Suppose the distributions satisfy the irrepresentability condition and between-group irrepresentability condition (Assumption \ref{Assumption 1} and Assumption \ref{Assumption 2}) with an $\alpha \in (0,1]$ and a $\psi \in (0,1)$, and $\boldsymbol{\Omega}^{k}_{0}=(\boldsymbol{\Sigma}_{0}^{k})^{-1}$. Let $\hat{\boldsymbol{\Omega}}=\{\hat{\boldsymbol{\Omega}}^{k},k=1,2,\ldots,K\}$ be the unique solutions to the optimizing problem (\ref{as 11}) with tuning parameter $\lambda+\rho=\frac{8}{\alpha}\delta$, where $\lambda>0,\rho \leq \frac{\lambda}{\psi(2-\alpha)}(2-2\psi+\alpha\psi)$, and $\delta=\underset{k}{max}\ \delta_k$, with each $\delta_{k}>0$, for some $\gamma>2$, with
\begin{equation}\label{as 58}
\delta_{k}=\delta_f(n_k,p^{\gamma})=8(1+12K_{1}^2)\ \underset{i}{max}\ \Sigma_{0ii}^{k}\sqrt{2\frac{log\ 4p^{\gamma}}{n_k}},
\end{equation}
for all $k=1,2,\ldots,K$, where $K_1$ is the sub-Gaussian parameter mentioned below. Suppose that
\begin{equation}\label{as 59}
d\leq \frac{1}{\kappa_{\boldsymbol{\Sigma}_{0}^{k}}}\ min\left \{1,\frac{1}{\kappa^{2}_{\boldsymbol{\Sigma}_{0}^{k}}\kappa_{\boldsymbol{\Gamma}_{0}^{k}}}\right \}\frac{1}{6(1+(8/\alpha))\underset{k_1}{max}\{\delta_{k_1} \kappa_{\boldsymbol{\Gamma}_{0}^{k_1}}\}},
\end{equation}
for all $k=1,2,\ldots, K$ is satisfied, giving an upper bound to the maximum degree of a vertex for each population $k$. Then, under Assumption \ref{Assumption 3} and Assumption \ref{Assumption 4}
\begin{enumerate}
\item if the samples satisfy Definition \ref{Definition 1.2.2} with $K_1=O(1)$, then
\begin{equation}\label{as 60}
||\boldsymbol{rem}^{k}||_{\infty}=O_{p}\bigg(\frac{1}{\alpha^2}\kappa_{\boldsymbol{\Gamma}_{0}^{k}}\ max\left\{ d^{3/2}\frac{log\ p}{n},\frac{1}{\alpha}\kappa_{\boldsymbol{\Gamma}_{0}^{k}}d^2\bigg(\frac{log\ p}{n}\bigg)^{3/2}\right \}\bigg).
\end{equation}
\item if the samples satisfy Definition \ref{Definition 1.2.3} with $K_1=O(1)$, then
\begin{equation}\label{as 61}
||\boldsymbol{rem}^{k}||_{\infty}=O_{p}\bigg(\frac{1}{\alpha^2}\kappa^{2}_{\boldsymbol{\Gamma}_{0}^{k}}\kappa_{\boldsymbol{\Sigma}_{0}^{k}}d\frac{log\ p}{n}\bigg).
\end{equation}
\end{enumerate}
\end{lemma}
\begin{proof}
Inverting the KKT conditions, for each $k$, we have $$\hat{\boldsymbol{\Omega}}^{k}+\hat{\boldsymbol{\Omega}}^{k}(\lambda\hat{\boldsymbol{Z}}^{k}+\rho\hat{\boldsymbol{M}}^{k})\hat{\boldsymbol{\Omega}}^{k}-\boldsymbol{\Omega}_{0}^{k}  =  \boldsymbol{\zeta}^{k}+\boldsymbol{re}\boldsymbol{m}^{k},$$
where $\boldsymbol{W}^{k}=\hat{\boldsymbol{\Sigma}}^{k}-\boldsymbol{\Sigma}_{0}^{k}$
, $\boldsymbol{\zeta}^{k}=-\boldsymbol{\Omega}_{0}^{k}\boldsymbol{W}^{k}\boldsymbol{\Omega}_{0}^{k}$,
 and
\begin{equation}\label{as 62}
\boldsymbol{rem}^{k}=-(\hat{\boldsymbol{\Omega}}^{k}-\boldsymbol{\Omega}_{0}^{k})\boldsymbol{W}^{k}\boldsymbol{\Omega}_{0}^{k}-(\hat{\boldsymbol{\Omega}}^{k}\hat{\boldsymbol{\Sigma}}^{k}-I)(\hat{\boldsymbol{\Omega}}^{k}-\boldsymbol{\Omega}_{0}^{k}).
\end{equation}
Note that $$|||\hat{\boldsymbol{\Omega}}^{k}|||_{\infty}\leq |||\hat{\boldsymbol{\Omega}}^{k}-\boldsymbol{\Omega}_{0}^{k}|||_{\infty}+|||\boldsymbol{\Omega}_{0}^{k}|||_{1}\leq 2d \kappa_{\boldsymbol{\Gamma}_{0}^{k}}(1+\frac{8}{\alpha})\delta+\sqrt{d}\Lambda_{max}(\boldsymbol{\Omega}_{0}^{k}).$$
Define $\boldsymbol{\beta}^{k}=\hat{\boldsymbol{\Omega}}^{k}-\boldsymbol{\Omega}_{0}^{k}$. Also, from Theorem \ref{Theorem 2.9} part 1, we have $\hat{\boldsymbol{\Omega}}^{k}_{S^c}=\boldsymbol{\Omega}^{k}_{0S^c}$. Thus, $\hat{\boldsymbol{\Omega}}^{k}$ has at most $d$ non-zero entries per row. Hence,
\begin{equation*}
|||\boldsymbol{\beta}^{k}|||_{\infty}\leq d||\hat{\boldsymbol{\Omega}}^{k}-\boldsymbol{\Omega}_{0}^{k}||_{\infty}.
\end{equation*}
Again from Theorem \ref{Theorem 2.9} part 2, $|||\boldsymbol{\beta}^{k}|||_{\infty}\leq  2d\kappa_{\boldsymbol{\Gamma}_{0}^{k}}(1+(8/\alpha))\delta$, for all $k=1,2,\ldots,K$.Hence, from (\ref{as 62}), we have
\begin{eqnarray}\label{as 63}
\notag
||\boldsymbol{rem}^{k}||_{\infty} &=& ||\boldsymbol{\beta}^{k}\boldsymbol{W}^{k}\boldsymbol{\Omega}_{0}^{k}+\hat{\boldsymbol{\Omega}}^{k}(\hat{\boldsymbol{\Sigma}}^{k}-(\hat{\boldsymbol{\Omega}}^{k})^{-1})\boldsymbol{\beta}^{k}||_{\infty}\\
\notag
& \leq & |||\boldsymbol{\beta}^{k}|||_{\infty}||\boldsymbol{W}^{k}\boldsymbol{\Omega}_{0}^{k}||_{\infty}+|||\hat{\boldsymbol{\Omega}}^{k}|||_{\infty}(\lambda ||\hat{\boldsymbol{Z}}^{k}||_{\infty}+\rho ||\hat{\boldsymbol{M}}^{k}||_{\infty})||\boldsymbol{\beta}^{k}||_{\infty}\\
\notag
& \leq & 2d \kappa_{\boldsymbol{\Gamma}_{0}^{k}}(1+(8/\alpha))\delta ||\boldsymbol{W}^{k}\boldsymbol{\Omega}_{0}^{k}||_{\infty}+|||\hat{\boldsymbol{\Omega}}^{k}|||_{\infty}(\lambda+\rho)||\boldsymbol{\beta}^{k}||_{\infty}\\
\notag
& \leq &  2d \kappa_{\boldsymbol{\Gamma}_{0}^{k}}(1+(8/\alpha))\delta ||\boldsymbol{W}^{k}\boldsymbol{\Omega}_{0}^{k}||_{\infty}+(\lambda+\rho)||\boldsymbol{\beta}^{k}||_{\infty}(2d \kappa_{\boldsymbol{\Gamma}_{0}^{k}}(1+\frac{8}{\alpha})\delta\\
\notag
& &+\sqrt{d}\Lambda_{max}(\boldsymbol{\Omega}_{0}^{k}))\\
\notag
& \leq &  2d \kappa_{\boldsymbol{\Gamma}_{0}^{k}}(1+(8/\alpha))\delta ||\boldsymbol{W}^{k}\boldsymbol{\Omega}_{0}^{k}||_{\infty}+4 (\lambda+\rho) d^{2} \kappa^{2}_{\boldsymbol{\Gamma}_{0}^{k}}(1+(8/\alpha))^2\delta^2\\
\notag
& & +2(\lambda+\rho)d^{3/2}\kappa_{\boldsymbol{\Gamma}_{0}^{k}}(1+(8/\alpha))\delta \Lambda_{max}(\boldsymbol{\Omega}_{0}^{k})\\
\notag
& = & 2d \kappa_{\boldsymbol{\Gamma}_{0}^{k}}(1+(8/\alpha))\delta ||\boldsymbol{W}^{k}\boldsymbol{\Omega}_{0}^{k}||_{\infty}+\frac{8}{\alpha}  d^{2} \kappa^{2}_{\boldsymbol{\Gamma}_{0}^{k}}(1+(8/\alpha))^2\delta^3\\
\notag
& &+2 \frac{8}{\alpha} d^{3/2}\kappa_{\boldsymbol{\Gamma}_{0}^{k}}(1+(8/\alpha))\delta^2 \Lambda_{max}(\boldsymbol{\Omega}_{0}^{k})\\
&=& O\bigg(\frac{1}{\alpha}\kappa_{\boldsymbol{\Gamma}_{0}^{k}}\ max\left \{\delta d ||\boldsymbol{W}^{k}\boldsymbol{\Omega}_{0}^{k}||_{\infty},\frac{1}{\alpha^2} \kappa_{\boldsymbol{\Gamma}_{0}^{k}} d^{2} \delta^{3},\frac{1}{\alpha}d^{3/2}\delta^2 \Lambda_{max}(\boldsymbol{\Omega}_{0}^{k}) \right \} \bigg).
\end{eqnarray}
As a consequence of Lemma \ref{Lemma 1.1}, Assumption \ref{Assumption 4}, and boundedness of eigenvalues of $\boldsymbol{\Omega}_{0}^{k}$, if we assume $K_1=O(1)$, we have $||\hat{\boldsymbol{\Sigma}}^{k}-\boldsymbol{\Sigma}_{0}^{k}||_{\infty}=O_{p}(\sqrt{log\ p/n})$, for all $k=1,2,\ldots,K$. Thus, $\delta=O_{p}(\sqrt{log\ p/n})$. Then, under sub-Gaussian condition (Definition \ref{Definition 1.2.2}),
$$||\boldsymbol{W}^{k}\boldsymbol{\Omega}_{0}^{k}||_{\infty}\leq |||\boldsymbol{\Omega}_{0}^{k}|||_{\infty}||\boldsymbol{W}^{k}||_{\infty}\leq \sqrt{d}||\boldsymbol{\Omega}_{0}^{k}||_{2}||\boldsymbol{W}^{k}||_{\infty}=O_{p}\bigg(\sqrt{\frac{d\ log\ p}{n}}\bigg).$$ Then from (\ref{as 63}), we have 
\begin{equation}\label{as 64}
||\boldsymbol{rem}^{k}||_{\infty}=O_{p}\bigg(\frac{1}{\alpha^2}\kappa_{\boldsymbol{\Gamma}_{0}^{k}}\ max\left \{d^{3/2}\frac{log\ p}{n}, \frac{1}{\alpha}\kappa_{\boldsymbol{\Gamma}_{0}^{k}}d^2\bigg(\frac{log\ p}{n} \bigg)^{3/2} \right \} \bigg).
\end{equation}
Again, (\ref{as 62}) can also be bounded as
\begin{eqnarray}\label{as 65}
||\boldsymbol{rem}^{k}||_{\infty}& \leq & |||\hat{\boldsymbol{\Omega}}^{k}-\boldsymbol{\Omega}_{0}^{k}|||_{\infty}||\boldsymbol{W}^{k}\boldsymbol{\Omega}_{0}^{k}||_{\infty}+|||\hat{\boldsymbol{\Omega}}^{k}-\boldsymbol{\Omega}_{0}^{k}|||_{\infty}||\hat{\boldsymbol{\Omega}}^{k}\hat{\boldsymbol{\Sigma}}^{k}-I||_{\infty}.
\end{eqnarray} 
Now,
\begin{eqnarray*}
||\hat{\boldsymbol{\Omega}}^{k}\hat{\boldsymbol{\Sigma}}^{k}-I||_{\infty} & \leq & ||(\hat{\boldsymbol{\Sigma}}^{k}-\boldsymbol{\Sigma}_{0}^{k})(\hat{\boldsymbol{\Omega}}^{k}-\boldsymbol{\Omega}_{0}^{k})+\boldsymbol{\Sigma}_{0}^{k}(\hat{\boldsymbol{\Omega}}^{k}-\boldsymbol{\Omega}_{0}^{k})+(\hat{\boldsymbol{\Sigma}}^{k}-\boldsymbol{\Sigma}_{0}^{k})\boldsymbol{\Omega}_{0}^{k}||_{\infty}\\
& \leq & ||\hat{\boldsymbol{\Sigma}}^{k}-\boldsymbol{\Sigma}_{0}^{k}||_{\infty}|||\hat{\boldsymbol{\Omega}}^{k}-\boldsymbol{\Omega}_{0}^{k}|||_{\infty}+|||\boldsymbol{\Sigma}_{0}^{k}|||_{\infty}||\hat{\boldsymbol{\Omega}}^{k}-\boldsymbol{\Omega}_{0}^{k}||_{\infty}+||\boldsymbol{W}^{k}\boldsymbol{\Omega}_{0}^{k}||_{\infty}.
\end{eqnarray*}
Hence, $||\boldsymbol{rem}^{k}||_{\infty}\leq I+II+III+I$, where
\begin{eqnarray*}
I&=& |||\hat{\boldsymbol{\Omega}}^{k}-\boldsymbol{\Omega}_{0}^{k}|||_{\infty}||\boldsymbol{W}^{k}\boldsymbol{\Omega}_{0}^{k}||_{\infty}\\
II&=& ||\hat{\boldsymbol{\Sigma}}^{k}-\boldsymbol{\Sigma}_{0}^{k}||_{\infty}|||\hat{\boldsymbol{\Omega}}^{k}-\boldsymbol{\Omega}_{0}^{k}|||^{2}_{\infty}\\
III&=& |||\hat{\boldsymbol{\Omega}}^{k}-\boldsymbol{\Omega}_{0}^{k}|||_{\infty} |||\boldsymbol{\Sigma}_{0}^{k}|||_{\infty}||\hat{\boldsymbol{\Omega}}^{k}-\boldsymbol{\Omega}_{0}^{k}||_{\infty}.
\end{eqnarray*}
Conditioned on the event that $\{ ||\hat{\boldsymbol{\Sigma}}^{k}-\boldsymbol{\Sigma}_{0}^{k}||_{\infty}\leq \delta,k=1,2,\ldots,K\}$, we have,
\begin{equation*}
|||\boldsymbol{\Omega}_{0}^{k}|||_{\infty}=\underset{i}{max}\sum_{j=1}^{p}|\Omega_{ij}^{k}|\leq \sqrt{d}\ \underset{i}{max}||\boldsymbol{\Omega}_{0i}^{k}||_{2}\leq \sqrt{d}\Lambda_{max}(\boldsymbol{\Omega}_{0}^{k}).
\end{equation*}
 From the fact that $|||\boldsymbol{\beta}^{k}|||_{\infty}\leq  2d\kappa_{\boldsymbol{\Gamma}_{0}^{k}}(1+(8/\alpha))\delta$, for all $k=1,2,\ldots,K$, we have,
\begin{eqnarray*}
I&\leq & 2d \kappa_{\boldsymbol{\Gamma}_{0}^{k}}(1+(8/\alpha))\delta ||\boldsymbol{W}^{k}\boldsymbol{\Omega}_{0}^{k}||_{\infty}\\
II& \leq & 4d^{2} \kappa^{2}_{\boldsymbol{\Gamma}_{0}^{k}}(1+(8/\alpha))^2\delta^3\\
III& \leq & 4d \kappa^{2}_{\boldsymbol{\Gamma}_{0}^{k}} \kappa_{\boldsymbol{\Sigma}_{0}^{k}}(1+(8/\alpha))^2\delta^2.
\end{eqnarray*}
Then
\begin{eqnarray}\label{as 66}
\notag
||\boldsymbol{rem}^{k}||_{\infty}& \leq & 4\ max\{ 2d \kappa_{\boldsymbol{\Gamma}_{0}^{k}}(1+(8/\alpha))\delta ||\boldsymbol{W}^{k}\boldsymbol{\Omega}_{0}^{k}||_{\infty}, 4d^{2} \kappa^{2}_{\boldsymbol{\Gamma}_{0}^{k}}(1+(8/\alpha))^2\\
\notag
& &\delta^3, 4d \kappa^{2}_{\boldsymbol{\Gamma}_{0}^{k}} \kappa_{\boldsymbol{\Sigma}_{0}^{k}}(1+(8/\alpha))^2\delta^2\}\\
&=& O\bigg(\frac{1}{\alpha}\kappa_{\boldsymbol{\Gamma}_{0}^{k}}\delta\ max\left\{d||\boldsymbol{W}^{k}\boldsymbol{\Omega}_{0}^{k}||_{\infty},\frac{1}{\alpha}d^{2} \kappa_{\boldsymbol{\Gamma}_{0}^{k}}\delta^2,\frac{1}{\alpha}d \kappa_{\boldsymbol{\Gamma}_{0}^{k}} \kappa_{\boldsymbol{\Sigma}_{0}^{k}}\delta\right\}\bigg),
\end{eqnarray}
for all $k=1,2,\ldots,K$. Under the sub-gaussianity vector condition, we have $||\boldsymbol{W}^{k}\boldsymbol{\Omega}_{0}^{k}||_{\infty}=\underset{i,j}{max}|\boldsymbol{\Omega}_{0i}^{k}\boldsymbol{W}^{k}\boldsymbol{e}_j|=O_{p}(\sqrt{log\ p/n})$. Thus, we have
\begin{eqnarray}\label{as 67}
\notag
||\boldsymbol{rem}^{k}||_{\infty}&=& O_{p}\bigg(\frac{1}{\alpha}\kappa_{\boldsymbol{\Gamma}_{0}^{k}}\ max\left\{d\frac{log\ p}{n},\frac{1}{\alpha}\kappa_{\boldsymbol{\Gamma}_{0}^{k}}d^{2}\bigg(\frac{log\ p}{n}\bigg)^{3/2},\frac{1}{\alpha}\kappa_{\boldsymbol{\Gamma}_{0}^{k}}\kappa_{\boldsymbol{\Sigma}_{0}^{k}}d\frac{log\ p}{n}\right\}\bigg)\\
&=&O_{p}\bigg(\frac{1}{\alpha^2}\kappa^{2}_{\boldsymbol{\Gamma}_{0}^{k}}\kappa_{\boldsymbol{\Sigma}_{0}^{k}}d\frac{log\ p}{n}\bigg),
\end{eqnarray}
where $n=\underset{k}{min}\ n_k$.
\end{proof}
\begin{Remark}\label{Remark 4}
Similar to Remark \ref{Remark 1} and Remark \ref{Remark 3}, when $\rho=m\lambda$, then Lemma \ref{Lemma 3.1} will hold under the simplified conditions of $\lambda$, $\rho$, and $\lambda+\rho=\frac{8}{\alpha}\delta$ to $\lambda>0,0<m \leq \frac{2-2\psi+\alpha\psi}{(2-\alpha)\psi}$, and $(1+m)\lambda=(8/\alpha)\delta$.
\end{Remark}
\begin{Remark}\label{Remark 5}
If $\frac{1}{\alpha}=O(1),\kappa_{\boldsymbol{\Sigma}_{0}^{k}}=O(1)$, and $\kappa_{\boldsymbol{\Gamma}_{0}^{k}}=O(1)$, for all $k=1,2,\ldots,K$, (\ref{as 59}) reduces to
\begin{eqnarray*}
d&\leq& \frac{1}{\underset{k_1}{max}\sqrt{\frac{log\ p}{n_k}}}=\underset{k_1}{min}\sqrt{\frac{n_k}{log\ p}}=\sqrt{\frac{n}{log\ p}},\text{ for all }k=1,2,\ldots,K\\
\implies d &\leq& \sqrt{\frac{n}{log\ p}}.
\end{eqnarray*}
Thus, $||\boldsymbol{rem}^{k}||_{\infty}=O_p(d^{3/2}\frac{log\ p}{\sqrt{n}})$, under Definition \ref{Definition 1.2.2}, and $||\boldsymbol{rem}^{k}||_{\infty}=O_p(d\frac{log\ p}{\sqrt{n}})$, under Definition \ref{Definition 1.2.3}.
\end{Remark}
Till now, we have shown that we can control the remainder term $||\boldsymbol{rem}^{k}||_{\infty}$, for all $k=1,2,\ldots, K$ upto order $d^{3/2}\frac{log\ p}{\sqrt{n}}$, when the observations follow sub-Gaussian condition, and upto order $d\frac{log\ p}{\sqrt{n}}$, when the observations satisfy the sub-Gaussian vector condition. But, we are interested in the asymptotic distribution of individual elements of the de-biased estimator for each $k=1,2,\ldots, K$. Note that this is possible if we obtain the asymptotic distribution of individual elements of $\boldsymbol{\zeta}^{k}$,$k=1,2,\ldots, K$. The next theorem solves this issue. The proof can be easily done by proving the Lindeberg condition. Hence, we skip the proof and only state the theorem.
\begin{theorem}\label{Theorem 3.2}
Suppose $\boldsymbol{x}_{i}^{k}\in \mathbb{R}^p$ are random samples from population $k$, with mean $0$ and covariance $\boldsymbol{\Sigma}_{0}^{k}$, $k=1,2,\ldots,K$. Let $\boldsymbol{\Omega}_{0}^{k}=(\boldsymbol{\Sigma}_{0}^{k})^{-1}$ and suppose the distributions satisfy the irrepresentability condition (Assumption \ref{Assumption 1}) with an $\alpha\in (0,1]$, between-group irrepresentability condition (Assumption \ref{Assumption 2}) with a $\psi\in (0,1)$. Let
 $$(\sigma_{ij}^{k})^2=var((\boldsymbol{\Omega}_{0i}^{k})^T \boldsymbol{x}_{1}^{k} (\boldsymbol{\Omega}_{0j}^{k})^T \boldsymbol{x}_{1}^{k})=\Omega_{0ii}^{k}\Omega_{0jj}^{k}+(\Omega_{0ij}^{k})^{2}$$,
 and suppose $1/\sigma_{ij}^{k}=O(1)$. Suppose $\hat{\boldsymbol{\Omega}}=\{\hat{\boldsymbol{\Omega}}^{k},k=1,2,\ldots,K\}$ be the unique solutions to the optimizing problem (\ref{as 11}) with tuning parameter $\lambda\asymp \rho \asymp \sqrt{log\ p /n}$. Also suppose 
 $$d^{3/2}=o\bigg(\frac{\sqrt{n}}{C_{1}log\ p}\bigg)$$ 
 under Definition \ref{Definition 1.2.2}, where 
 $$C_1=\underset{k}{max}\ max\left\{\frac{\kappa_{\boldsymbol{\Gamma}_{0}^{k}}}{\alpha^2},\frac{\kappa^{2}_{\boldsymbol{\Gamma}_{0}^{k}}}{\alpha^{9/8}}n^{-1/4}_{k}(log\ p)^{1/8},\frac{d\ max\ \{\kappa_{\boldsymbol{\Sigma}_{0}^{k}}\kappa_{\boldsymbol{\Gamma}_{0}^{k_1}},\kappa^{3}_{\boldsymbol{\Sigma}_{0}^{k}}\kappa^{2}_{\boldsymbol{\Gamma}_{0}^{k_1}}\}}{\alpha^{3/2}}(n_k log\ p)^{-1/4}\right\},$$
 while 
$$d=o\bigg(\frac{\sqrt{n}}{C_{2}log\ p}\bigg)$$
under Definition \ref{Definition 1.2.3}, where 
$$C_2=\frac{1}{\alpha}\underset{k}{max}\ \kappa_{\boldsymbol{\Gamma}_{0}^{k}}\ max\ \left\{\frac{1}{\alpha}\kappa_{\boldsymbol{\Gamma}_{0}^{k}}\kappa_{\boldsymbol{\Sigma}_{0}^{k}},\kappa_{\boldsymbol{\Sigma}_{0}^{k}}(log\ p)^{-1/2},\kappa_{\boldsymbol{\Gamma}_{0}^{k}}\kappa^{3}_{\boldsymbol{\Sigma}_{0}^{k}}(log\ p)^{-1/2}\right\}.$$
Then,under Assumption \ref{Assumption 3} and Assumption \ref{Assumption 4}, the element-wise asymptotic distribution of the de-biased estimator $\hat{\boldsymbol{\Omega}}_{d}^{k}$ satisfies
\begin{equation}\label{as 68}
\frac{\sqrt{n_{k}}}{\sigma_{ij}^{k}}(\hat{\Omega}_{d,ij}^{k}-\Omega_{0,ij}^{k})\stackrel{D}{\rightarrow}N(0,1),
\end{equation}
for all $i,j$, and for all $k=1,2,\ldots,K$. 
\end{theorem}
\begin{Remark}\label{Remark 6}
Note that the variance $(\sigma_{ij}^{k})^2$ is unknown and needs to be estimated. A consistent estimate shall be $(\hat{\sigma}_{ij}^{k})^{2}=\hat{\Omega}_{ii}^{k}\hat{\Omega}_{jj}^{k}+(\hat{\Omega}_{ij}^{k})^{2}$, where $\hat{\Omega}_{ij}^{k}$ is the $(i,j)^{th}$ element of the group graphical lasso estimator for $k=1,2,\ldots,K$. Hence, under the assumptions of Theorem \ref{Theorem 3.2}, from Slutsky's theorem, we have
\begin{equation}\label{as 69}
\frac{\sqrt{n_{k}}}{\hat{\sigma}_{ij}^{k}}(\hat{\Omega}_{d,ij}^{k}-\Omega_{0,ij}^{k})\stackrel{D}{\rightarrow}N(0,1),
\end{equation}
for all $i,j$, and for all $k=1,2,\ldots,K$.
\end{Remark}
We have derived the asymptotic distribution of the elements of the de-biased group graphical lasso estimator for multiple populations. Oftentimes, one would like to test whether there is any difference in magnitude among the non-zero entries of the matrices across the population models. Under the assumption of similar sparsity structure across the models, suppose we want to test $H_{0ij}: Y_{ij}=\sum_{k=1}^{K}a_k \Omega_{0,ij}^{k}=0$, $(i,j)\in S$, where $a_k$ are constants. Then, the test statistic can be summarized in the following lemma:
\begin{lemma}\label{Lemma 3.3}
Suppose $\boldsymbol{x}_{i}^{k}\in \mathbb{R}^p$ are random samples from population $k$, with mean $0$ and covariance $\boldsymbol{\Sigma}_{0}^{k}$, $k=1,2,\ldots,K$. Let $\boldsymbol{\Omega}_{0}^{k}=(\boldsymbol{\Sigma}_{0}^{k})^{-1}$ and suppose the distributions satisfy the irrepresentability condition (Assumption \ref{Assumption 1}) with an $\alpha\in (0,1]$, between-group irrepresentability condition (Assumption \ref{Assumption 2}) with an $\psi\in (0,1)$. Also, suppose $\hat{\boldsymbol{\Omega}}=\{\hat{\boldsymbol{\Omega}}^{k},k=1,2,\ldots,K\}$ be the unique solutions to the optimizing problem (\ref{as 11}) with tuning parameter $\lambda\asymp \rho \asymp \sqrt{log\ p /n}$. Also suppose $\hat{\boldsymbol{\Omega}}=\{\hat{\boldsymbol{\Omega}}^{k},k=1,2,\ldots,K\}$ be the unique solutions to the optimizing problem (\ref{as 11}) with tuning parameter $\lambda\asymp \rho \asymp \sqrt{log\ p /n}$. Also suppose $$d^{3/2}=o\bigg(\frac{\sqrt{n}}{C_{1}log\ p}\bigg)$$ under Definition \ref{Definition 1.2.2}, where $$C_1=\underset{k}{max}\ max\left\{\frac{\kappa_{\boldsymbol{\Gamma}_{0}^{k}}}{\alpha^2},\frac{\kappa^{2}_{\boldsymbol{\Gamma}_{0}^{k}}}{\alpha^{9/8}}n^{-1/4}_{k}(log\ p)^{1/8},\frac{d\ max\ \{\kappa_{\boldsymbol{\Sigma}_{0}^{k}}\kappa_{\boldsymbol{\Gamma}_{0}^{k_1}},\kappa^{3}_{\boldsymbol{\Sigma}_{0}^{k}}\kappa^{2}_{\boldsymbol{\Gamma}_{0}^{k_1}}\}}{\alpha^{3/2}}(n_k log\ p)^{-1/4}\right\},$$
while 
$$d=o\bigg(\frac{\sqrt{n}}{C_{2}log\ p}\bigg)$$
under Definition \ref{Definition 1.2.3}, where $$C_2=\frac{1}{\alpha}\underset{k}{max}\ \kappa_{\boldsymbol{\Gamma}_{0}^{k}}\ max\ \left\{\frac{1}{\alpha}\kappa_{\boldsymbol{\Gamma}_{0}^{k}}\kappa_{\boldsymbol{\Sigma}_{0}^{k}},\kappa_{\boldsymbol{\Sigma}_{0}^{k}}(log\ p)^{-1/2},\kappa_{\boldsymbol{\Gamma}_{0}^{k}}\kappa^{3}_{\boldsymbol{\Sigma}_{0}^{k}}(log\ p)^{-1/2}\right\}.$$
Then, if $\hat{Y}_{ij}=\sum_{k=1}^{K}a_{k}\hat{\Omega}_{d,ij}^{k}$, under Assumption \ref{Assumption 3} and Assumption \ref{Assumption 4}, we have
\begin{equation}\label{as 70}
T_{ij}=\frac{\hat{Y}_{ij}-Y_{ij}}{\sqrt{\sum_{k=1}^{K}\frac{a_{k}^{2}(\hat{\sigma}_{ij}^{k})^{2}}{n_k}}}\stackrel{D}{\rightarrow}N(0,1),
\end{equation}
for all $(i,j)\in S$.
\end{lemma}
Thus, to test $H_{0ij}: Y_{ij}=\sum_{k=1}^{K}a_k \Omega_{0,ij}^{k}=0$, for some $(i,j)\in S$, we use the test statistic $T_{ij}=\sqrt{n}\hat{Y}_{ij}$, and we reject $H_{0ij}$, at $\alpha$ level of significance if
\begin{equation}\label{as 71}
\bigg|\frac{\hat{Y}_{ij}}{\sqrt{\sum_{k=1}^{K}\frac{a_{k}^{2}(\hat{\sigma}_{ij}^{k})^{2}}{n_k}}}\bigg|>\tau_{\alpha/2},
\end{equation}
where $\tau_{\alpha/2}$ is the $(1-\alpha/2)$ upper quantile of
a standard normal distribution.

\section{\large Statistical Inference without assuming irrepresentability conditions, between-group irrepresentability conditions, and similar sparsity pattern}\label{Section 4}

In this section, we shall obtain a result for the convergence of the estimated precision matrices to the actual precision matrices in $l_1$-operator norm without Assumption \ref{Assumption 1} and Assumption \ref{Assumption 2}. Here, we shall not assume the sparsity pattern to be similar for all the graphs. We shall consider the observations satisfying the sub-Gaussian vector conditions (Definition \ref{Definition 1.2.3}) satisfying Assumption \ref{Assumption 3}. Also, the sample sizes should satisfy Assumption \ref{Assumption 4}. The following theorem bounds the sum of the $l_1$ operator norm of the difference of the estimated precision matrices and actual precision matrices for each population for $K=2$.

\begin{theorem}\label{Theorem 4.1}
    Suppose Assumption \ref{Assumption 3} and Assumption \ref{Assumption 4} holds , and the tuning parameter $\lambda$
and $\rho$ satisfies the following inequality 
\begin{equation}\label{as 72}
    2(\rho+\delta)\leq\lambda\leq\frac{c}{8L}\text{, and }\frac{8\lambda^{2}(s_{1}+s_{2})}{c}+\frac{4p\delta^{2}}{c}\leq\frac{\delta}{2L}
\end{equation}

. Then, we have 
\begin{equation}\label{as 73}
    c\sum_{k}||\hat{\boldsymbol{\Omega}}^{k}-\boldsymbol{\Omega}_{0}^{k}||_{F}^{2}+\lambda\sum_{k}||(\hat{\boldsymbol{\Omega}}^{k}-\boldsymbol{\Omega}_{0}^{k})^{-}||_{1}\leq\frac{8\lambda^{2}(s_{1}+s_{2})}{c}+\frac{4p\delta^{2}}{c}
\end{equation}
, and 
\begin{equation}\label{as 74}
    \sum_{k}|||\hat{\boldsymbol{\Omega}}^{k}-\boldsymbol{\Omega}_{0}^{k}|||_{1}\leq\frac{4\lambda(8s_{1}+8s_{2}+p)}{c}
\end{equation}
 on the set $\{max_{k}||\hat{\boldsymbol{\Sigma}}^{k}-\boldsymbol{\Sigma}_{0}^{k}||_{\infty}\leq\delta\}$,
where $c=1/(8L^{2})$.

\end{theorem}
\begin{proof}
    The proof follows the same line of arguments from Theorem 1 of \cite{zhang2024application}. We shall only prove the inequality $$\sum_{i\neq j}\sqrt{\sum_{k=1}^{2}(\Omega_{0ij}^{k})^{2}}-\sum_{i\neq j}\sqrt{\sum_{k=1}^{2}(\tilde{\Omega}_{ij}^{k})^{2}}\leq||(\tilde{\boldsymbol{\Omega}}^{1})^{-}-(\boldsymbol{\Omega}_{0}^{1})^{-}||_{1}+||(\tilde{\boldsymbol{\Omega}}^{2})^{-}-(\boldsymbol{\Omega}_{0}^{2})^{-}||_{1}$$. For this, note that \begin{eqnarray*}
\sum_{i\neq j}\left\{ \sqrt{\sum_{k=1}^{2}(\Omega_{0ij}^{k})^{2}}-\sqrt{\sum_{k=1}^{2}(\tilde{\Omega}_{ij}^{k})^{2}}\right\}  & = & \sum_{i\neq j}\left\{ \sqrt{\sum_{k=1}^{2}(\Omega_{0ij}^{k}-\tilde{\Omega}_{ij}^{k}+\tilde{\Omega}_{ij}^{k})^{2}}-\sqrt{\sum_{k=1}^{2}(\tilde{\Omega}_{ij}^{k})^{2}}\right\} \\
 & \leq & \sum_{i\neq j}\left\{ \sqrt{\sum_{k=1}^{2}(\Omega_{0ij}^{k}-\tilde{\Omega}_{ij}^{k})^{2}}+\sqrt{\sum_{k=1}^{2}(\tilde{\Omega}_{ij}^{k})^{2}}-\sqrt{\sum_{k=1}^{2}(\tilde{\Omega}_{ij}^{k})^{2}}\right\} \text{, from}\\
 &  & \text{ Triangle inequality }\sqrt{\sum_{i}(a_{i}+b_{i})^{2}}\leq\sqrt{\sum_{i}a_{i}^{2}}+\sqrt{\sum_{i}b_{i}^{2}}\\
 & = & \sum_{i\neq j}\sqrt{\sum_{k=1}^{2}(\Omega_{0ij}^{k}-\tilde{\Omega}_{ij}^{k})^{2}}=\sum_{i\neq j}\sqrt{\sum_{k=1}^{2}(\Delta_{ij}^{k})^{2}}
\end{eqnarray*}

Now, we know $\sqrt{\sum_{i}a_{i}^{2}}\leq\sum_{i}|a_{i}|$. Then,
$\sum_{i\neq j}\sqrt{\sum_{k=1}^{2}(\Delta_{ij}^{k})^{2}}\leq\sum_{i\neq j}\sum_{k=1}^{2}|\Delta_{ij}^{k}|=\sum_{k=1}^{2}||\boldsymbol{\Delta}^{k-}||_{1}$.
Hence, we have the inequality.
\end{proof}

We extend this theorem to general $K>2$. The theorem should be as follows:
\begin{theorem}\label{Theorem 4.2}
    For any general $K>2$, under the assumptions of Assumption \ref{Assumption 3}, Assumption \ref{Assumption 4} and
Definition \ref{Definition 1.2.3}, along with 
\begin{equation}\label{as 75}
2(\rho+\delta)\leq\lambda\leq\frac{c}{8L}\text{, and }\frac{8\lambda^{2}\sum_{k}s_{k}}{c}+\frac{2Kp\delta^{2}}{c}\leq\frac{\delta}{2L}
\end{equation}
, we have 
\begin{equation}\label{as 76}
c\sum_{k}||\hat{\boldsymbol{\Omega}}^{k}-\boldsymbol{\Omega}_{0}^{k}||_{F}^{2}+\lambda\sum_{k}||(\hat{\boldsymbol{\Omega}}^{k}-\boldsymbol{\Omega}_{0}^{k})^{-}||_{1}\leq\frac{8\lambda^{2}\sum_{k}s_{k}}{c}+\frac{2Kp\delta^{2}}{c}
\end{equation}
, and 
\begin{equation}\label{as 77}
\sum_{k}|||\hat{\boldsymbol{\Omega}}^{k}-\boldsymbol{\Omega}_{0}^{k}|||_{1}\leq\frac{2K\lambda(8\sum_{k}s_{k}+\frac{KP}{2})}{c}
\end{equation}
\end{theorem}
The proof of Theorem \ref{Theorem 4.2} shall be more like the
proof of Theorem \ref{Theorem 4.1}. Hence, we omit the proof.
\begin{Remark}\label{Remark 8}
From inequality (\ref{as 74}) or (\ref{as 77}), we see that we should select $\lambda$ so that $\lambda p\rightarrow 0$. However, for sub-Gaussian model, $\lambda$ is mostly selected as $\lambda\asymp \sqrt{log\ p/n}$. Hence, if we do not assume the irrepresentability condition, we can only confirm consistency for a moderately high-dimensional regime, where the dimension increases but at a much slower rate than the sample size.
\end{Remark}
We have the de-biased group graphical lasso estimators for each model $k$ which are as $\hat{\boldsymbol{\Omega}}_{d}^{k}=\hat{\boldsymbol{\Omega}}^{k}+\hat{\boldsymbol{\Omega}}^{k}\left((\hat{\boldsymbol{\Omega}}^{k})^{-1}-\hat{\boldsymbol{\Sigma}}^{k}\right)\hat{\boldsymbol{\Omega}}^{k}=2\hat{\boldsymbol{\Omega}}^{k}-\hat{\boldsymbol{\Omega}}^{k}\hat{\boldsymbol{\Sigma}}^{k}\hat{\boldsymbol{\Omega}}^{k}$. 
Thus, the test statistic for the testing problem \begin{equation}\label{as 78}
H_{0ij}:\Omega_{0,ij}^{1}=\Omega_{0,ij}^{2}\text{ vs }H_{1ij}:\Omega_{0,ij}^{1}\neq\Omega_{0,ij}^{2}
\end{equation}
is 
\begin{equation}\label{as 79}
    \hat{Y}_{ij}=\left(\hat{\boldsymbol{\Omega}}_{d}^{1}-\hat{\boldsymbol{\Omega}}_{d}^{2}\right)_{ij}=\left(2\hat{\boldsymbol{\Omega}}^{1}-\hat{\boldsymbol{\Omega}}^{1}\hat{\boldsymbol{\Sigma}}^{1}\hat{\boldsymbol{\Omega}}^{1}-(2\hat{\boldsymbol{\Omega}}^{2}-\hat{\boldsymbol{\Omega}}^{2}\hat{\boldsymbol{\Sigma}}^{2}\hat{\boldsymbol{\Omega}}^{2})\right)_{ij}
\end{equation}
, using the de-biased estimators. As the sparsity pattern of the precision matrices is not the same, we do not have $S_k=S$, for all $k=1,2,\ldots, K$. Since $s_k$ and $d_k$ can be different between the populations, let $s=max\{s_1,...,s_K\}$, and $d=max\{d_1,...,d_K\}$. For $K=2$, we have $s=max\{s_1,s_2\}$, and $d=max\{d_1,d_2\}$. Then, we have the following theorem:

\begin{theorem}\label{Theorem 4.3}
    Under Assumption \ref{Assumption 3}, Assumption \ref{Assumption 4}, Definition \ref{Definition 1.2.3}, and $\lambda\asymp\rho\asymp\sqrt{log\ p/n}$
and $(p+s)\sqrt{d}=o(\sqrt{n}/log\ p)$, 
\begin{equation}\label{as 80}
    \left(\hat{\boldsymbol{\Omega}}_{d}^{1}-\hat{\boldsymbol{\Omega}}_{d}^{2}\right)-\left(\boldsymbol{\Omega}_{0}^{1}-\boldsymbol{\Omega}_{0}^{2}\right)=\boldsymbol{\zeta}^{1}-\boldsymbol{\zeta}^{2}+\boldsymbol{rem},
\end{equation}
where $||\boldsymbol{rem}||_{\infty}=||\boldsymbol{rem}^{1}-\boldsymbol{rem}^{2}||_{\infty}=o_{p}(1/\sqrt{n})$ and $\boldsymbol{\zeta}^k$ and $\boldsymbol{rem}^{k}$ are same as in (\ref{as 57}).
Moreover, element-wise, we have $T_{ij}=(\hat{Y}_{ij}-Y_{ij})/\sigma_{ij}\stackrel{D}{\rightarrow} N(0,1)$,
where $Y_{ij}=(\Omega_{0,ij}^{1}-\Omega_{0,ij}^{2})$, and $\sigma_{ij}^{2}=(\left(\sigma_{ij}^{1}\right)^{2}/n_1)+(\left(\sigma_{ij}^{2}\right)^{2}/n_2)$.
Here, $\left(\sigma_{ij}^{k}\right)^{2}=var\left((\boldsymbol{\Omega}_{0i}^{k})^{T}\boldsymbol{x}_{1}^{k}(\boldsymbol{x}_{1}^{k})^{T}\boldsymbol{\Omega}_{0j}^{k}\right)=\Omega_{0,ii}^{k}\Omega_{0,jj}^{k}+(\Omega_{0,ij}^{k})^{2},k=1,2$.
\end{theorem}
\begin{proof}
  Note that, 
  \begin{equation*}
      ||\boldsymbol{rem}||_{\infty}\leq\sum_{k=1}^{2}||(\hat{\boldsymbol{\Omega}}^{k}-\boldsymbol{\Omega}_{0}^{k})\boldsymbol{W}^{k}\boldsymbol{\Omega}_{0}^{k}||_{\infty}+\sum_{k=1}^{2}||(\hat{\boldsymbol{\Omega}}^{k}\hat{\boldsymbol{\Sigma}}^{k}-I)(\hat{\boldsymbol{\Omega}}^{k}-\boldsymbol{\Omega}_{0}^{k})||_{\infty}.
  \end{equation*}
Multiplying the KKT conditions by $\hat{\boldsymbol{\Omega}}^{k}$
on both sides, we have $I-\hat{\boldsymbol{\Sigma}}^{k}\hat{\boldsymbol{\Omega}}^{k}=(\lambda\hat{\boldsymbol{Z}}^{k}+\rho\hat{\boldsymbol{M}}^{k})\hat{\boldsymbol{\Omega}}^{k}$.
Thus, 
\begin{equation*}
    ||\boldsymbol{rem}||_{\infty}\leq\sum_{k=1}^{2}|||(\hat{\boldsymbol{\Omega}}^{k}-\boldsymbol{\Omega}_{0}^{k})|||_{1}||\boldsymbol{W}^{k}||_{\infty}|||\boldsymbol{\Omega}_{0}^{k}|||_{1}+\sum_{k=1}^{2}|||(\hat{\boldsymbol{\Omega}}^{k}-\boldsymbol{\Omega}_{0}^{k})|||_{1}||\lambda\hat{\boldsymbol{Z}}^{k}+\rho\hat{\boldsymbol{M}}^{k}||_{\infty}|||\hat{\boldsymbol{\Omega}}^{k}|||_{1}.
\end{equation*}
Now, $||\lambda\hat{\boldsymbol{Z}}^{k}+\rho\hat{\boldsymbol{M}}^{k}||_{\infty}\leq||\lambda\hat{\boldsymbol{Z}}^{k}||_{\infty}+\rho||\hat{\boldsymbol{M}}^{k}||_{\infty}\leq(\lambda+\rho)$.
Also, $|||(\hat{\boldsymbol{\Omega}}^{k}-\boldsymbol{\Omega}_{0}^{k})|||_{1}\leq b(p+s)\lambda$,
from Theorem \ref{Theorem 4.1}, where $b$ depends on $L$. From Cauchy Schwarz inequality,
it can be shown that $|||\boldsymbol{\Omega}_{0}^{k}|||_{1}\leq\sqrt{d+1}\Lambda_{max}(\boldsymbol{\Omega}_{0}^{k})$.
Similarly, the bound of $|||\hat{\boldsymbol{\Omega}}^{k}|||_{1}$
can be obtained by 
\begin{equation*}
    |||\hat{\boldsymbol{\Omega}}^{k}|||_{1}\leq|||\hat{\boldsymbol{\Omega}}^{k}-\boldsymbol{\Omega}_{0}^{k}|||_{1}+|||\boldsymbol{\Omega}_{0}^{k}|||_{1}\sqrt{d+1}\Lambda_{max}(\boldsymbol{\Omega}_{0}^{k}),
\end{equation*}
due to the rate of $\lambda$. Also, the sub-Gaussian random vector
implies $||\boldsymbol{W}^{k}||_{\infty}=O_{P}\left(\sqrt{\frac{log\ p}{n}}\right)$.
Thus, we have $||\boldsymbol{rem}||_{\infty}=o_{p}(1/\sqrt{n})$.
Using Lindeberg CLT and bounded fourth moments, we have the asymptotic
normality.
\end{proof}
A consistent estimator of $\sigma_{ij}^{2}$ is $\hat{\sigma}_{ij}^{2}=\left(\bigg(\hat{\Omega}_{ii}^{1}\hat{\Omega}_{jj}^{1}+\left(\hat{\Omega}_{ij}^{1}\right)^{2}\bigg)\bigg/n_1\right)+\left(\bigg(\hat{\Omega}_{ii}^{2}\hat{\Omega}_{jj}^{2}+\left(\hat{\Omega}_{ij}^{2}\right)^{2}\bigg)\bigg/n_2\right)$.
Thus we reject $H_{0ij}$, at $\alpha$ level of significance if 
\begin{equation}\label{as 81}
    \bigg|\frac{\hat{Y}_{ij}}{\hat{\sigma}_{ij}}\bigg|>\tau_{\alpha/2}   
\end{equation}
, where $\tau_{\alpha/2}$ is the $(1-\alpha/2)$ upper quantile of
a standard normal distribution.
We extend this theorem to more than $K>2$ classes. For $K>2$ classes, suppose we want to test $H_0:\sum_{k=1}^{K}a_k\Omega_{0ij}^{k}=0$, where $a_k$ are constants. The theorem shall be as follows:
\begin{theorem}\label{Theorem 4.4}
    Under similar assumptions of Theorem \ref{Theorem 4.3}, we have 
    \begin{equation}\label{as 82}
        \hat{\boldsymbol{Y}}_{lc}-\boldsymbol{Y}_{lc}=\sum_{k=1}^{K}\boldsymbol{\zeta^k}+\boldsymbol{rem}_{lc}
    \end{equation}, where $\boldsymbol{rem}_{lc}=\sum_{k=1}^{K}a_k\boldsymbol{rem}^k$, $\hat{\boldsymbol{Y}}_{lc}=\sum_{k=1}^{K}a_k\hat{\boldsymbol{\Omega}}_{d}^k$, and $\boldsymbol{Y}_{lc}=\sum_{k=1}^{K}a_k\boldsymbol{\Omega}_{0}^{k}$. We have $||\boldsymbol{rem^k}||_{\infty}=o_p(1/\sqrt{n})$. Hence, we have $||\boldsymbol{rem}_{lc}||_{\infty}=o_p(1/\sqrt{n})$. Thus, elementwise, we have the asymptotic normality \begin{equation}\label{as 83}
        T_{ij}=\frac{(\hat{\boldsymbol{Y}}_{lc}-\boldsymbol{Y}_{lc})_{ij}}{\sigma_{ij}}\stackrel{D}{\rightarrow} N(0,1),
    \end{equation}
    where $\sigma_{ij}^2=\sum_{k=1}^{K}a_{k}^{2}(\sigma_{ij}^{k})^2/n_k=\sum_{k=1}^{K}a_{k}^{2}(\Omega_{0,ii}^{k}\Omega_{0,jj}^{k}+(\Omega_{0,ij}^{k})^{2})/n_k$.
\end{theorem}
A consistent estimate of $\sigma_{ij}^{2}$ is $\hat{\sigma}_{ij}^{2}=\sum_{k=1}^{K}a_{k}^{2}\left(\hat{\Omega}_{ii}^{k}\hat{\Omega}_{jj}^{k}+\left(\hat{\Omega}_{ij}^{k}\right)^{2}\right)\bigg/n_k$. Thus we reject $H_{0ij}$, at $\alpha$ level of significance if 
\begin{equation}\label{as 84}
    \bigg|\frac{(\hat{\boldsymbol{Y}}_{lc})_{ij}}{\hat{\sigma}_{ij}}\bigg|>\tau_{\alpha/2}.   
\end{equation}
\begin{Remark}\label{Remark 9}
Without assuming irrepresentability condition, we can only obtain the asymptotic distribution up to the regime where the dimension grows at a much slower rate than the sample size. This is the sacrifice we make without Assumption \ref{Assumption 1} and \ref{Assumption 2}.
\end{Remark}

\section{\large Empirical results}\label{Section 5}
\subsection{Simulation study}\label{Subsection 5.1}
We perform the illustration of our results, such as model consistency, sup-norm difference, and asymptotic normality of the test statistics, empirically. Before that, we will discuss how we are selecting the penalty parameters. Our theory suggests that the penalty parameters should be $\lambda=C_{1}\sqrt{\frac{log\ p}{n}}$, and $\rho=C_{2}\sqrt{\frac{log\ p}{n}}$. We use data-driven methods such as the extended-BIC for selecting the penalty parameters. The e-BIC \cite{foygel2010extended} for multiple models can be defined as 
\begin{eqnarray*}
e-BIC_{\gamma}(C_{1},C_{2}) & = & -2\sum_{k=1}^{K}\left[log\ det\left(\hat{\boldsymbol{\Omega}}^{k}(C_{1},C_{2})\right)-trace\left(\hat{\boldsymbol{\Sigma}}^{k}\hat{\boldsymbol{\Omega}}^{k}(C_{1},C_{2})\right)\right]\\
 &  & +\sum_{k=1}^{K}|E_{k}(C_{1},C_{2})|log\ n+4\gamma log\ p\sum_{k=1}^{K}|E_{k}(C_{1},C_{2})|.
\end{eqnarray*}

We select $(C_{1},C_{2})$ using grid selection and select that value of $(C_{1},C_{2})$ such that $e-BIC$ is minimum at $\gamma=0.5$. 

For empirical illustrations of the results, we select chain graphs and star graphs. These graphs satisfy the irrepresentability conditions that were required for theoretical proofs of the results. The precision matrix of the chain graphs should have a tridiag($\rho^{k},1,\rho^{k})$ structure for population $k,k=1,2$, where $\rho^{1}=0.2$, and $\rho^{2}=0.35$ for populations 1 and 2, respectively. For the star graphs, the hub node has the maximum degree $d_{max}$. We select one of the coordinates randomly, say $X_{u}$, and select $d_{max}$ coordinates randomly, leaving out $X_{u}$. This $X_{u}$ shall serve as a hub node. The randomly selected coordinates shall serve as the nodes sharing an edge only with the hub node. The rest of the coordinates are completely disconnected from each other. Let $b$ be the set containing the index of $d_{max}$ coordinates selected randomly. We develop the precision matrix for population 1 as 
\begin{eqnarray*}
\Omega_{ij}^{1} & = & \begin{cases}
2 & ,i=j\\
0.3 & ,i=u,j\in b\\
0.3 & ,i\in b,j=u\\
0 & ,o.w.
\end{cases}
\end{eqnarray*}

, and the precision matrix for population 2 as

\begin{eqnarray*}
\Omega_{ij}^{2} & = & \begin{cases}
2.5 & ,i=j\\
0.45 & ,i=u,j\in b.\\
0.45 & ,i\in b,j=u\\
0 & ,o.w.
\end{cases}
\end{eqnarray*}

We take dimensions $p=50,75,100,150$, and generate $n$ samples from each class, where $n$ is between $200$ and $700$. The samples drawn are from multivariate gaussian distribution with mean $0$ and variance $\left(\boldsymbol{\Omega}^{k}\right)^{-1},k=1,2$. We repeat the process for $B=100$ times, and report our empirical results. The solutions of the group graphical lasso optimization problem are obtained using the $\texttt{JGL}$ function of \cite{danaher2014joint}. 
\begin{center}
\begin{figure}[H]
\centering{}\caption{Star shaped graphs}
\includegraphics[scale=0.26]{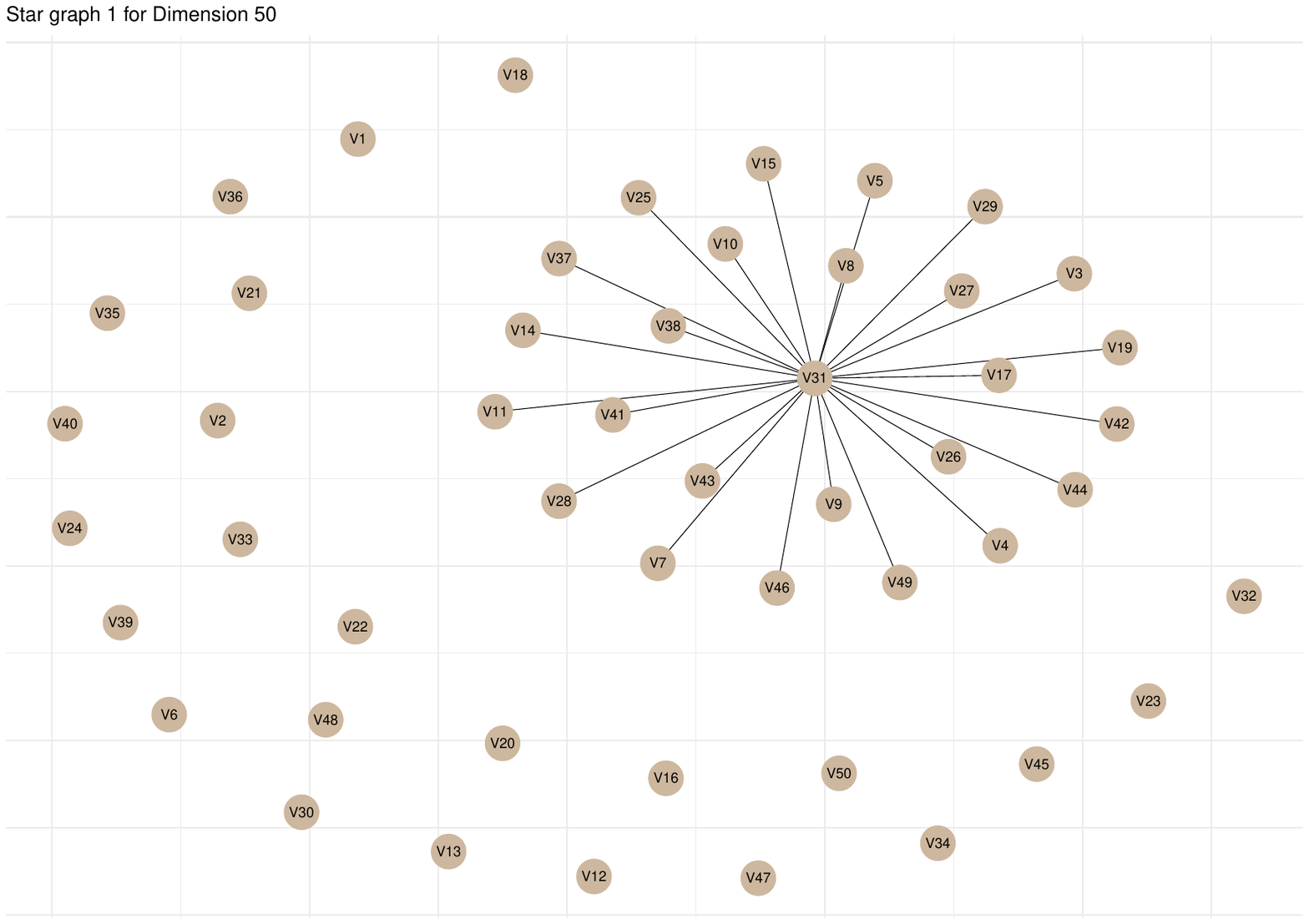}\includegraphics[scale=0.26]{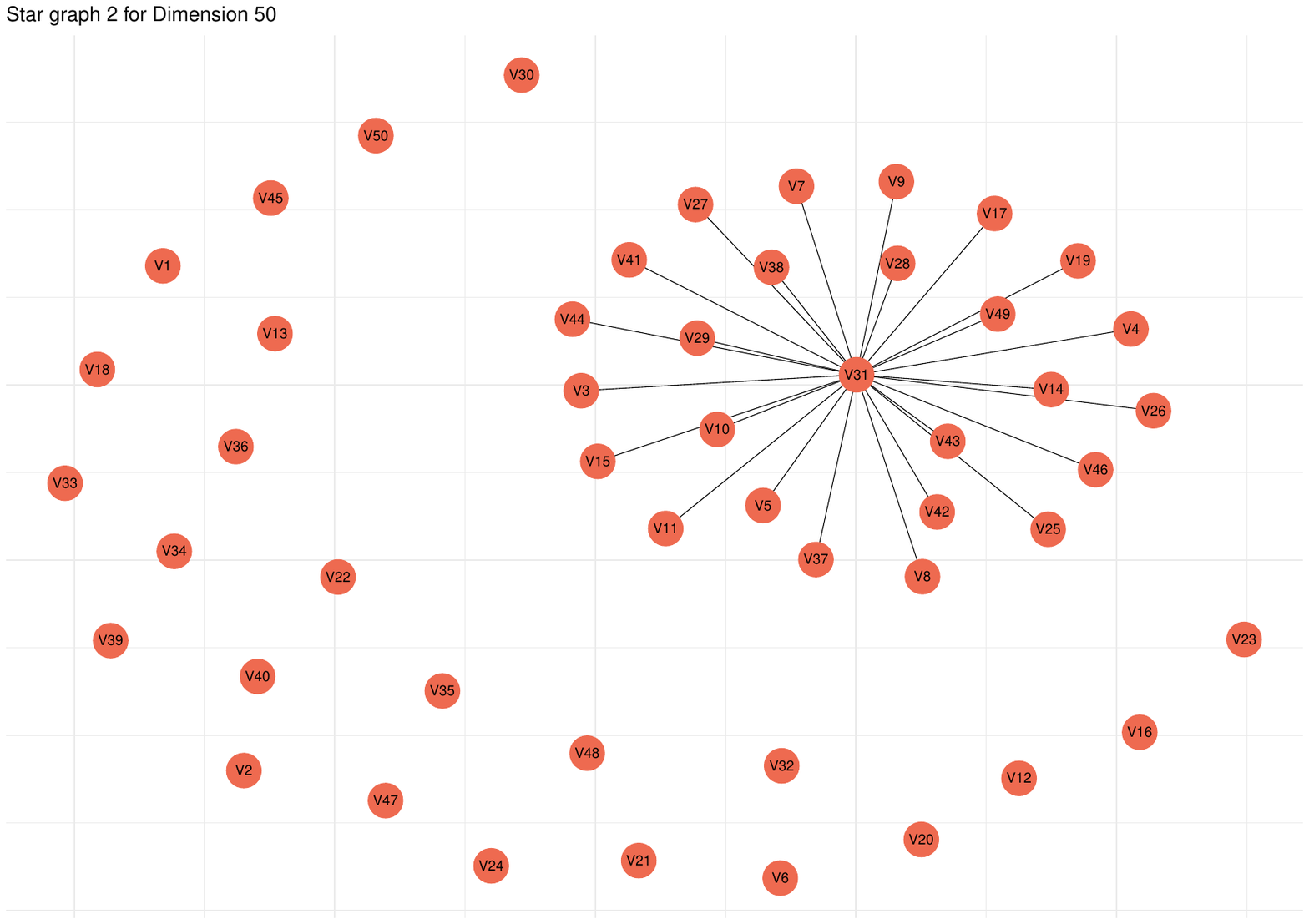}
\label{Figure 1}
\end{figure}
\par\end{center}
\subsubsection{Model consistency against sample size}
We illustrate the result, inferring about the probability of recovering the correctly signed edge-set for both populations in Figure \ref{Figure 2}. For a fixed sample size $n$ and dimension $p$, we report the proportion of the samples successfully identifying the edge set with signs matching from the true precision matrices for both the population models. For the star graphs, we take $d_{max}=25$.
\begin{center}
\begin{figure}[H]
\caption{Model consistency against sample size}

\centering{}\includegraphics[scale=0.26]{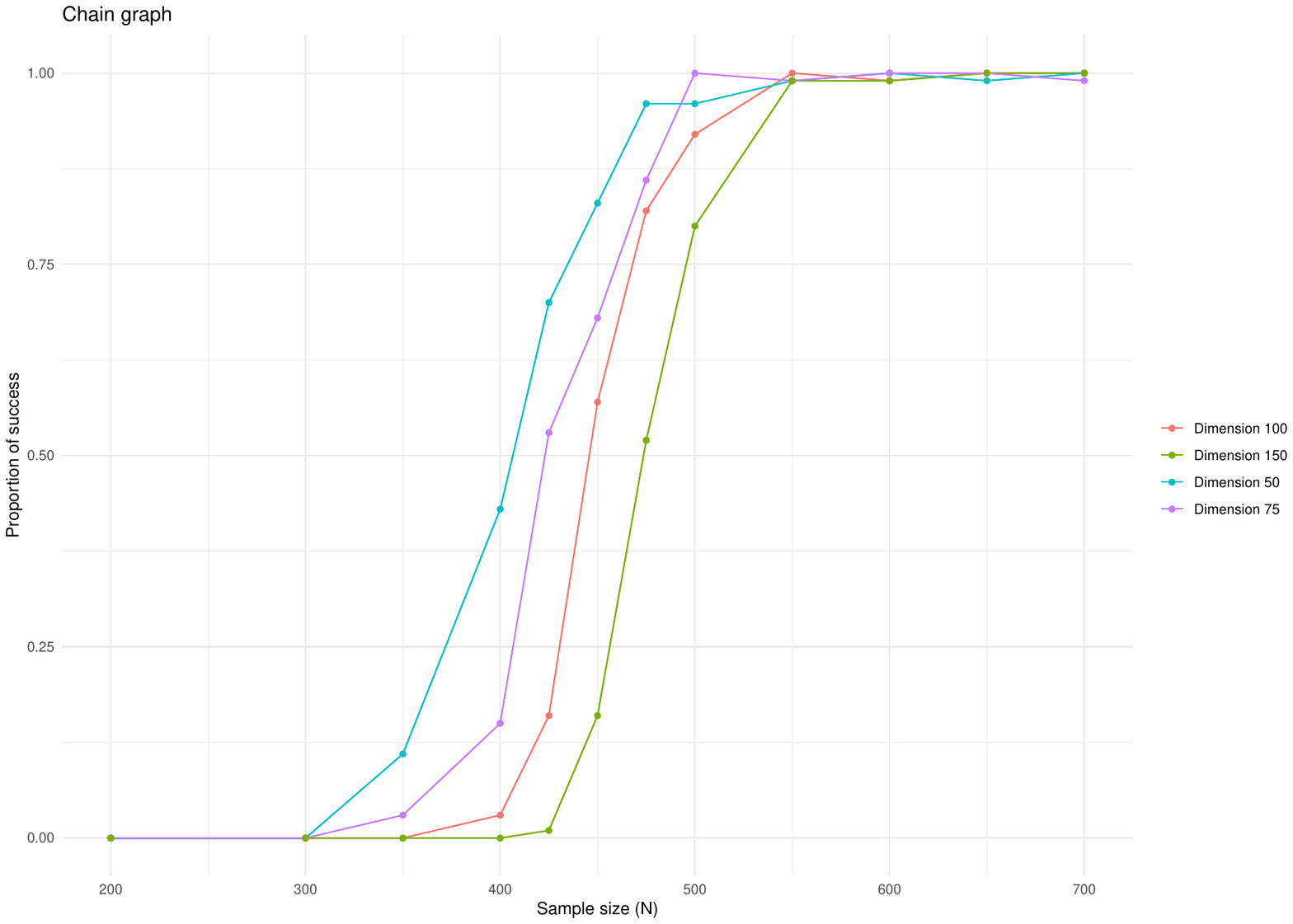}\includegraphics[scale=0.26]{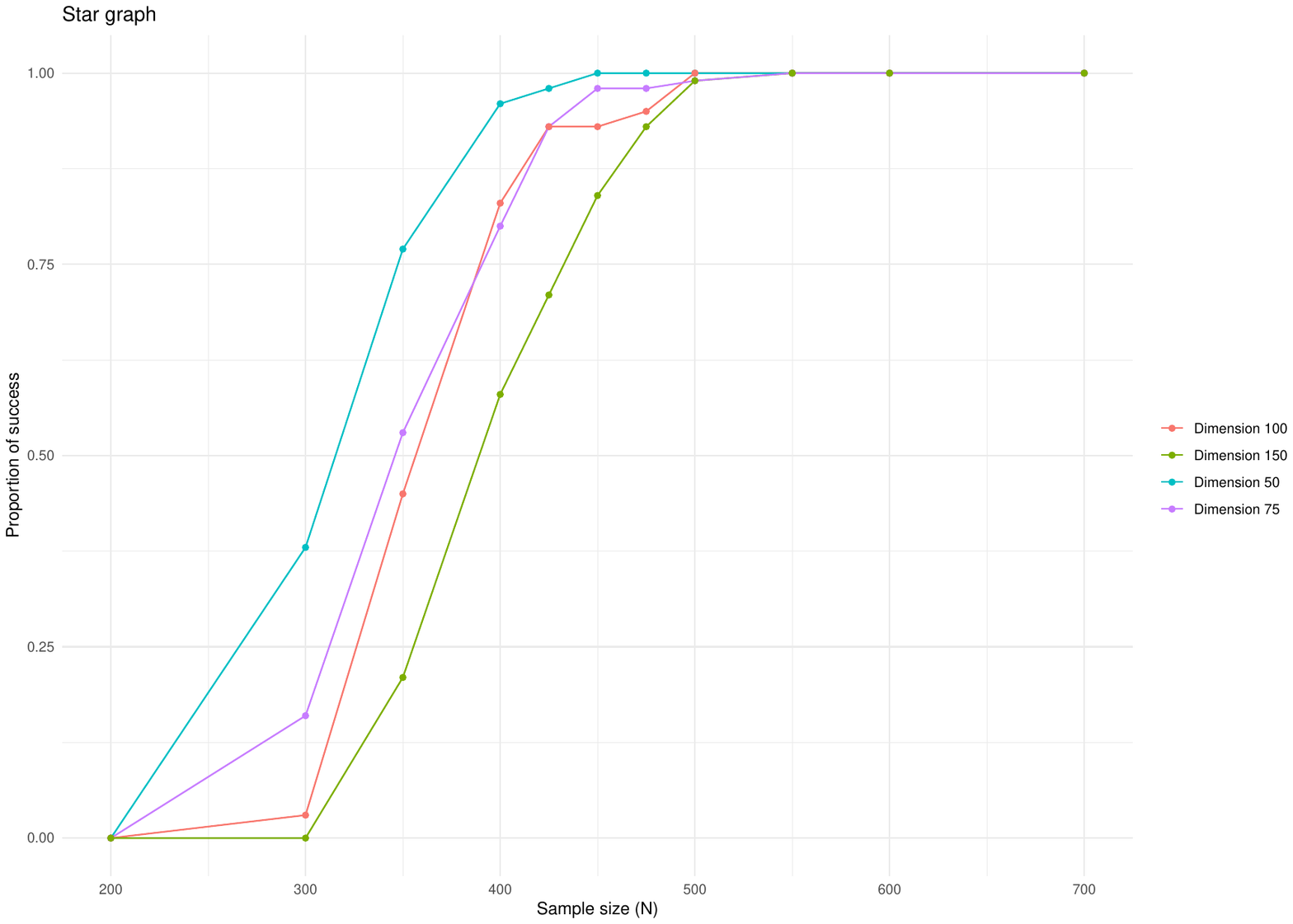}
\label{Figure 2}
\end{figure}
\par\end{center}

As it is expected, the probability of success shall start from 0, and then climb to 1 as the sample size increases, making the model consistent. Also, as the dimension increases, the consistency should be slower, and it can also be seen from the figures for both graphs.


\subsubsection{True positives and false positives against sample size}
We demonstrate the true positives and false positives against the sample size in this section. The number of true positives should increase as the sample size increases for both graphs, whereas the rate of false negatives should decrease. Figure \ref{Figure 3} and Figure \ref{Figure 4} demonstrate this interpretation. The results are given in terms of the number of true positive and false positive edges identified by the estimators.

\begin{figure}
\caption{True positive and false positive graphs against sample size (Chain
graph)}

\centering{}\includegraphics[scale=0.4]{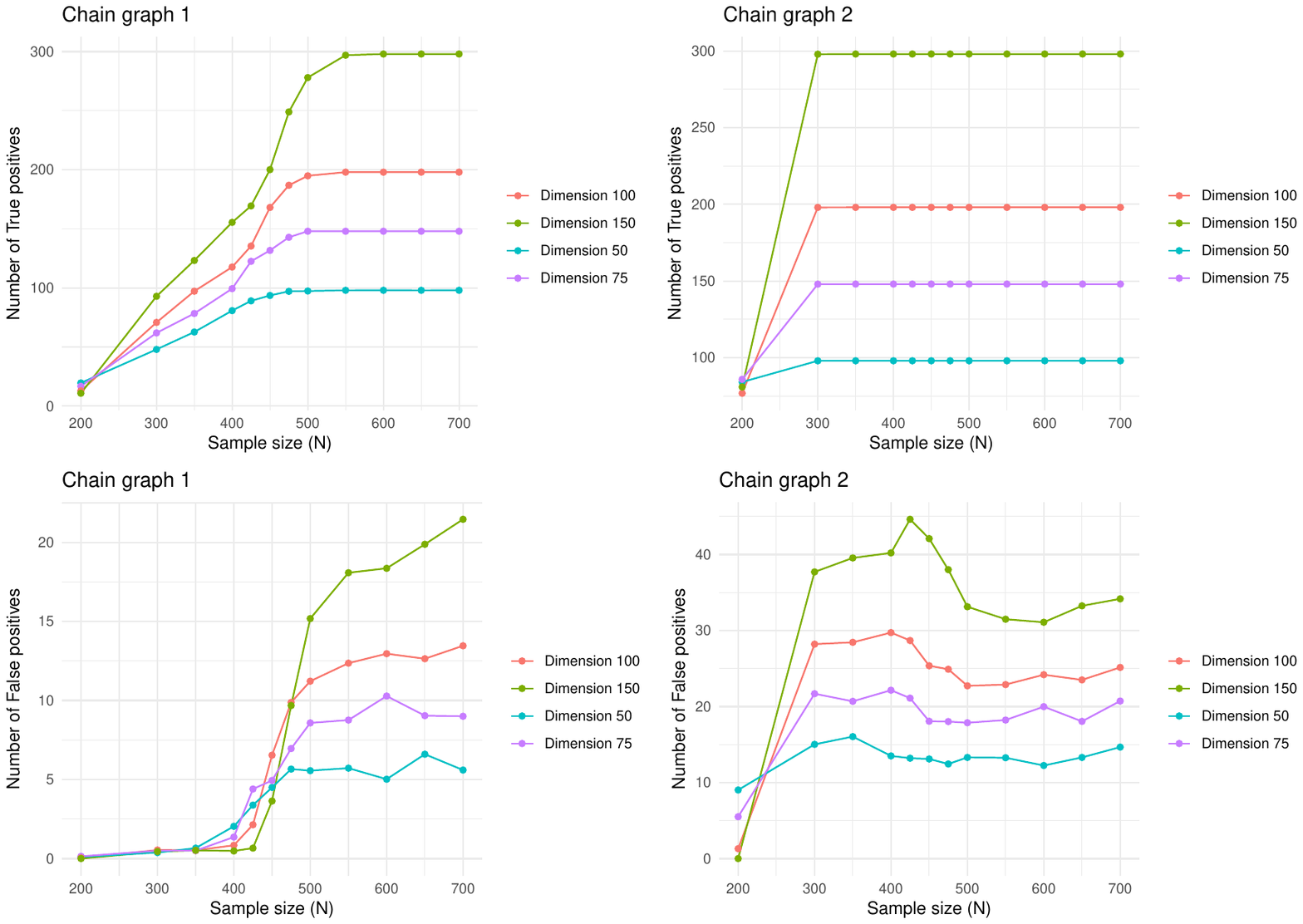}
\label{Figure 3}
\end{figure}

\begin{figure}
\caption{True positive and false positive graphs against sample size (Star
graph)}

\centering{}\includegraphics[scale=0.4]{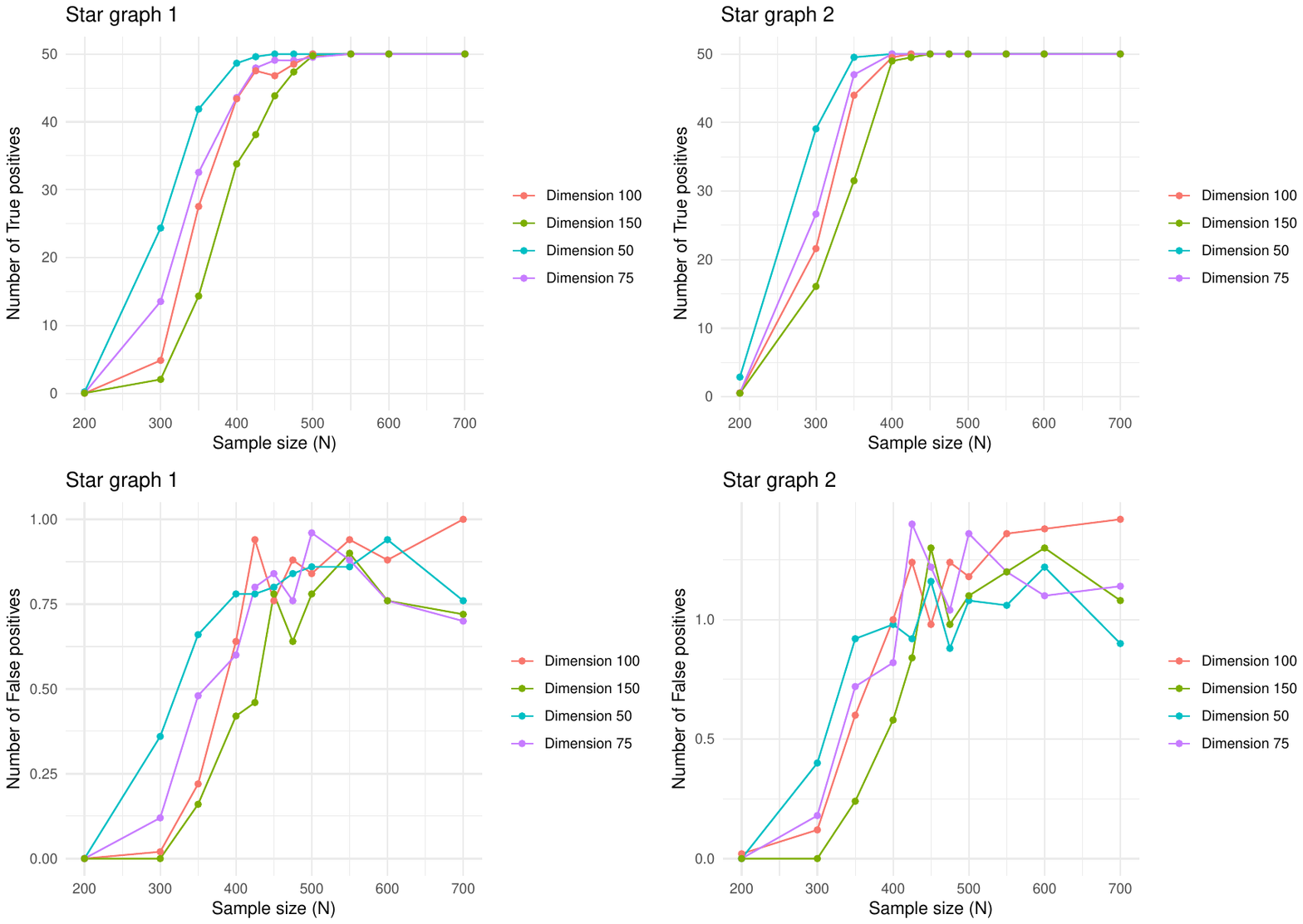}
\label{Figure 4}
\end{figure}
\subsubsection{Model consistency against maximum degree}
For this illustration, we select the star graph and take $p=100$ coordinates. However, the hub node has degree $d=25,30,36$. The precision matrix for population 1 has the structure
\begin{eqnarray*}
\Omega_{ij}^{1} & = & \begin{cases}
2 & ,i=j\\
\frac{5}{d} & ,i=u,j\in b\\
\frac{5}{d} & ,i\in b,j=u\\
0 & ,o.w.
\end{cases}
\end{eqnarray*}

, and the precision matrix for population 2 has the structure

\begin{eqnarray*}
\Omega_{ij}^{2} & = & \begin{cases}
2.5 & ,i=j\\
\frac{8}{d} & ,i=u,j\in b\\
\frac{8}{d} & ,i\in b,j=u\\
0 & ,o.w.
\end{cases}.
\end{eqnarray*}

We report the proportion of the samples successfully identifying the edge set with signs matching from the true precision matrices for both population models in Figure \ref{Figure 5}. As the maximum degree increases, the number of samples required for model consistency increases, as expected.

\begin{figure}[H]
\caption{Model consistency against maximum degree for star graphs}

\centering{}\includegraphics[scale=0.4]{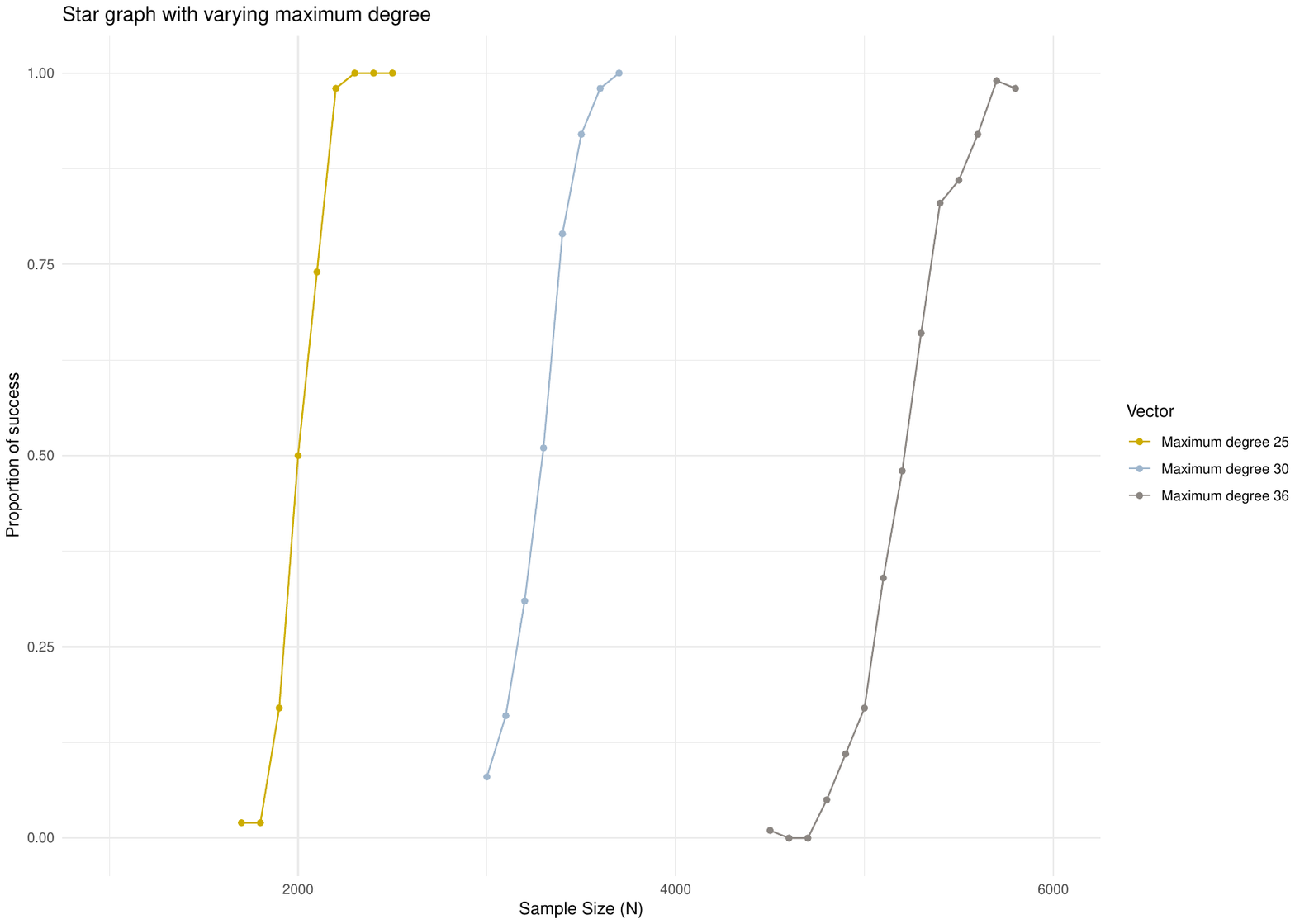}
\label{Figure 5}
\end{figure}
\subsubsection{Sup-norm convergence against sample size}
We report the average of the sup-norm distance between the predicted and true precision matrices for a fixed sample size and dimension, and for both populations. The sup-norm distance decreases as the sample size is increasing. However, the dimension affects the performance in the way that the convergence is slower as the dimension increases as reflected in Figure \ref{Figure 6} and Figure \ref{Figure 7}.

\begin{figure}[H]
\caption{Sup-norm distance against sample size (Chain graphs)}

\centering{}\includegraphics[scale=0.26]{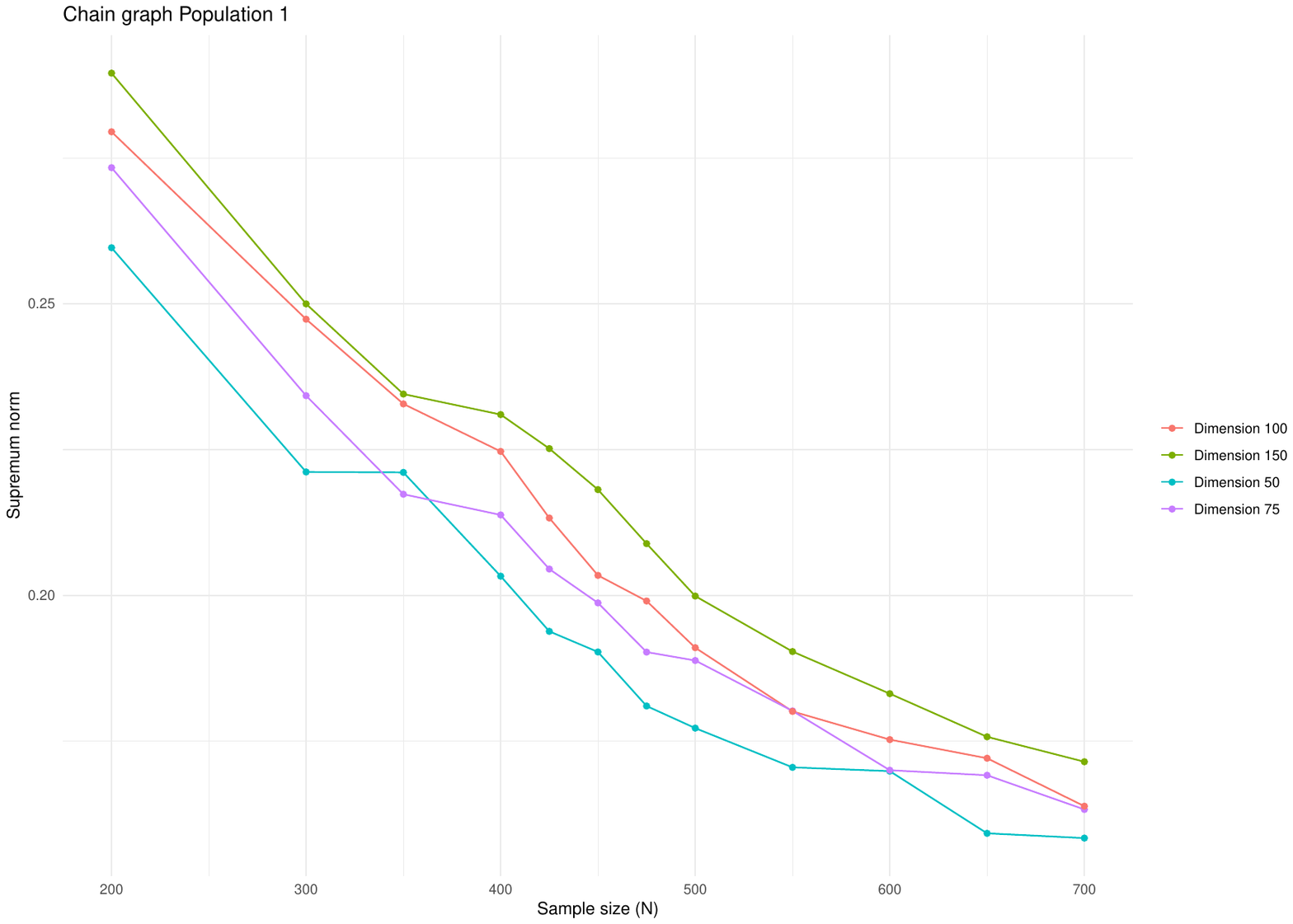}\includegraphics[scale=0.26]{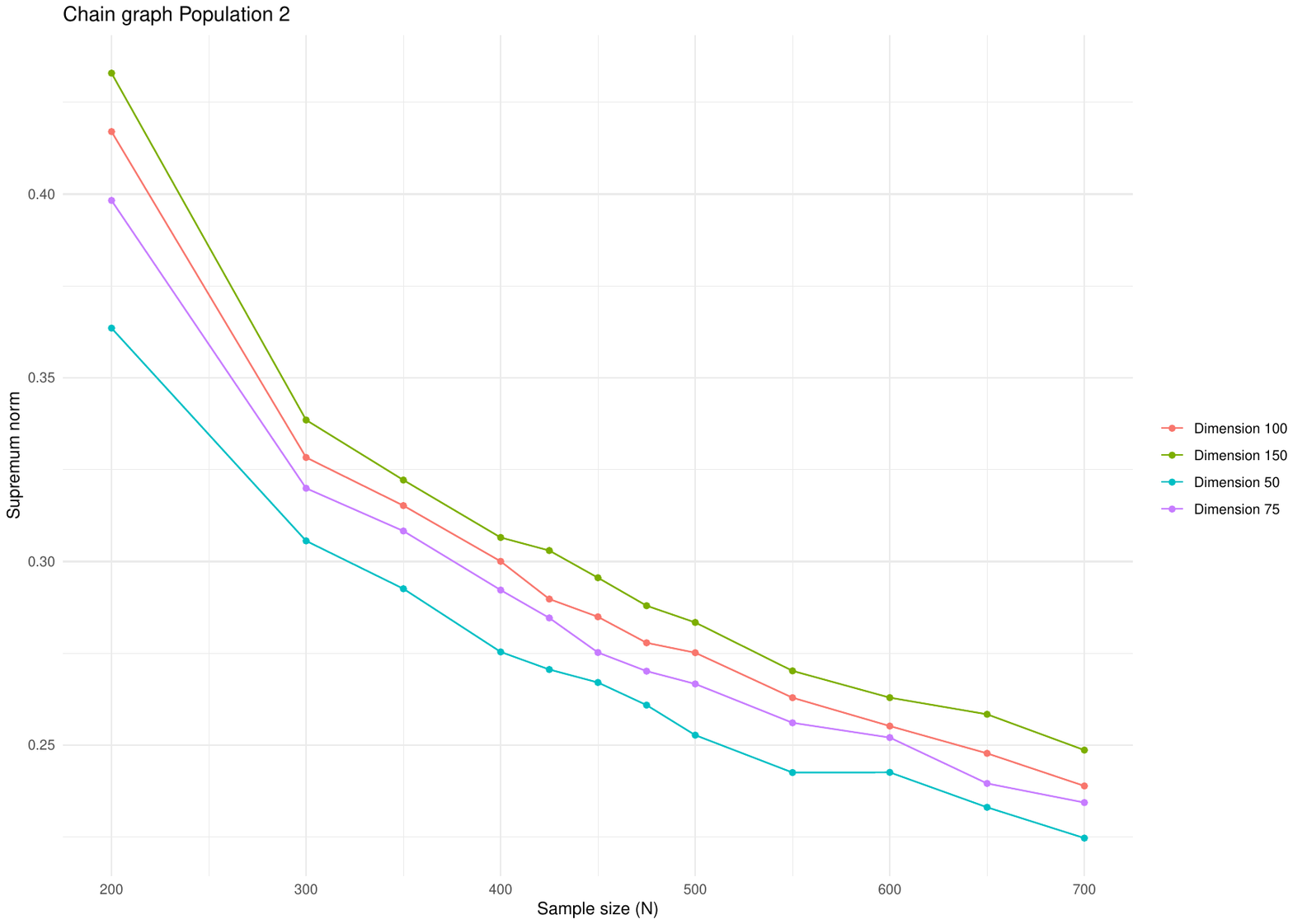}
\label{Figure 6}
\end{figure}

\begin{figure}[H]
\caption{Sup-norm distance against sample size(Star graphs)}

\centering{}\includegraphics[scale=0.26]{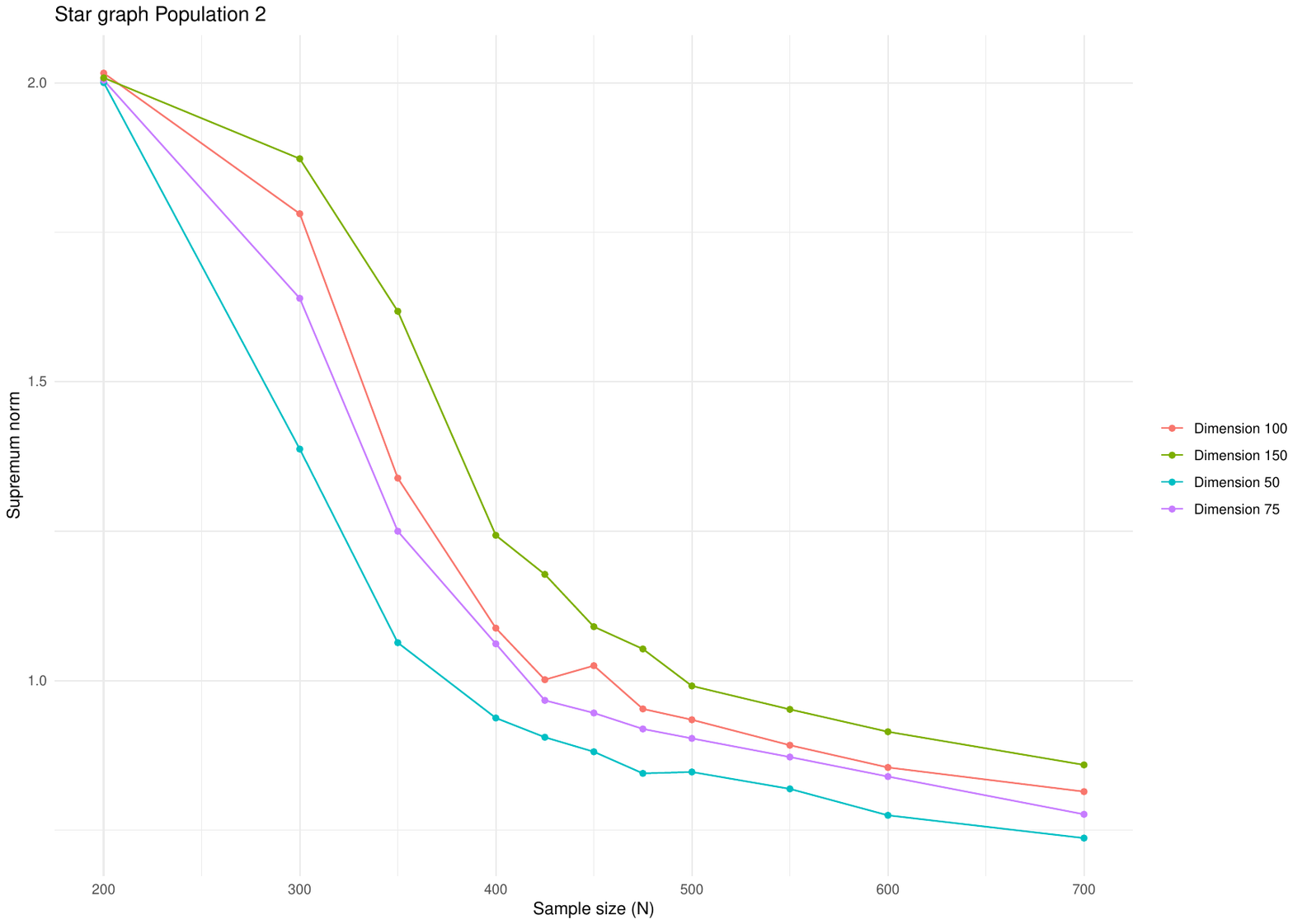}\includegraphics[scale=0.26]{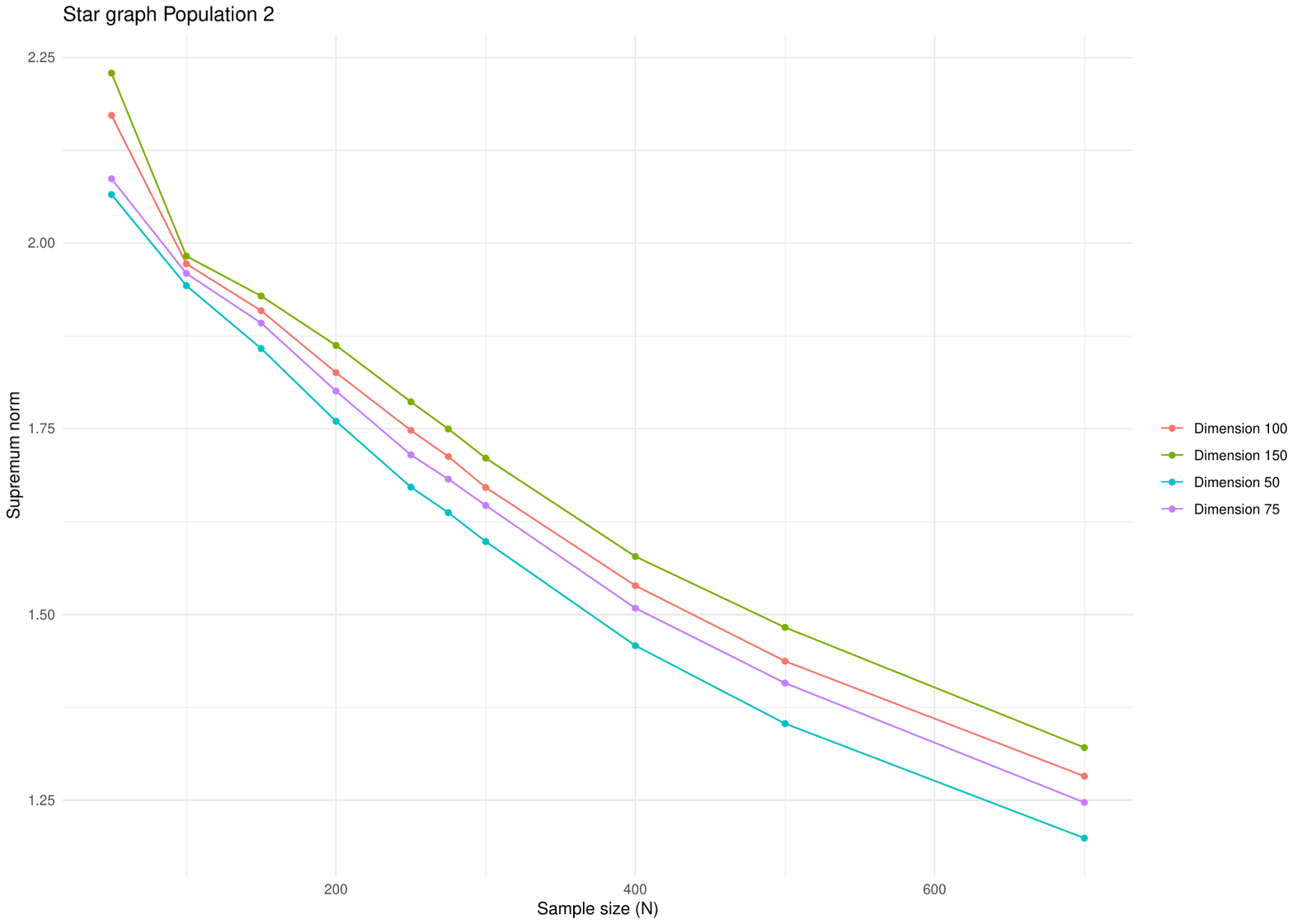}
\label{Figure 7}
\end{figure}
\subsubsection{Asymptotic normality of the test statistic}
Till now, we have dealt with the remainder term of the debiased graphical lasso. In this part, we shall illustrate the asymptotic normality of the debiased graphical lasso estimators for both populations. We fix the sample size at $n=600,$ and report the histogram of each debiased estimator and the test statistic $T_{ij}=\hat{\Omega}_{d,ij}^{1}-\hat{\Omega}_{d,ij}^{2}$. We also superimpose the standard normal curve on the histogram as a reference to how close the asymptotic normality result is. For demonstration, we take $(i,j)=\{(1,1),(1,2),(2,3),(3,4)\}$ as the edges for the chain graphs. The edges for demonstration of the star graphs are mentioned in Table \ref{Table 2}.

\begin{table}
\caption{Dimension and Edges for demonstration of test statistic for Star graphs}

\centering{}%
\begin{tabular}{|c|c|c|c|}
\hline 
$p=50$ & $p=75$ & $p=100$ & $p=150$\tabularnewline
\hline 
\hline 
(1,1) & (1,1) & (1,1) & (1,1)\tabularnewline
\hline 
(1,2) & (1,2) & (1,2) & (1,2)\tabularnewline
\hline 
(31,15) & (31,52) & (7,31) & (14,51)\tabularnewline
\hline 
(8,31) & (72,31) & (31,80) & (82,14)\tabularnewline
\hline 
\end{tabular}
\label{Table 2}
\end{table}

The following figures explain the asymptotic normality of the debiased
estimators.
\begin{figure}[H]
\caption{Asymptotic normality of $\hat{\Omega}_{d,ij}^{1}$ and $\hat{\Omega}_{d,ij}^{2}$
for Dimension 50 (Chain graphs)}

\centering{}\includegraphics[scale=0.26]{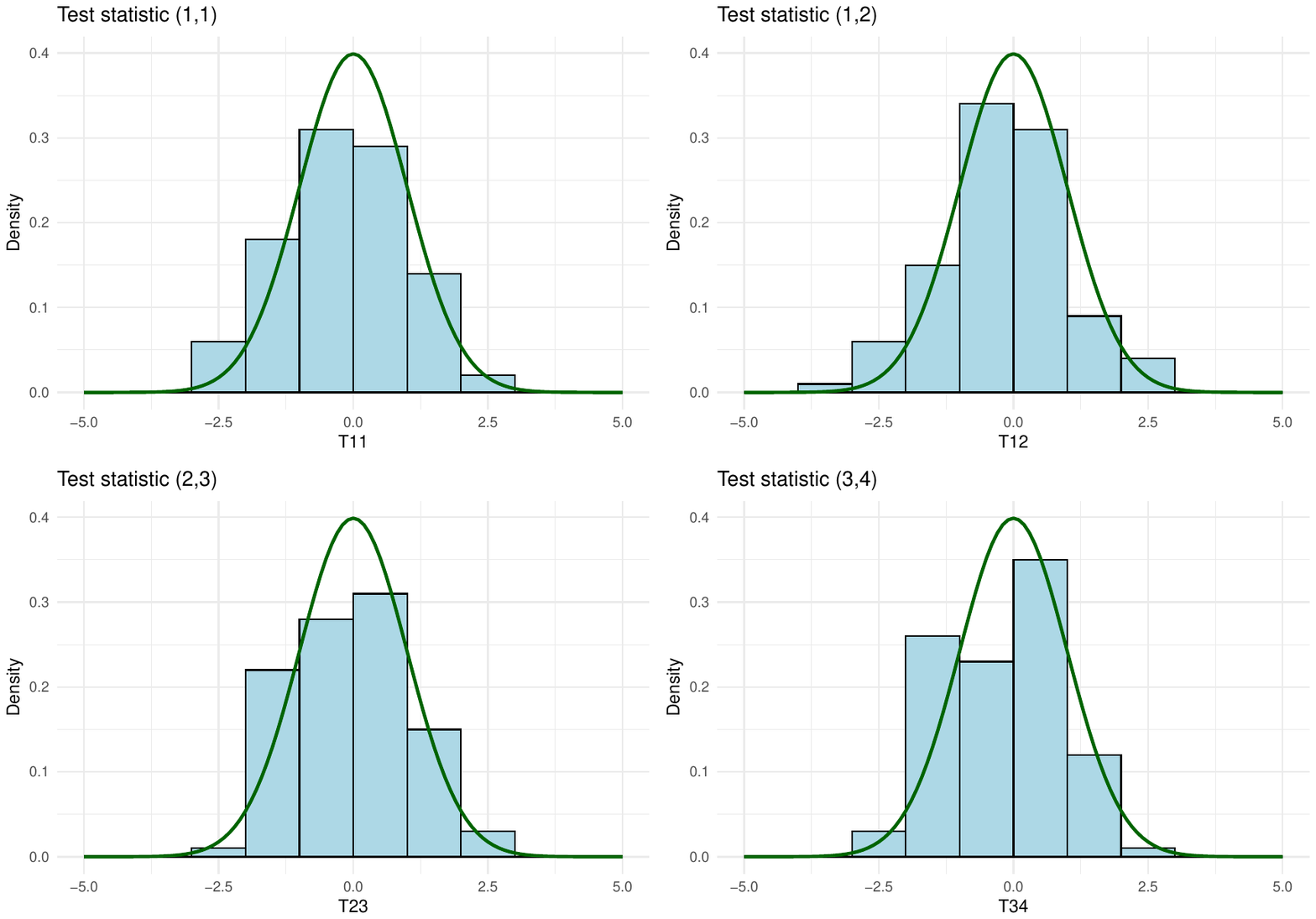}\includegraphics[scale=0.26]{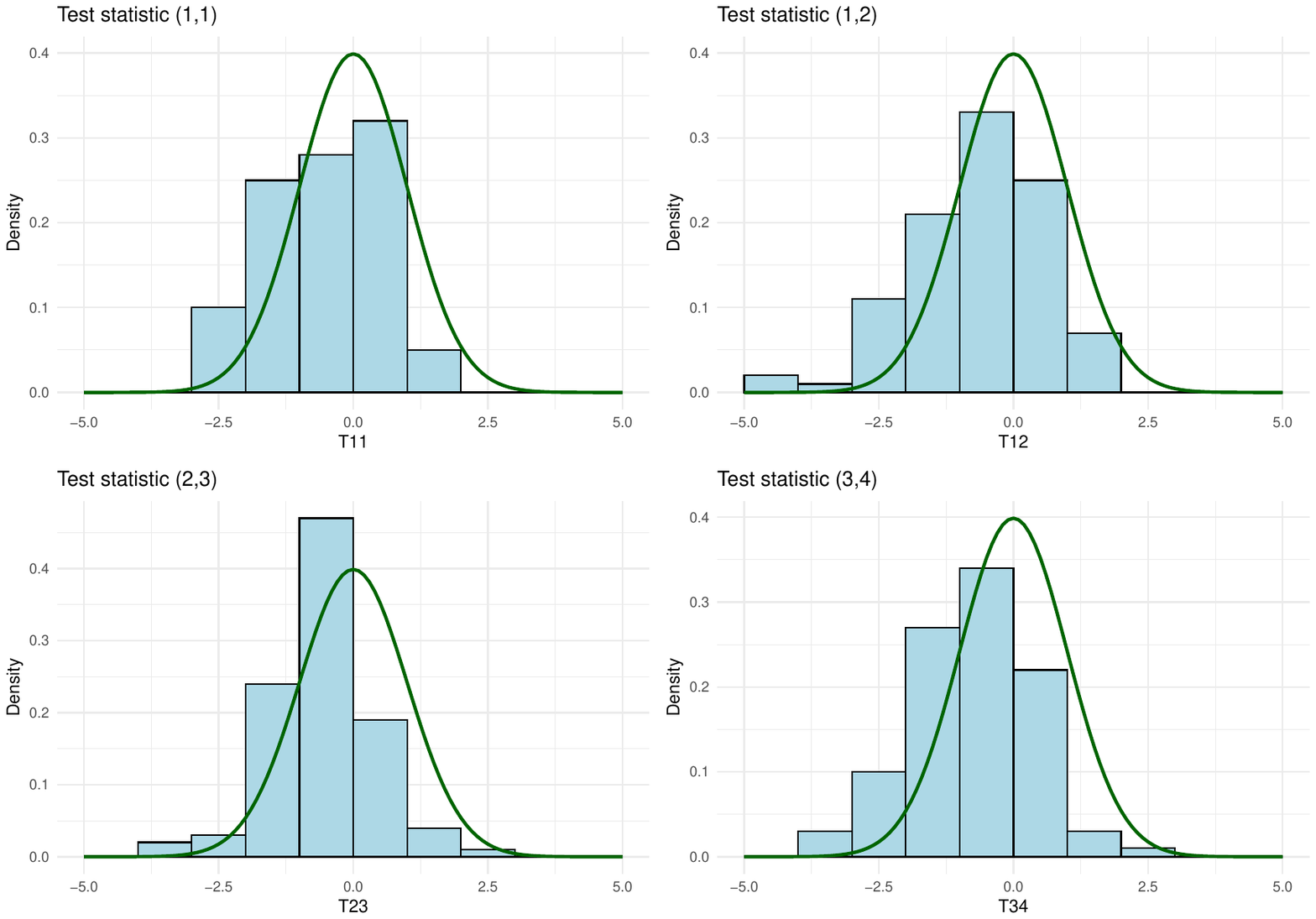}
\label{Figure 8}
\end{figure}

\begin{figure}[H]
\caption{Asymptotic normality of $\hat{\Omega}_{d,ij}^{1}$ and $\hat{\Omega}_{d,ij}^{2}$
for Dimension 75 (Chain graphs)}

\centering{}\includegraphics[scale=0.26]{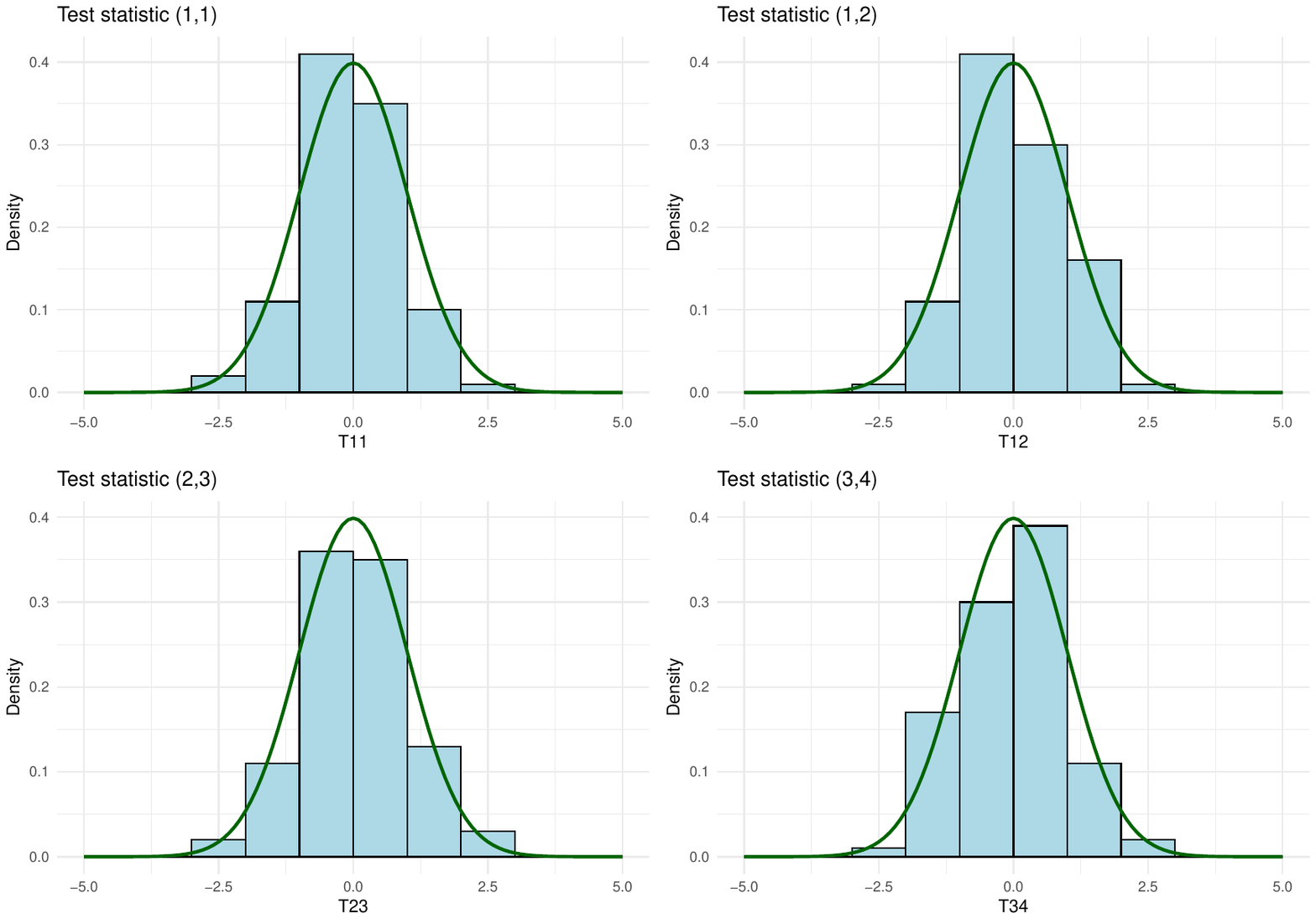}\includegraphics[scale=0.26]{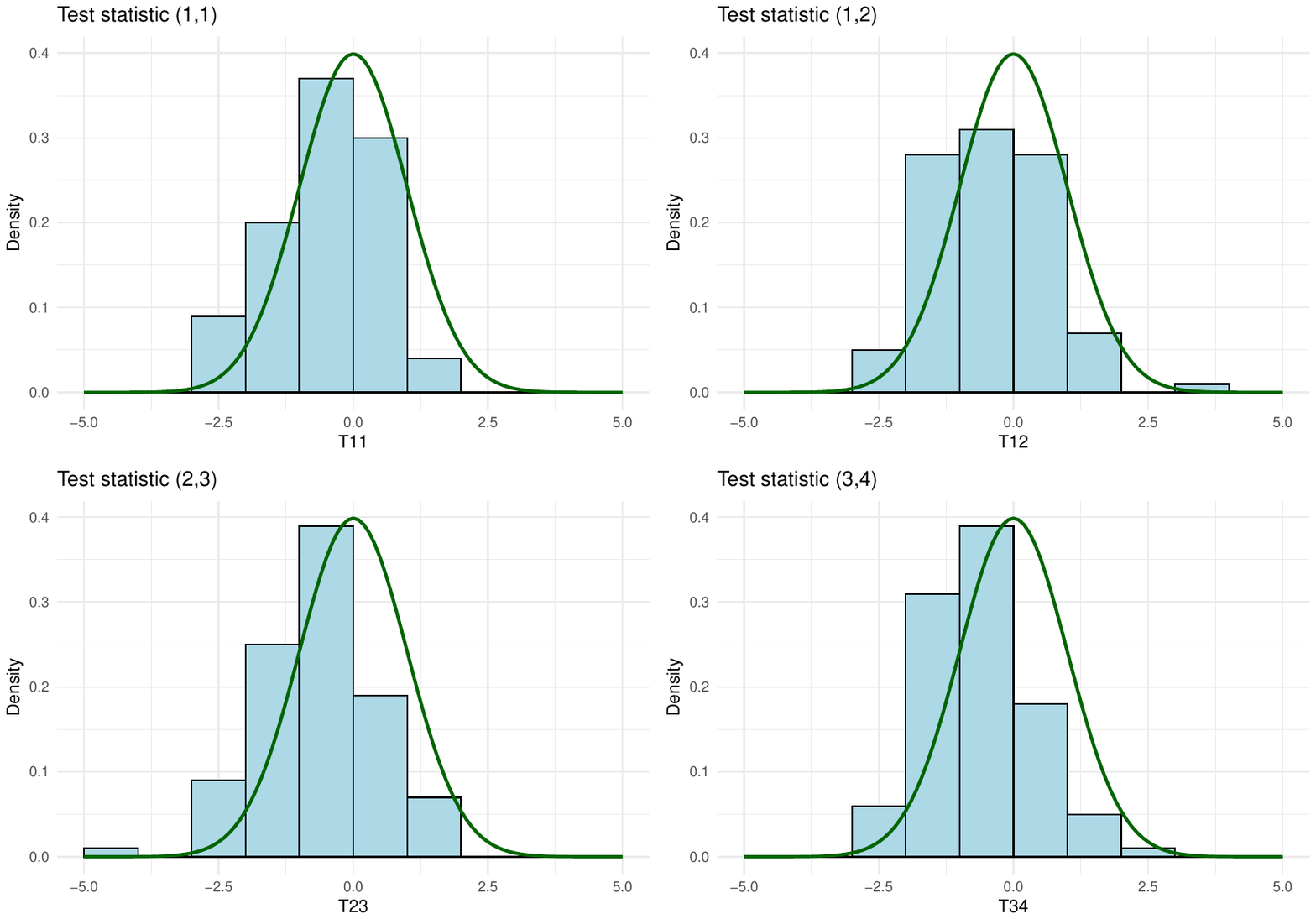}
\label{Figure 9}
\end{figure}

\begin{figure}[H]
\caption{Asymptotic normality of $\hat{\Omega}_{d,ij}^{1}$ and $\hat{\Omega}_{d,ij}^{2}$
for Dimension 100 (Chain graphs)}

\centering{}\includegraphics[scale=0.26]{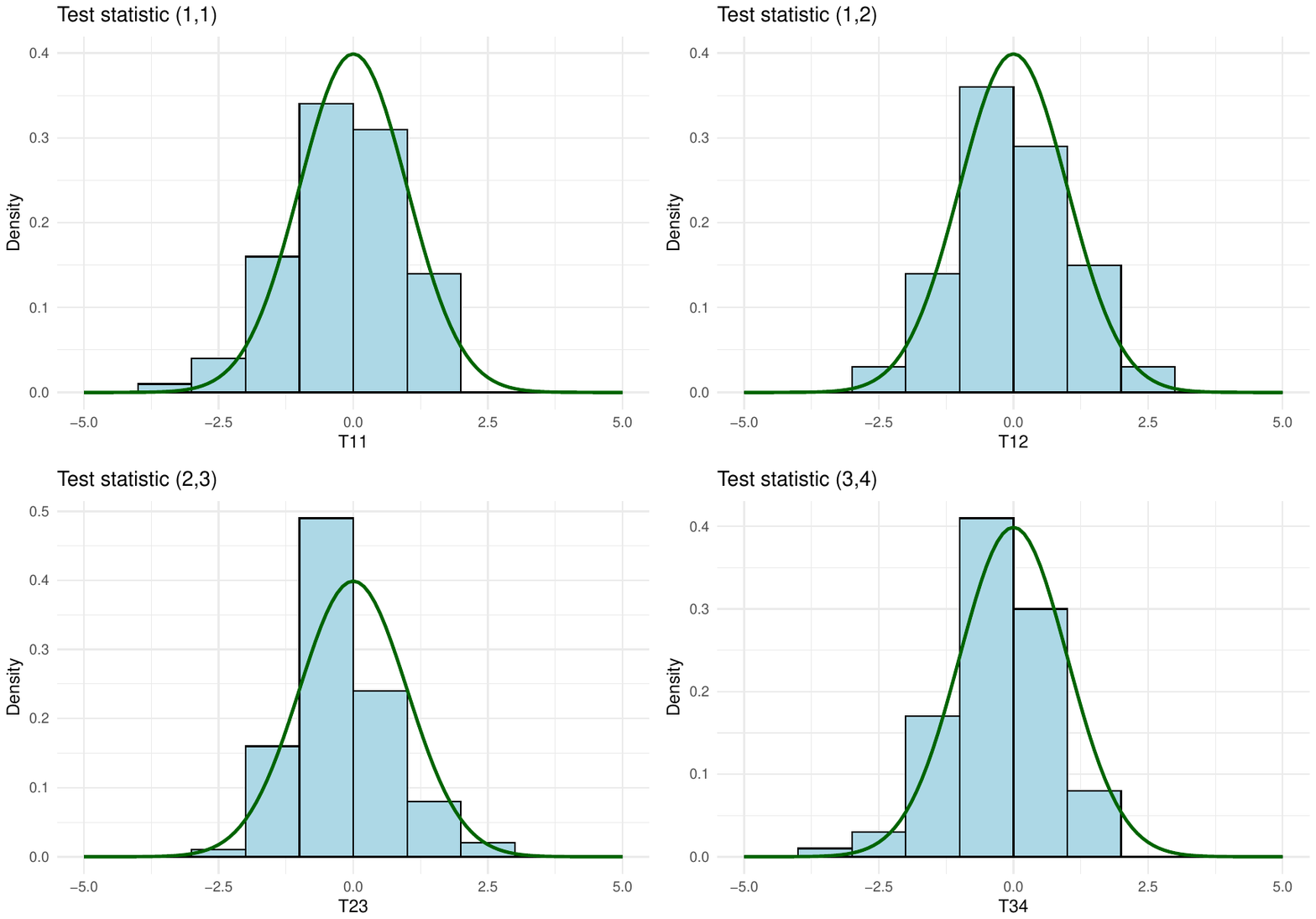}\includegraphics[scale=0.26]{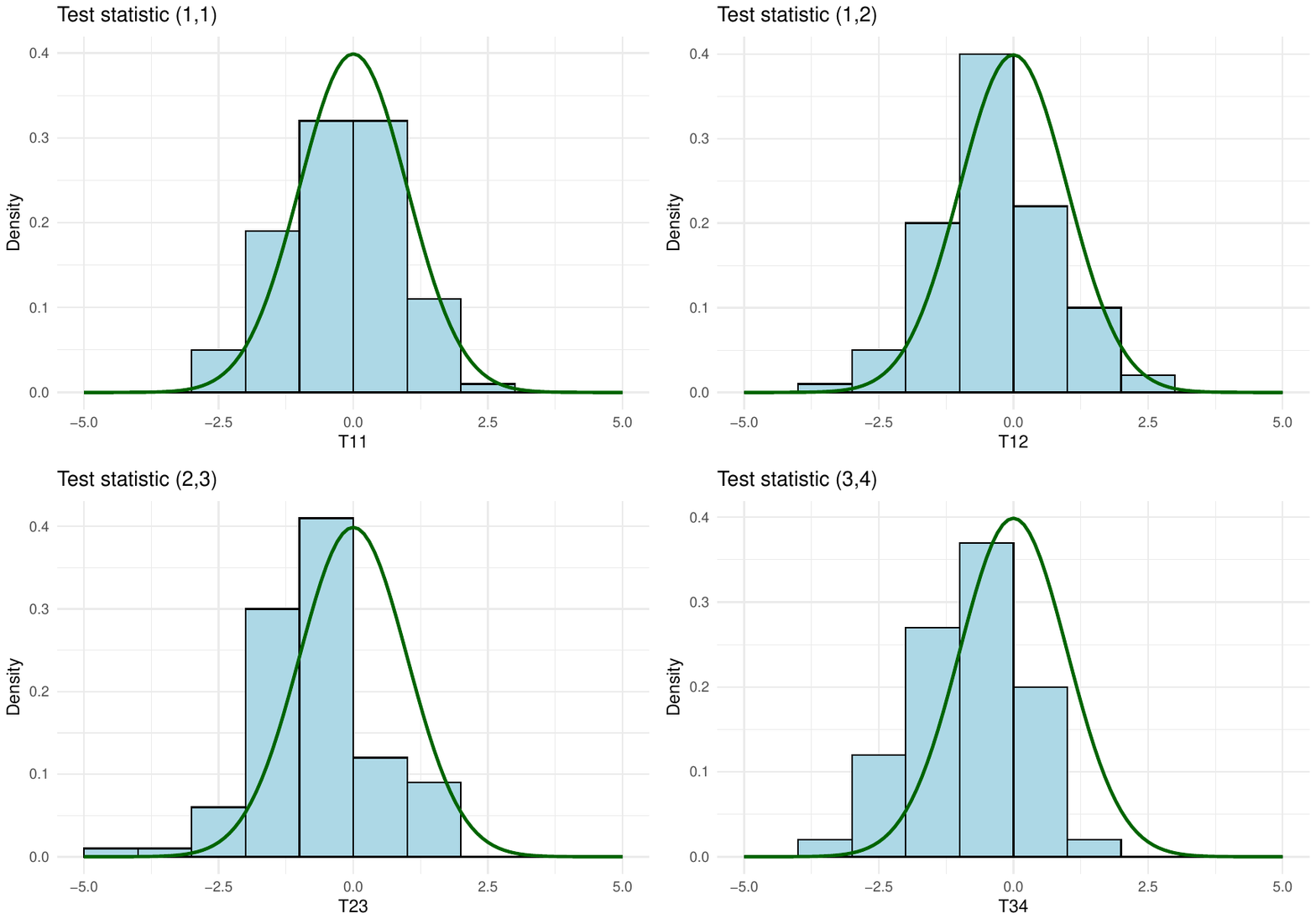}
\label{Figure 10}
\end{figure}

\begin{figure}[H]
\caption{Asymptotic normality of $\hat{\Omega}_{d,ij}^{1}$ and $\hat{\Omega}_{d,ij}^{2}$
for Dimension 150 (Chain graphs)}

\centering{}\includegraphics[scale=0.26]{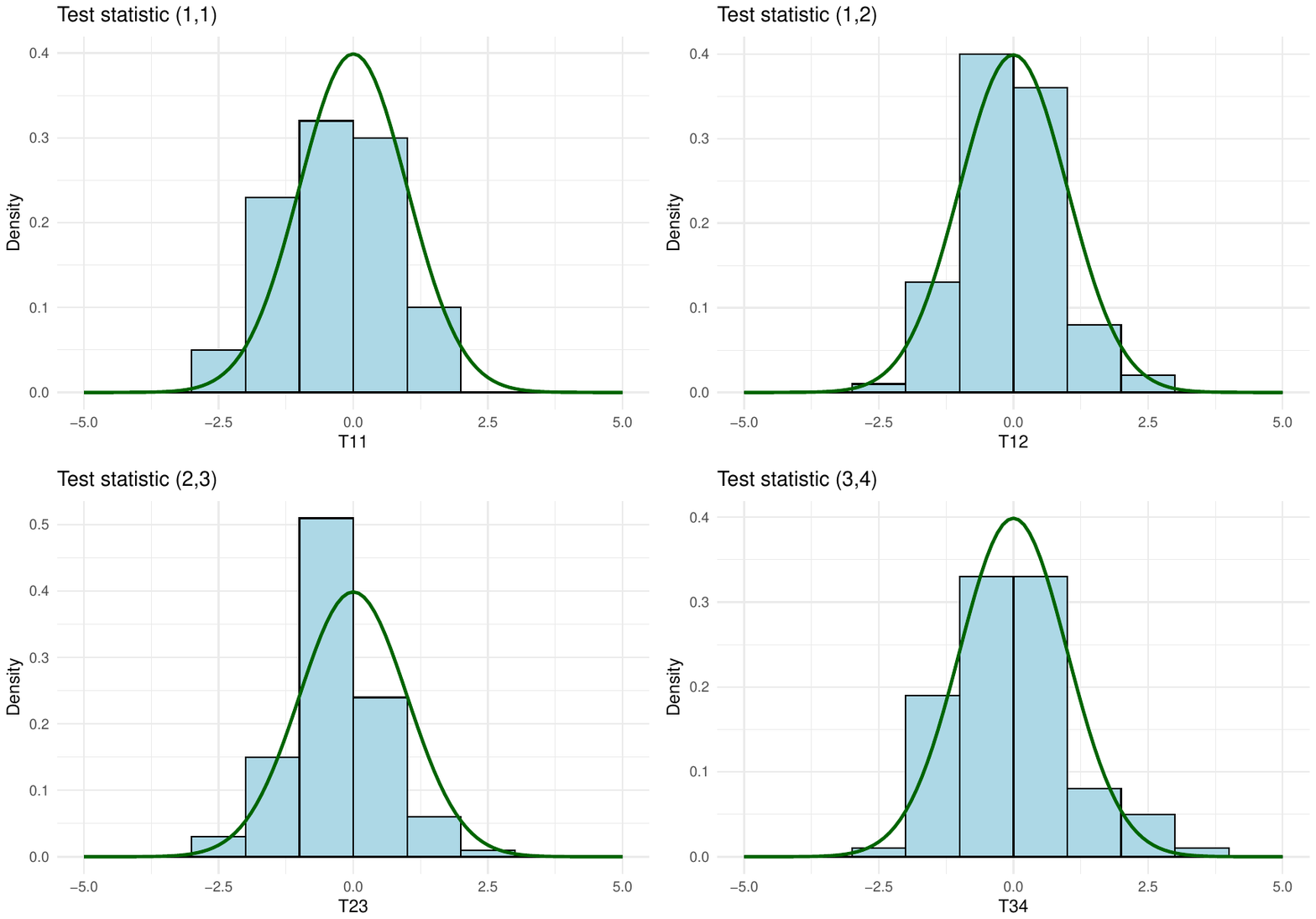}\includegraphics[scale=0.26]{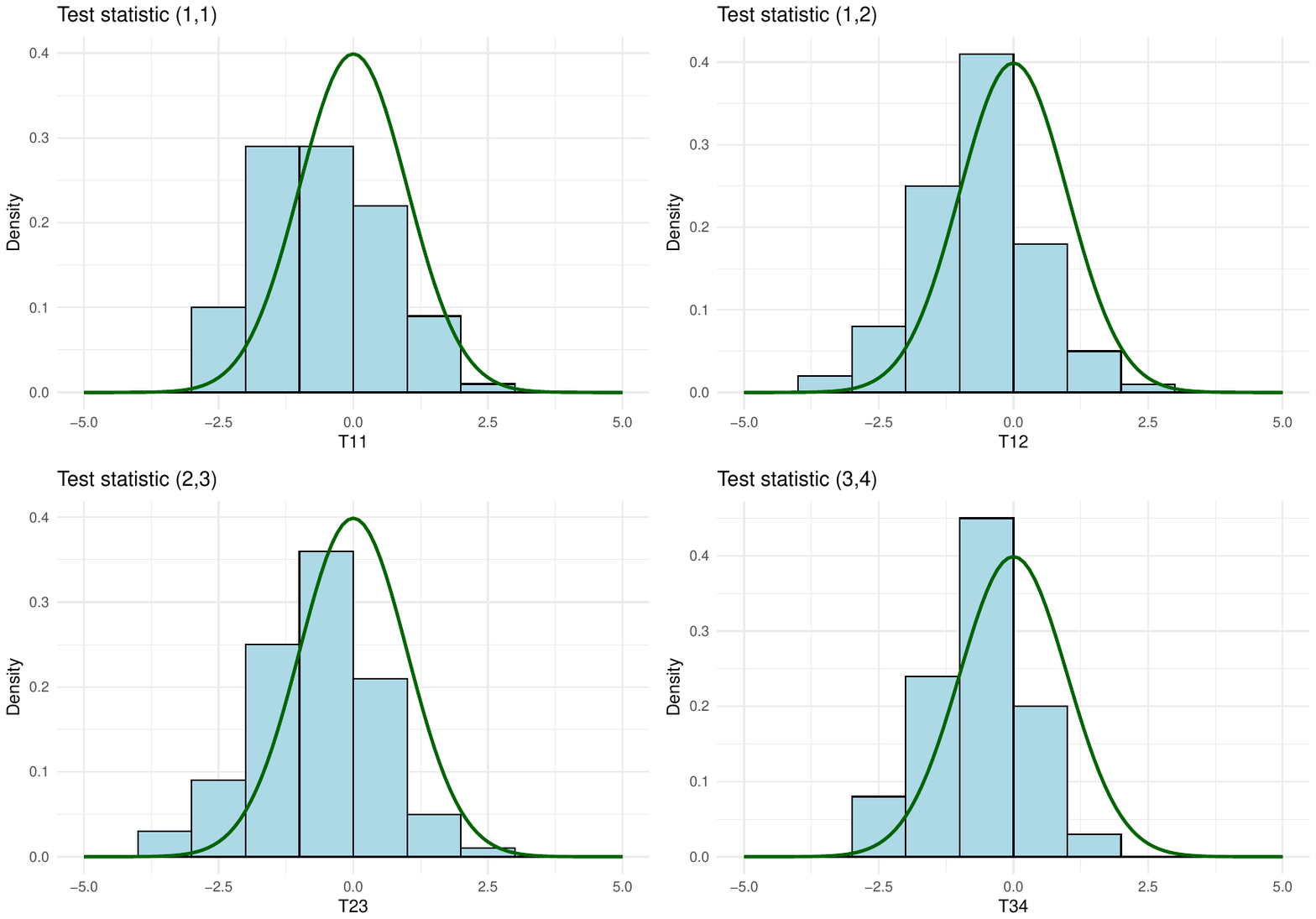}
\label{Figure 11}
\end{figure}

\begin{figure}[H]
\caption{Asymptotic normality of $\hat{\Omega}_{d,ij}^{1}$ and $\hat{\Omega}_{d,ij}^{2}$
for Dimension 50 (Star graphs)}

\centering{}\includegraphics[scale=0.26]{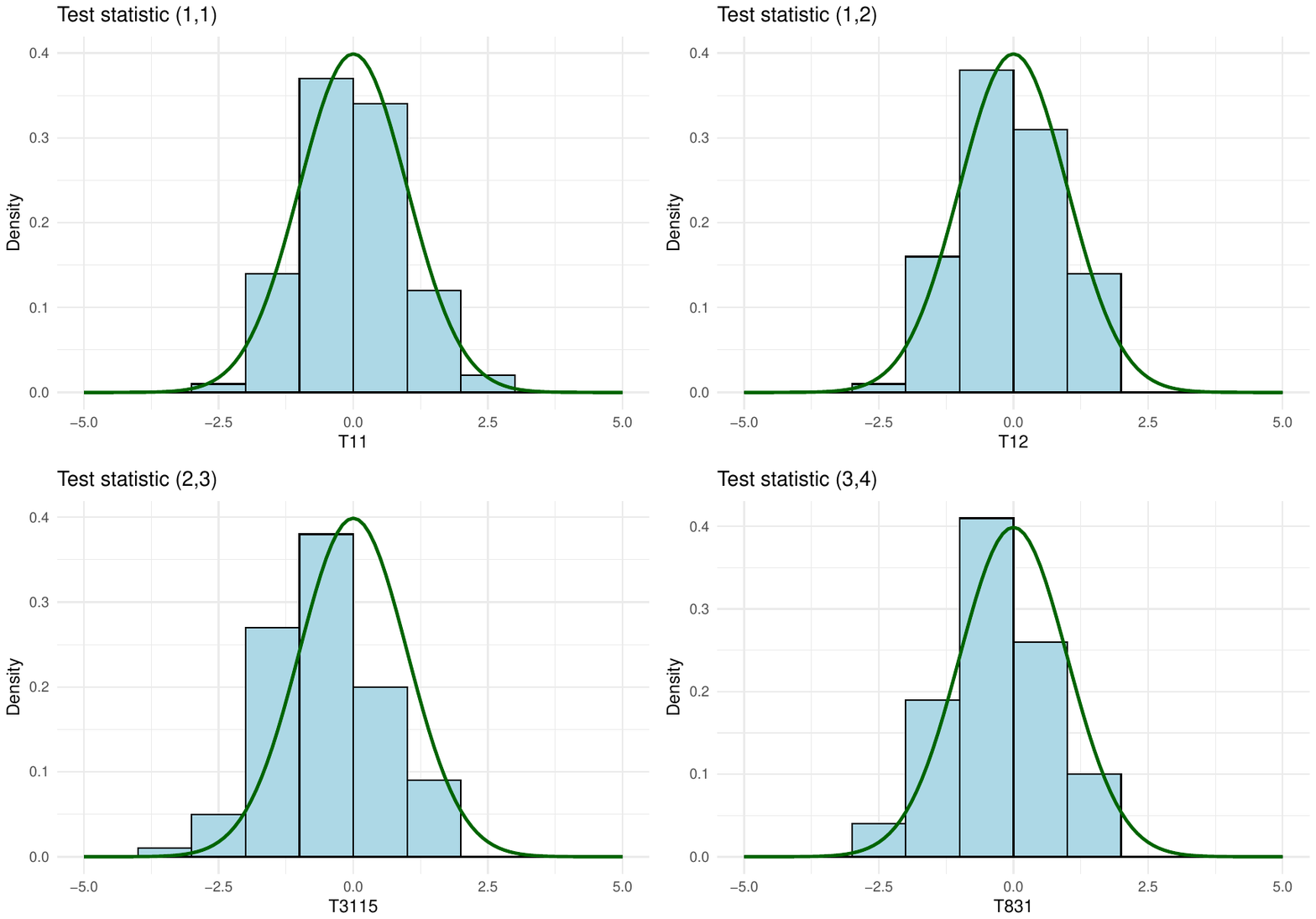}\includegraphics[scale=0.26]{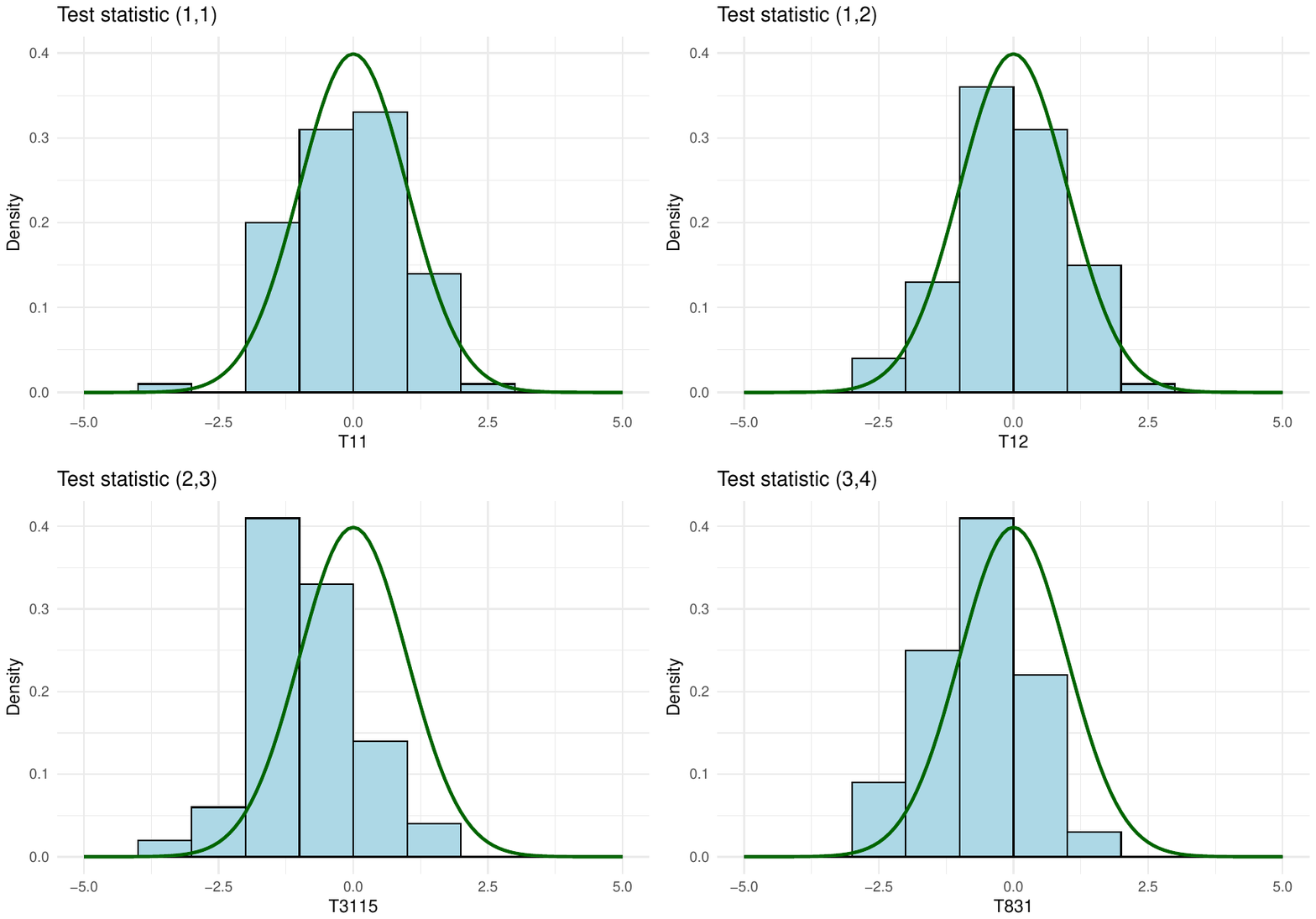}
\label{Figure 12}
\end{figure}

\begin{figure}[H]
\caption{Asymptotic normality of $\hat{\Omega}_{d,ij}^{1}$ and $\hat{\Omega}_{d,ij}^{2}$
for Dimension 75 (Star graphs)}

\centering{}\includegraphics[scale=0.26]{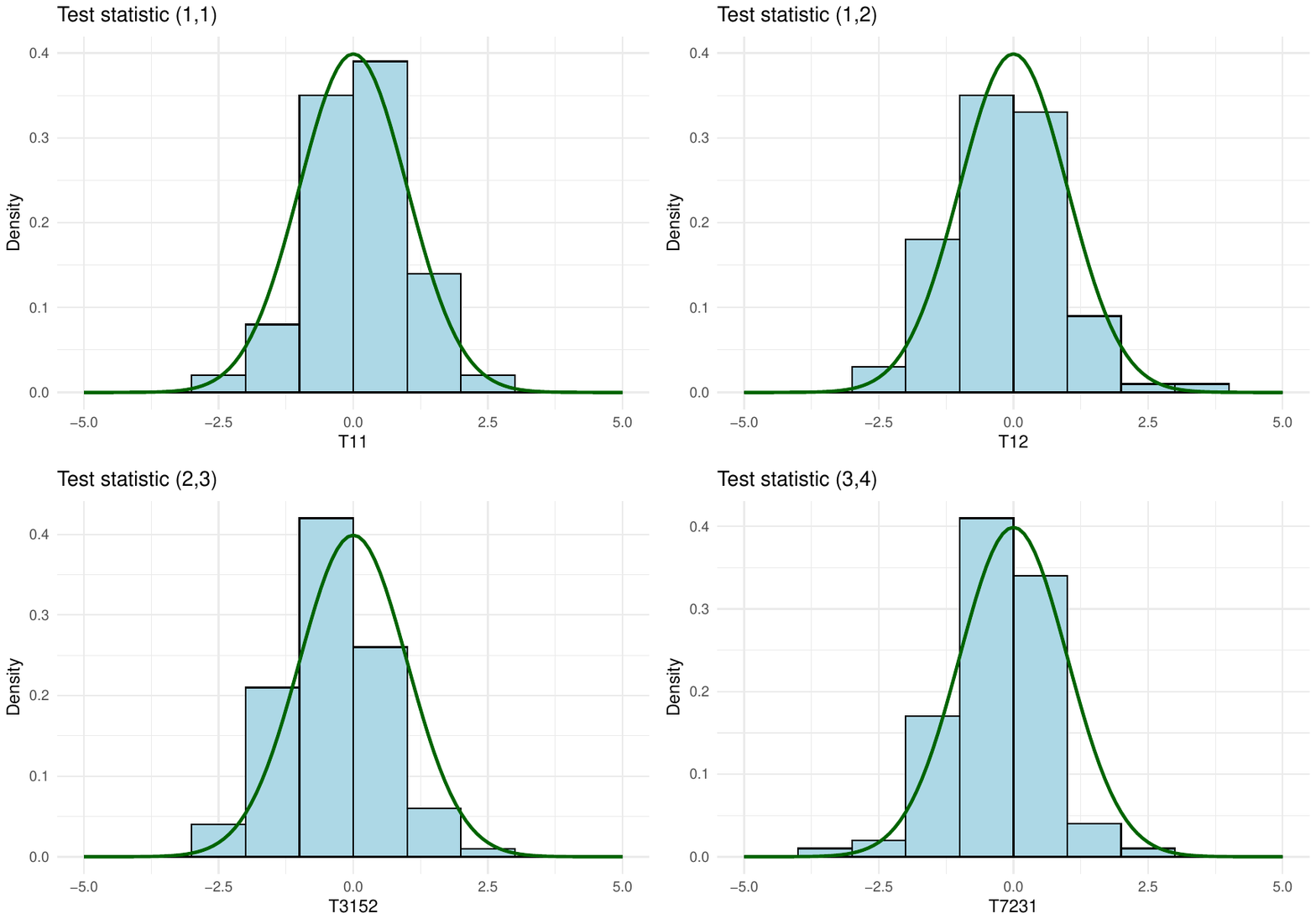}\includegraphics[scale=0.26]{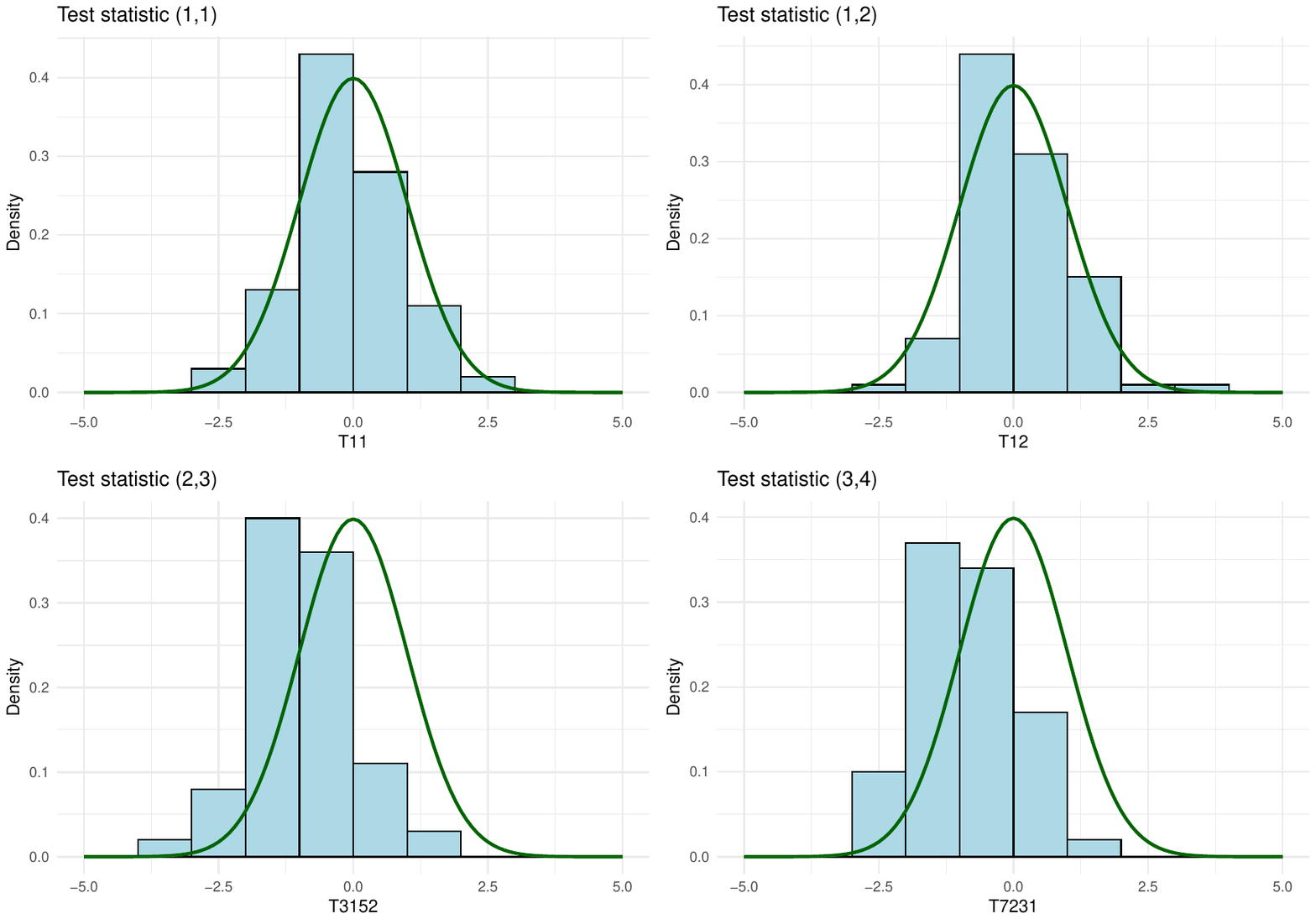}
\label{Figure 13}
\end{figure}

\begin{figure}[H]
\caption{Asymptotic normality of $\hat{\Omega}_{d,ij}^{1}$ and $\hat{\Omega}_{d,ij}^{2}$ for
Dimension 100 (Star graphs)}

\centering{}\includegraphics[scale=0.26]{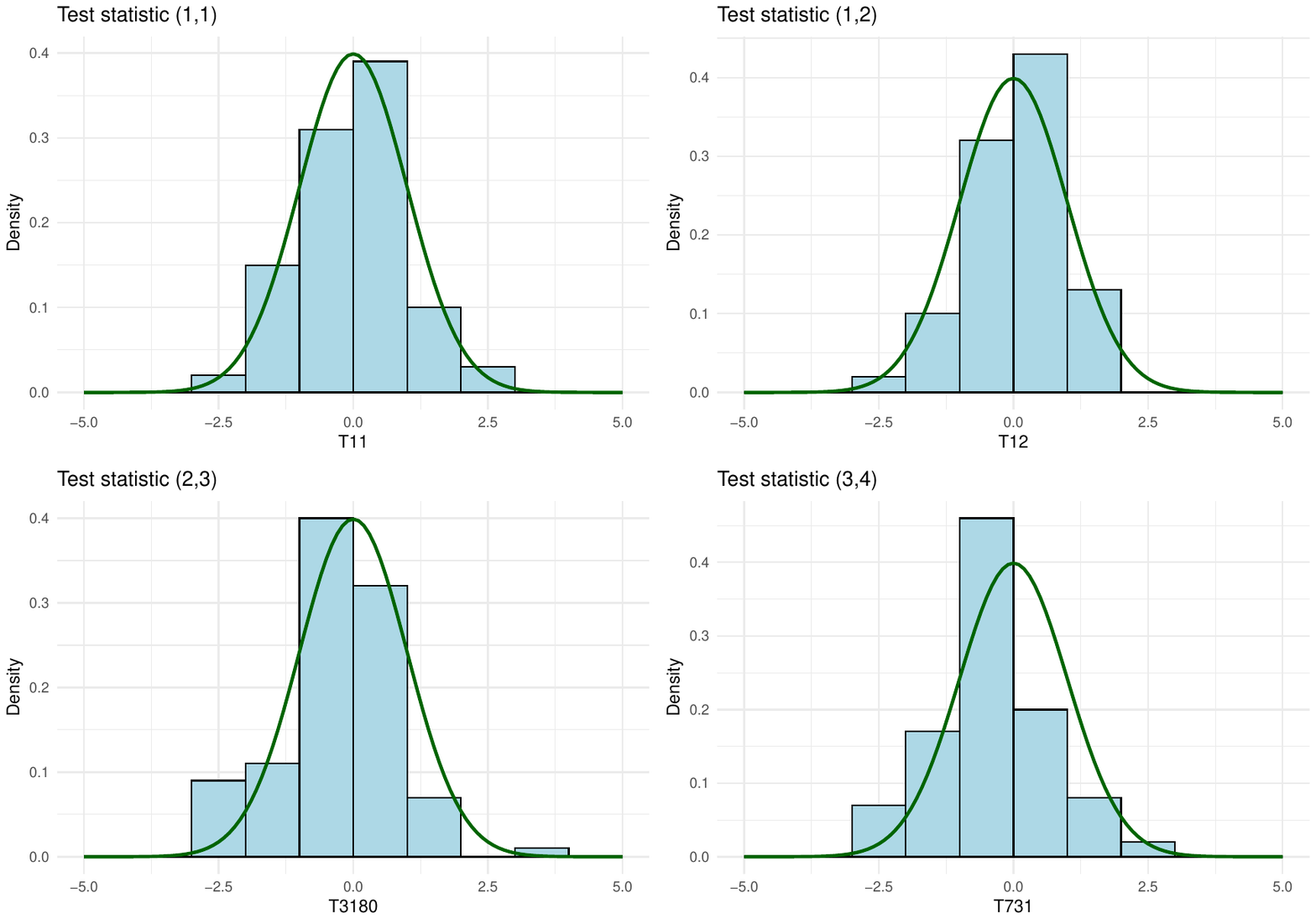}\includegraphics[scale=0.26]{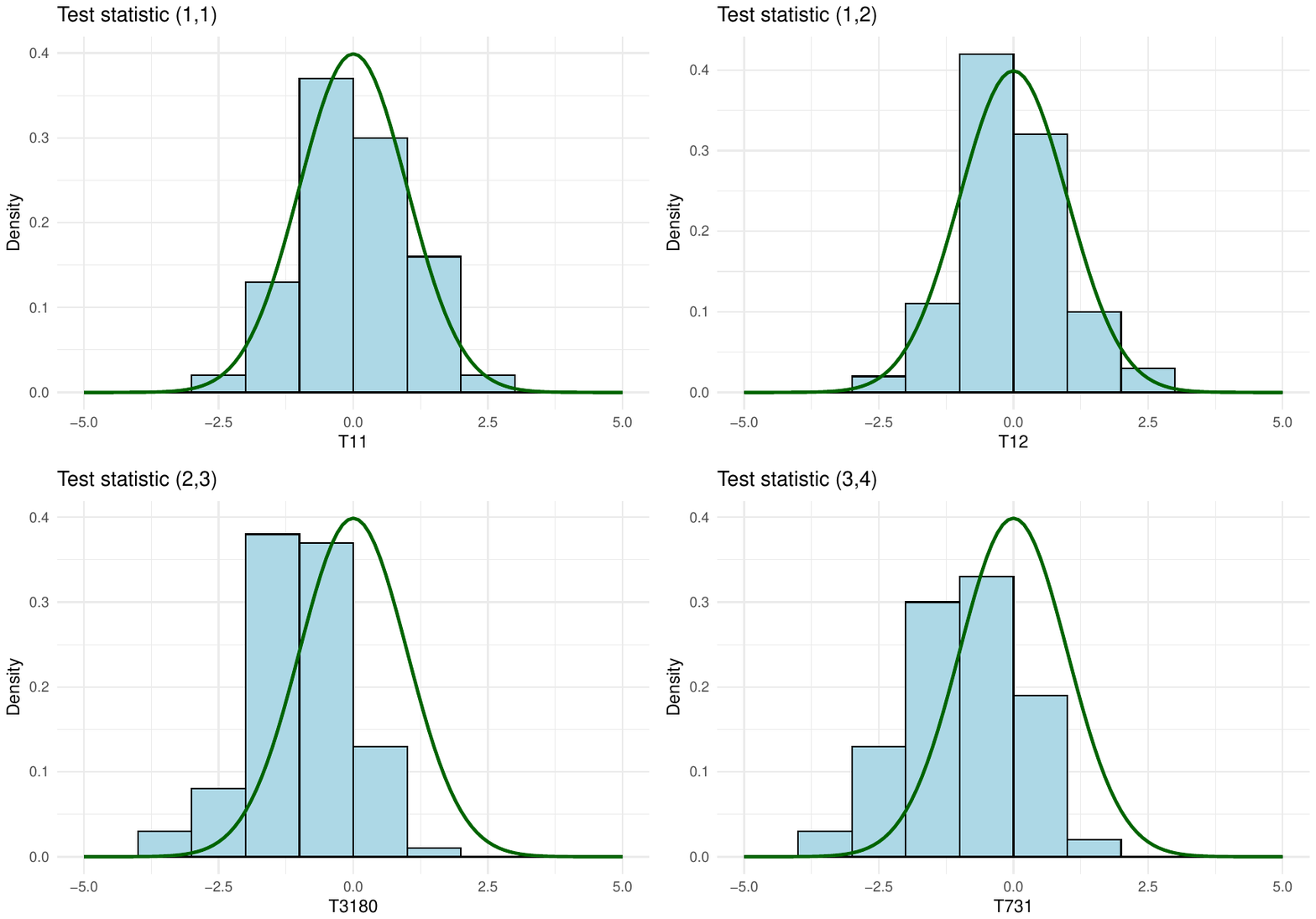}
\label{Figure 14}
\end{figure}

\begin{figure}[H]
\caption{Asymptotic normality of $\hat{\Omega}_{d,ij}^{1}$ and $\hat{\Omega}_{d,ij}^{2}$ for Dimension 150
(Star graphs)}

\centering{}\includegraphics[scale=0.26]{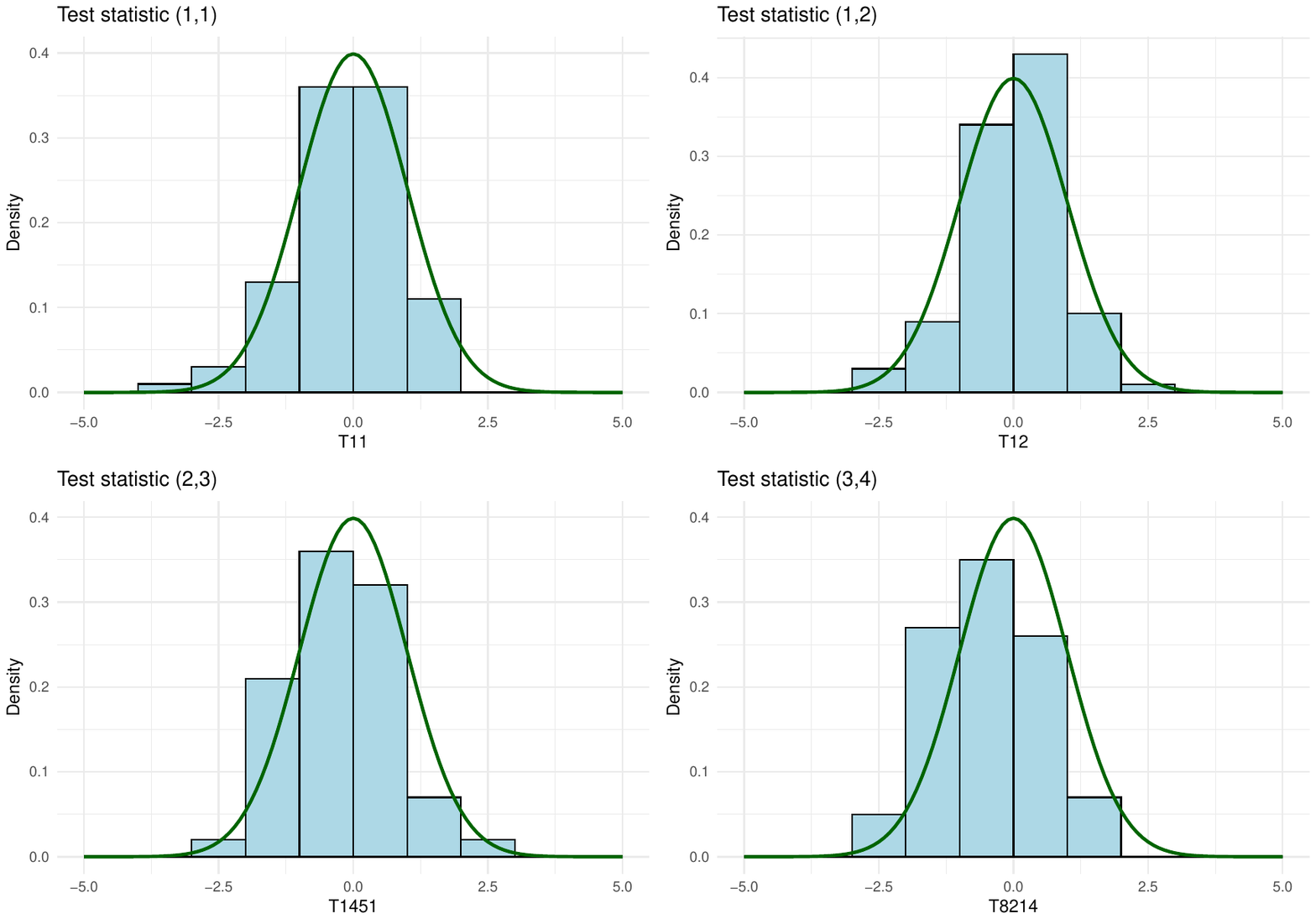}\includegraphics[scale=0.26]{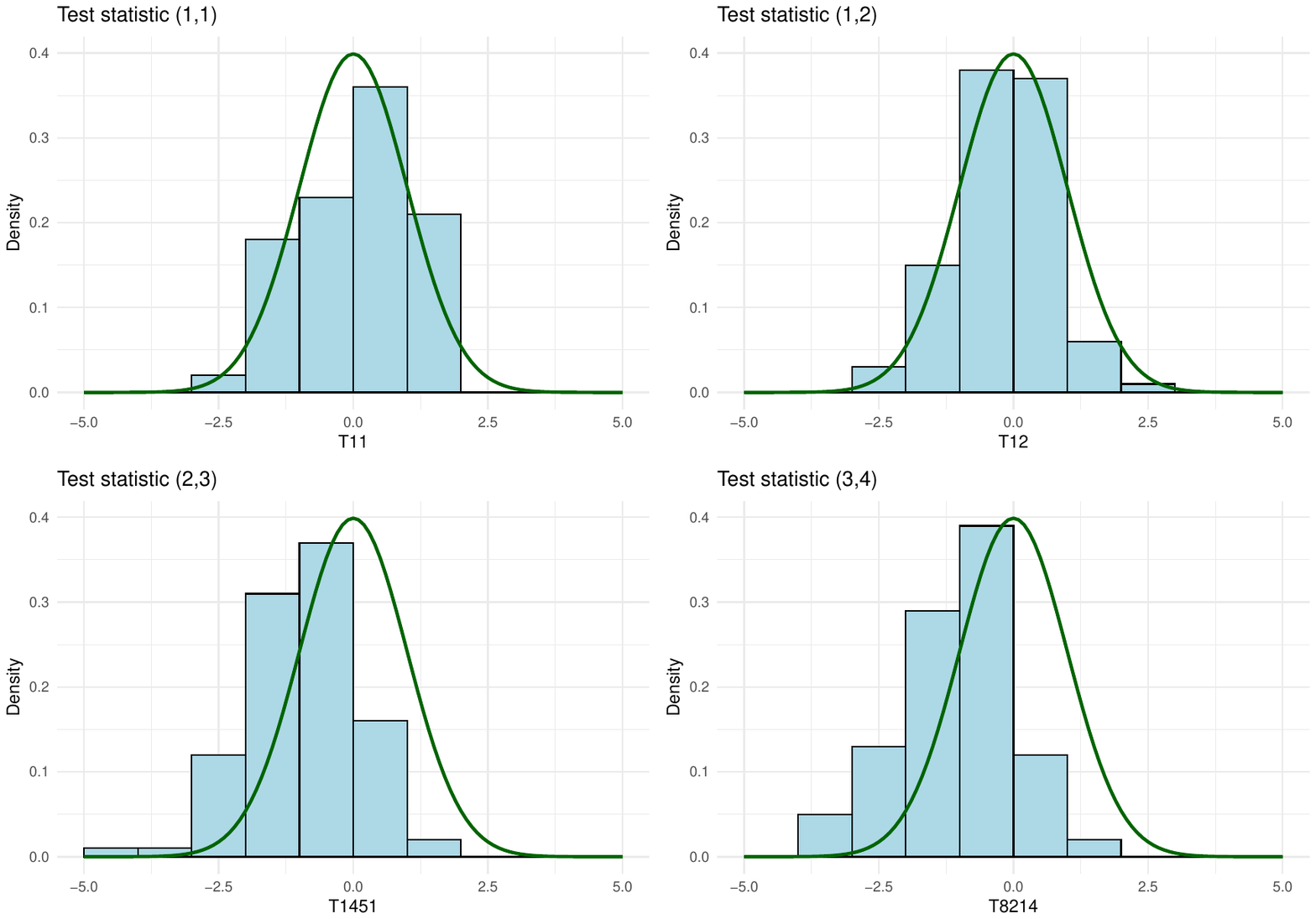}
\label{Figure 15}
\end{figure}

We now empirically demonstrate the asymptotic normality of the test statistic $T_{ij}$.

\begin{figure}[H]
\caption{Asymptotic normality of the test statistic $T_{ij}$ for Dimension
50 (Chain graphs and Star graphs)}

\centering{}\includegraphics[scale=0.26]{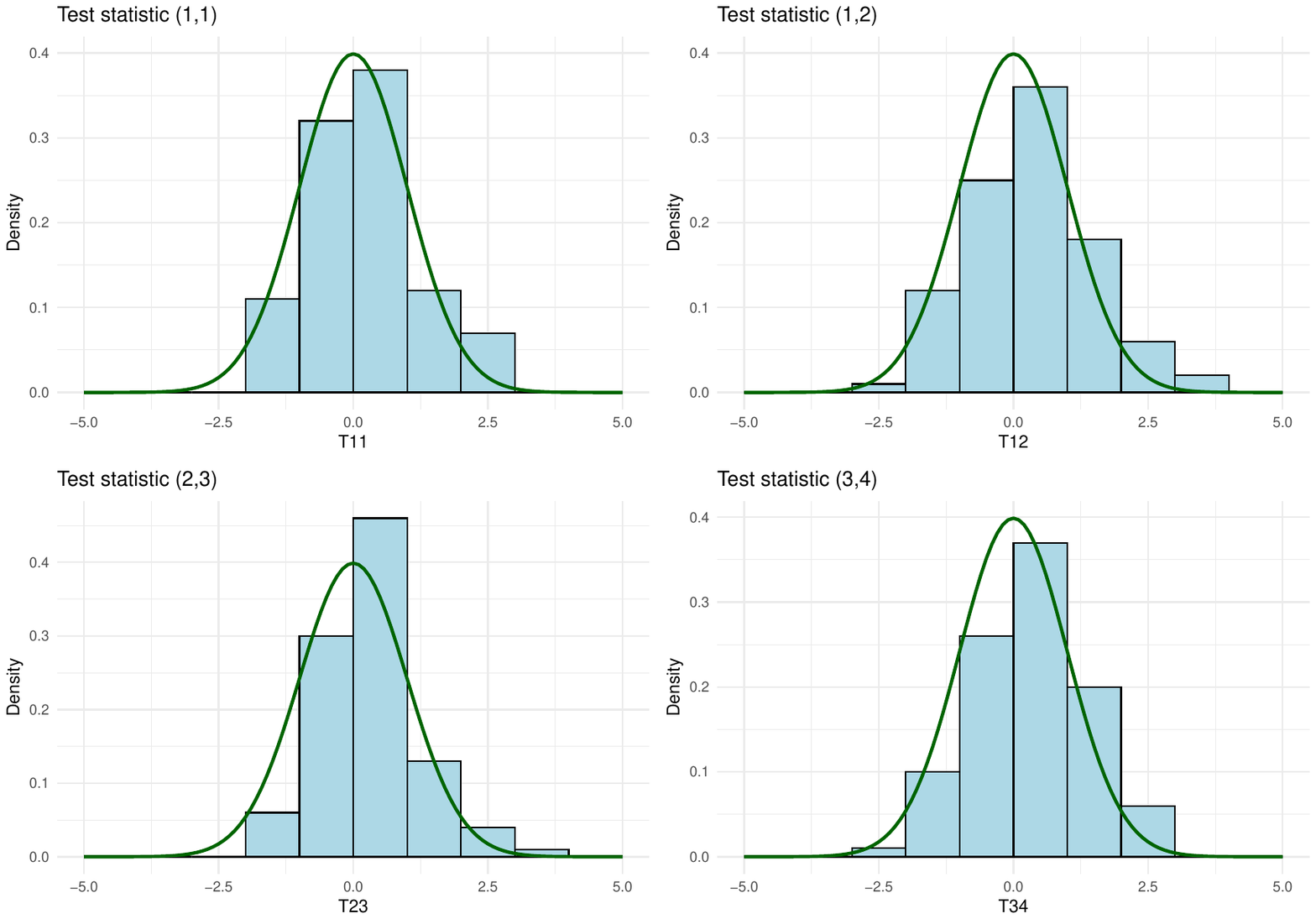}\includegraphics[scale=0.26]{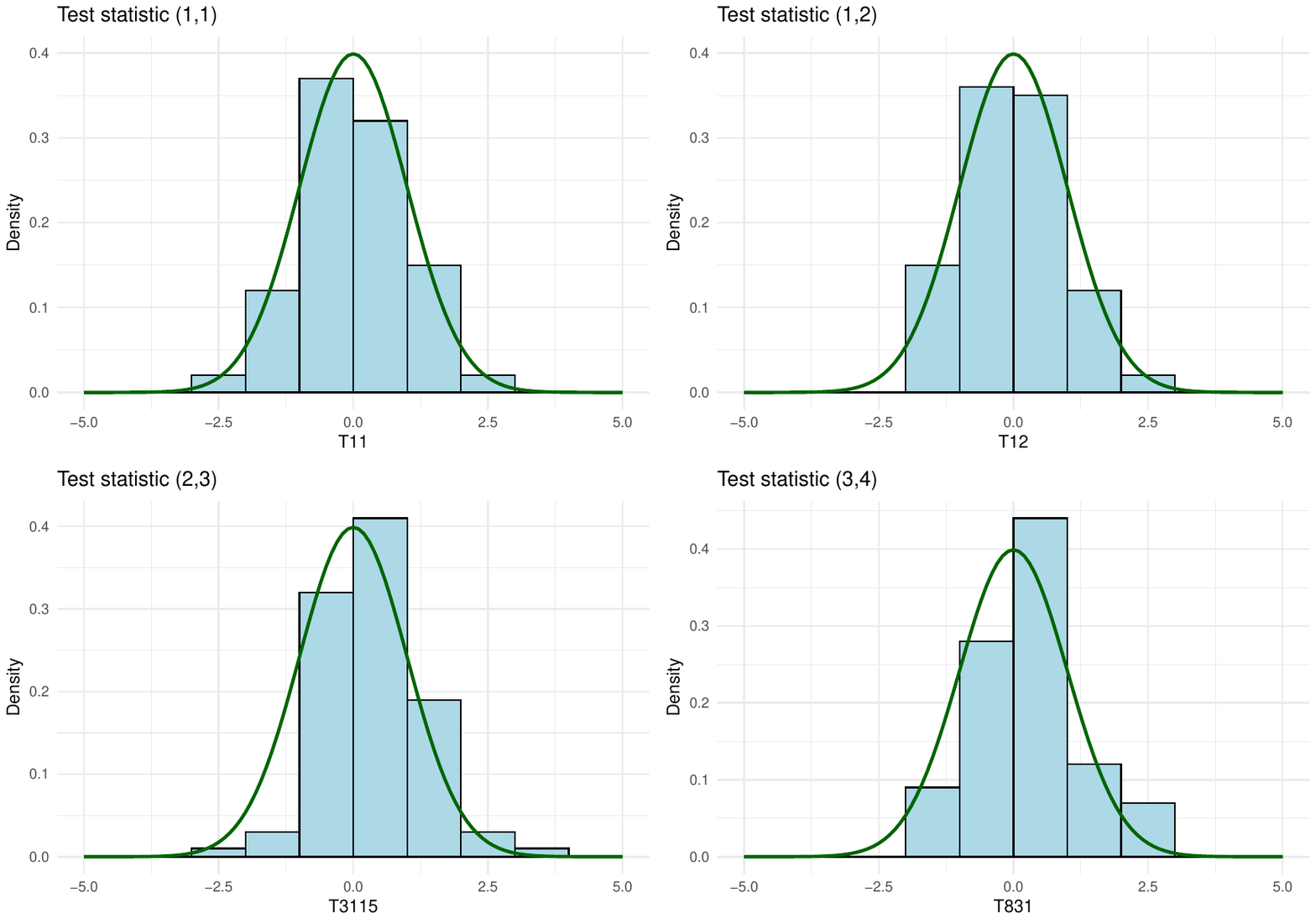}
\label{Figure 16}
\end{figure}

\begin{figure}[H]
\caption{Asymptotic normality of the test statistic $T_{ij}$ for Dimension
75 (Chain graphs and Star graphs)}

\centering{}\includegraphics[scale=0.26]{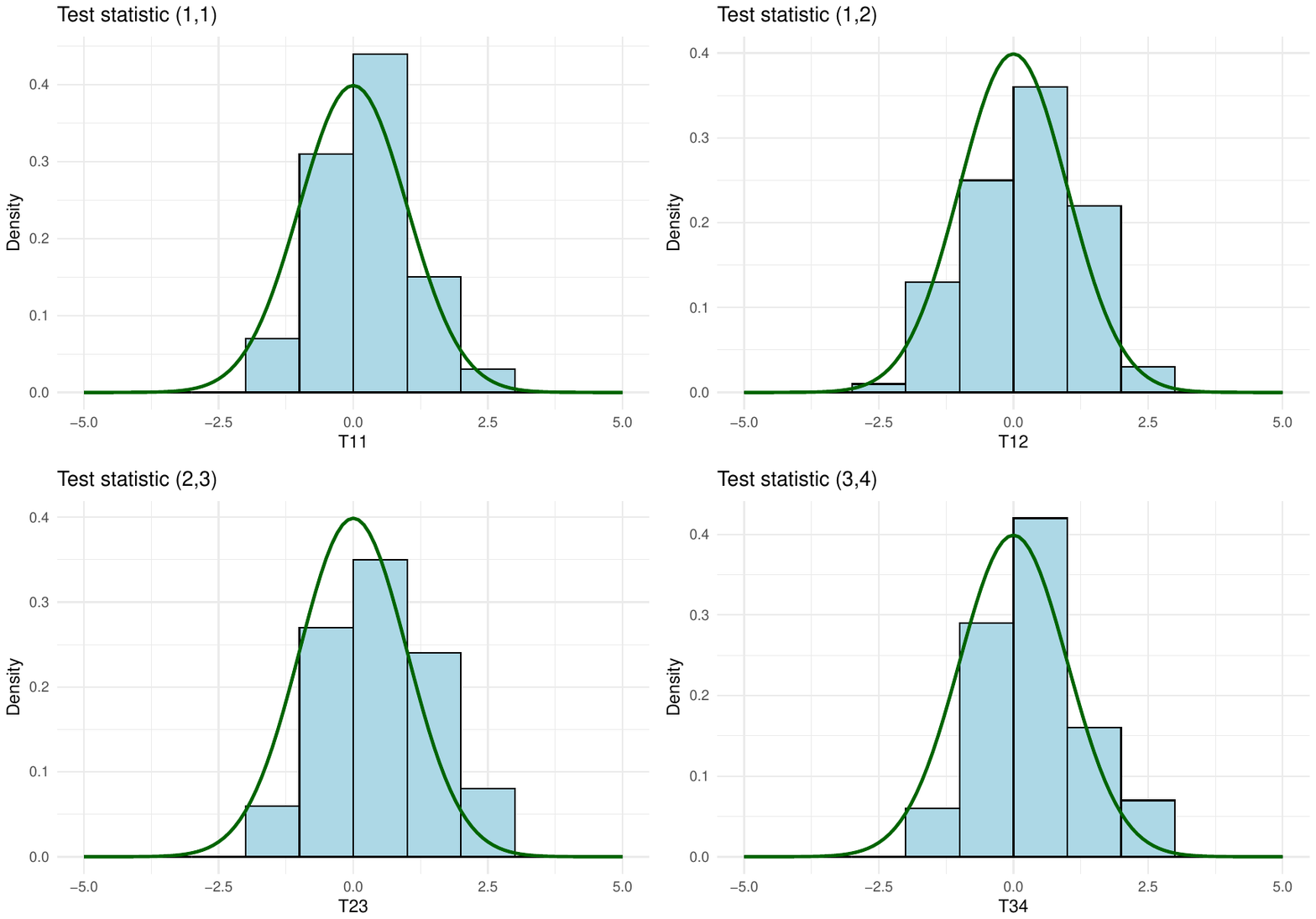}\includegraphics[scale=0.26]{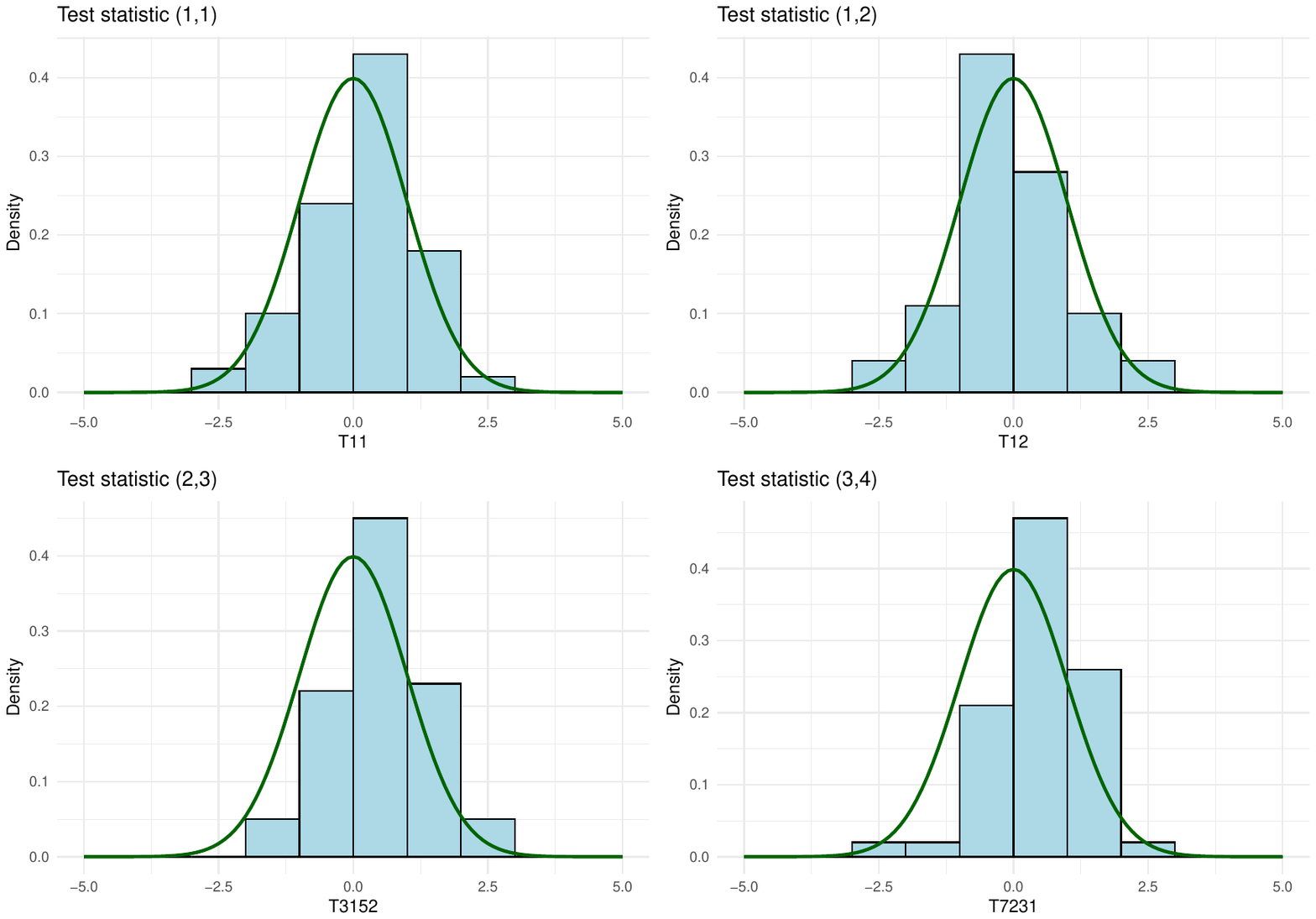}
\label{Figure 17}
\end{figure}

\begin{figure}[H]
\caption{Asymptotic normality of the test statistic $T_{ij}$ for Dimension
100 (Chain graphs and Star graphs)}

\centering{}\includegraphics[scale=0.26]{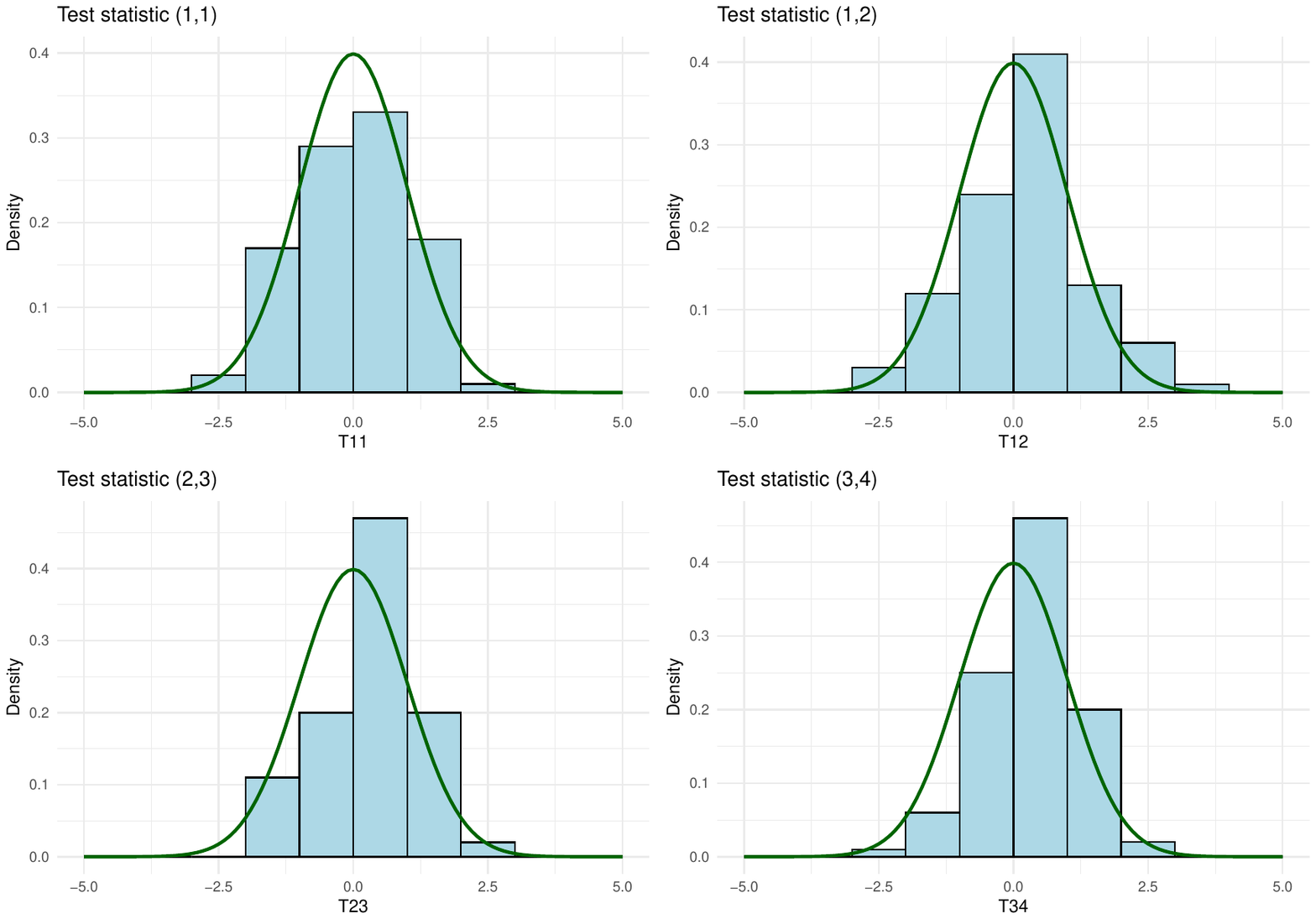}\includegraphics[scale=0.26]{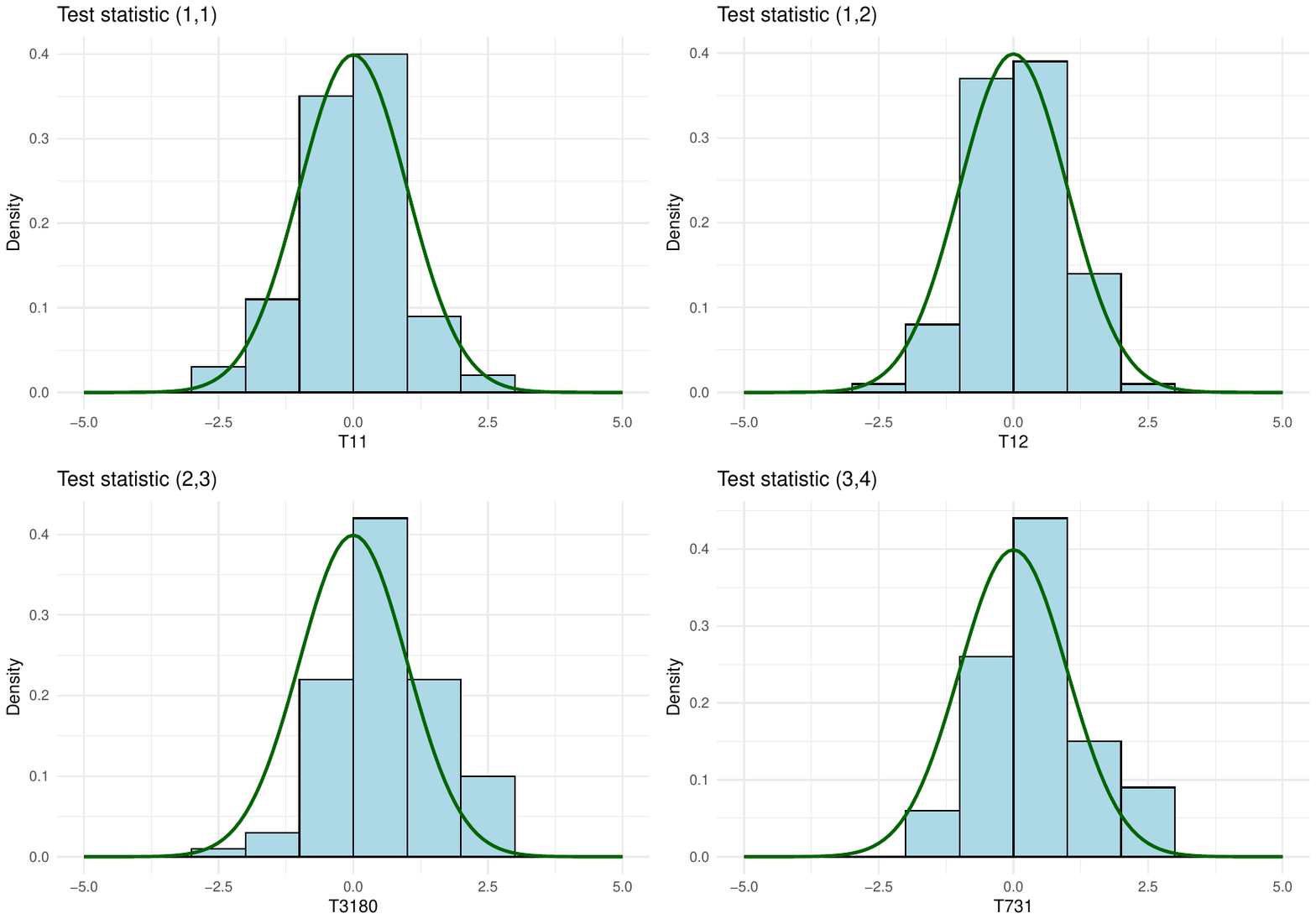}
\label{Figure 18}
\end{figure}

\begin{figure}[H]
\caption{Asymptotic normality of the test statistic $T_{ij}$ for Dimension
150 (Chain graphs and Star graphs)}

\centering{}\includegraphics[scale=0.26]{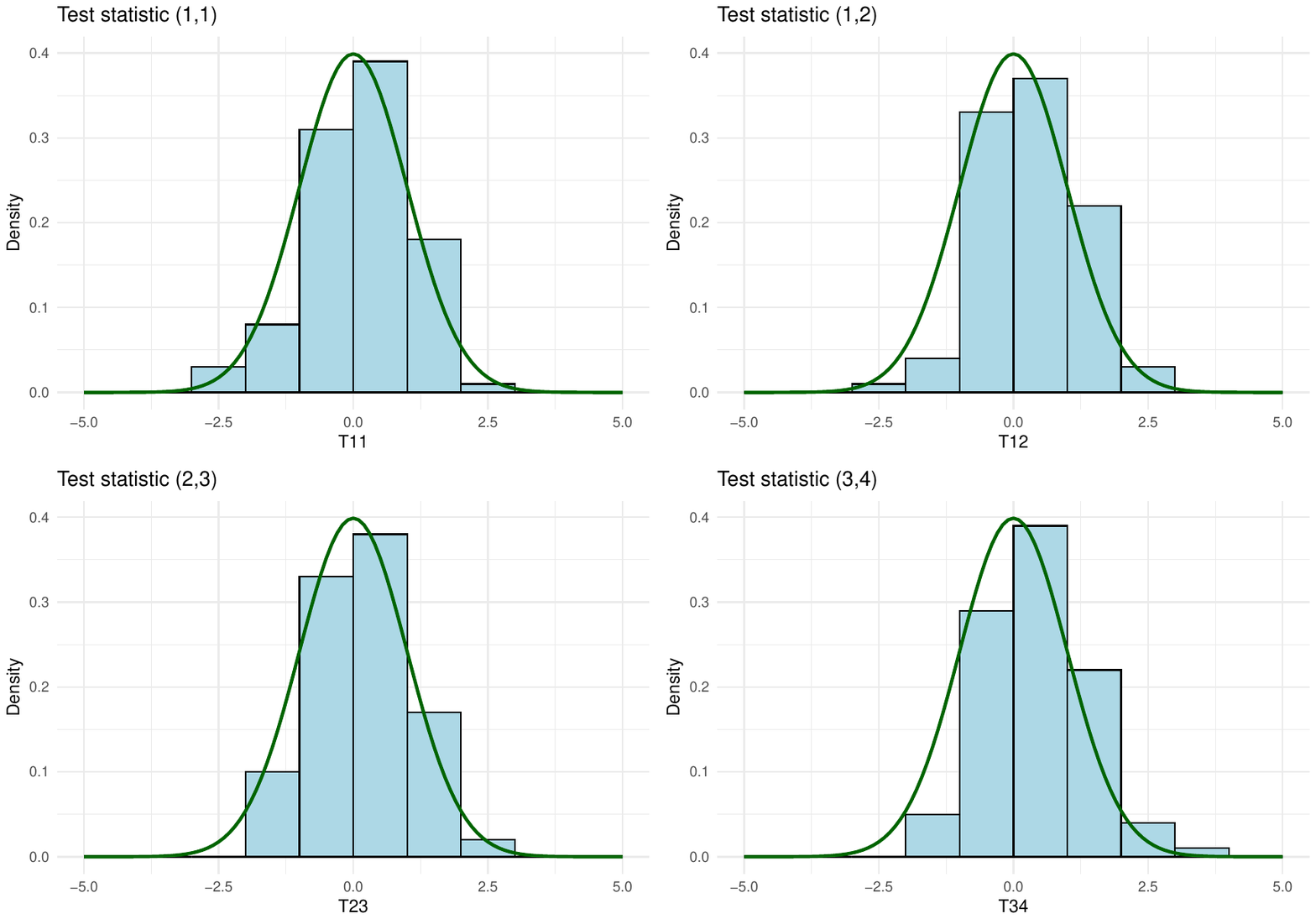}\includegraphics[scale=0.26]{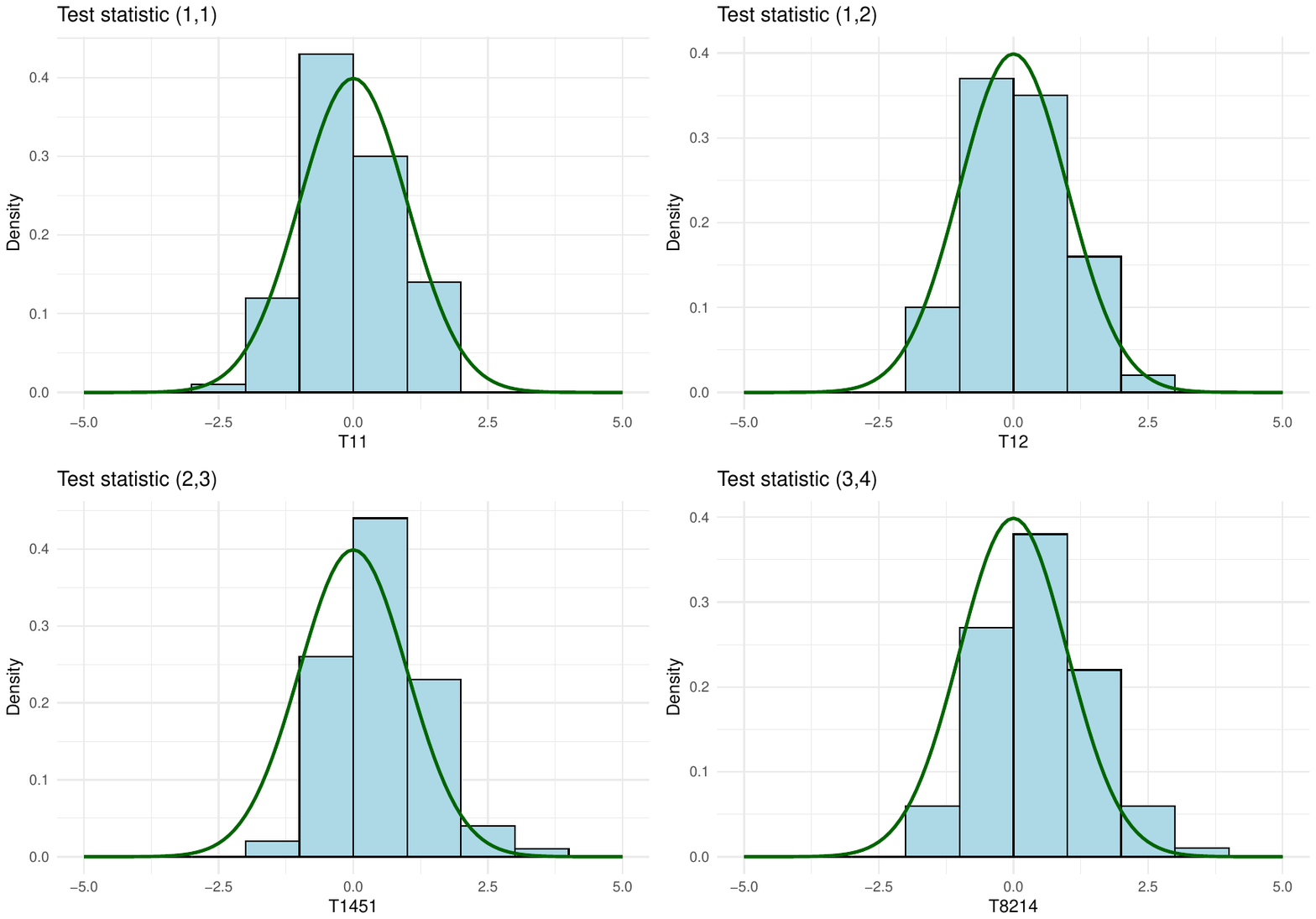}
\label{Figure 19}
\end{figure}
\subsubsection{Average coverage of the true parameter and average length of the
confidence interval}
For a given $(i,j)\in V\times V$, we have the $100(1-\alpha)\%$
confidence interval for $\Omega_{0,ij}^{k}$ is 
\[
I_{ij}^{k}=\left[\hat{\Omega}_{d,ij}^{k}-\frac{\hat{\sigma}_{ij}^{k}}{n_{k}}\tau_{\alpha/2},\hat{\Omega}_{d,ij}^{k}+\frac{\hat{\sigma}_{ij}^{k}}{n_{k}}\tau_{\alpha/2}\right],
\]
where $\tau_{\alpha/2}$ is the $(1-\alpha/2)$ upper quantile of
a standard normal distribution, for $k=1,2$. We define $\hat{\alpha}_{ij}^{k}$
as the proportion of the confidence intervals containing the true
parameter, i.e., 
\[
\alpha_{ij}^{k}=\frac{1}{B}\sum_{b=1}^{B}\mathbb{I}\left(\Omega_{0,ij}^{k}\in\left[\hat{\Omega}_{d,ij,b}^{k}-\frac{\hat{\sigma}_{ij,b}^{k}}{n_{k}}\tau_{\alpha/2},\hat{\Omega}_{d,ij,b}^{k}+\frac{\hat{\sigma}_{ij,b}^{k}}{n_{k}}\tau_{\alpha/2}\right]\right),
\]
where $\hat{\Omega}_{d,ij,b}^{k}$ is the $(i,j)^{th}$ element of
the debiased group graphical lasso estimator for population $k=1,2$
for the $b^{th}$ sample, \textbf{$b=1,2,\ldots,B$}. Finally, we define
the average coverage for population $k$ over any set $A\subset V\times V$
as 
\[
\Gamma_{A}^{k}=\frac{1}{|A|}\sum_{(i,j)\in A}\alpha_{ij}^{k}.
\]
Thus, for set $S$ and $S^{c}$, we have the average coverage for
each population $k$ as $\Gamma_{S}^{k}$ and $\Gamma_{S^{c}}^{k}$,
respectively. We also define $l_{ij}^{k}$ as the average length of
the confidence interval for the $(i,j)^{th}$ element, and hence,
define the average length over any set $A$ as 
\[
L_{A}^{k}=\frac{1}{|A|}\sum_{(i,j)\in A}l_{ij}^{k}.
\]
Hence, we also have the average length over $S$ and $S^{c}$ for
population $k=1,2$. The result is given in Table \ref{Table 3}.

\begin{table}
\caption{Average coverage and average length}

\centering{}%
\begin{tabular}{|c|c|c|c|c|c|c|}
\hline 
\multirow{2}{*}{} & \multirow{2}{*}{Graph} & \multirow{2}{*}{Dimension} & \multicolumn{2}{c|}{$S$} & \multicolumn{2}{c|}{$S^{c}$}\tabularnewline
\cline{4-7} \cline{5-7} \cline{6-7} \cline{7-7} 
 &  &  & $k=1$ & $k=2$ & $k=1$ & $k=2$\tabularnewline
\hline 
\multirow{8}{*}{Average coverage} & \multirow{4}{*}{Chain graph} & $p=50$ & 0.9539 & 0.9029 & 0.9584 & 0.9758\tabularnewline
\cline{3-7} \cline{4-7} \cline{5-7} \cline{6-7} \cline{7-7} 
 &  & $p=75$ & 0.9484 & 0.8916 & 0.9587 & 0.9761\tabularnewline
\cline{3-7} \cline{4-7} \cline{5-7} \cline{6-7} \cline{7-7} 
 &  & $p=100$ & 0.9515 & 0.8896 & 0.9597 & 0.9771\tabularnewline
\cline{3-7} \cline{4-7} \cline{5-7} \cline{6-7} \cline{7-7} 
 &  & $p=150$ & 0.9496 & 0.8843 & 0.9593 & 0.9774\tabularnewline
\cline{2-7} \cline{3-7} \cline{4-7} \cline{5-7} \cline{6-7} \cline{7-7} 
 & \multirow{4}{*}{Star graph} & $p=50$ & 0.9392 & 0.9040 & 0.9525 & 0.9549\tabularnewline
\cline{3-7} \cline{4-7} \cline{5-7} \cline{6-7} \cline{7-7} 
 &  & $p=75$ & 0.9384 & 0.8656 & 0.9513 & 0.9535\tabularnewline
\cline{3-7} \cline{4-7} \cline{5-7} \cline{6-7} \cline{7-7} 
 &  & $p=100$ & 0.9352 & 0.8524 & 0.9515 & 0.9532\tabularnewline
\cline{3-7} \cline{4-7} \cline{5-7} \cline{6-7} \cline{7-7} 
 &  & $p=150$ & 0.942 & 0.8332 & 0.9508 & 0.952\tabularnewline
\hline 
\multirow{8}{*}{Average length} & \multirow{4}{*}{Chain graph} & $p=50$ & 0.1423 & 0.1076 & 0.1444 & 0.1162\tabularnewline
\cline{3-7} \cline{4-7} \cline{5-7} \cline{6-7} \cline{7-7} 
 &  & $p=75$ & 0.142 & 0.1071 & 0.144 & 0.1153\tabularnewline
\cline{3-7} \cline{4-7} \cline{5-7} \cline{6-7} \cline{7-7} 
 &  & $p=100$ & 0.1416 & 0.1062 & 0.1436 & 0.1142\tabularnewline
\cline{3-7} \cline{4-7} \cline{5-7} \cline{6-7} \cline{7-7} 
 &  & $p=150$ & 0.1414 & 0.1051 & 0.1433 & 0.1126\tabularnewline
\cline{2-7} \cline{3-7} \cline{4-7} \cline{5-7} \cline{6-7} \cline{7-7} 
 & \multirow{4}{*}{Star graph} & $p=50$ & 0.4493 & 0.6361 & 0.6256 & 0.9473\tabularnewline
\cline{3-7} \cline{4-7} \cline{5-7} \cline{6-7} \cline{7-7} 
 &  & $p=75$ & 0.4395 & 0.609 & 0.6328 & 0.9632\tabularnewline
\cline{3-7} \cline{4-7} \cline{5-7} \cline{6-7} \cline{7-7} 
 &  & $p=100$ & 0.4362 & 0.6354 & 0.6354 & 0.972\tabularnewline
\cline{3-7} \cline{4-7} \cline{5-7} \cline{6-7} \cline{7-7} 
 &  & $p=150$ & 0.4216 & 0.5778 & 0.6386 & 0.9848\tabularnewline
\hline 
\end{tabular}
\label{Table 3}
\end{table}
\subsection{Real-world dataset}\label{Subsection 5.2}
\subsubsection{Climatic data}
The main motivation for using group graphical lasso is to be used in estimating the precision matrices of multiple populations when there is a common pattern in the sparsity of the precision matrices of the populations. Such cases in real life can occur when we have a dataset corresponding to variables across different adjacent districts. As an application of this method, we work with a dataset where we have observations on various climatic and greenhouse gas measurements over two different cities, namely Oklahoma and Colorado. These cities lie in the same climatic belt and hence can be expected to have homogeneity in the sparsity pattern of the precision matrices. We gather information on 16 variables, namely Carbon monoxide (CO), Carbon dioxide (CO\textsubscript{2}), Methane (CH\textsubscript{4}), Nitrous Oxide (N\textsubscript{2}O), Hydrogen (H\textsubscript{2}), and Sulfur hexafluoride (SF\textsubscript{6}), along with climatic variables such as maximum temperature (MAT in $^{\circ}$F), average temperature (AVT in $^{\circ}$F), minimum temperature (MIT in $^{\circ}$F), dew point (DEP in $^{\circ}$F), precipitation (PER in inches), wind speed (WIS in mph), gust wind speed (GWS in mph), sea level pressure (SLP in inches), snowfall (SNF in inches), and the number of rainfall days (RFD). The data for greenhouse gases are gathered from \href{https://gml.noaa.gov/data/data.php}{Global Monitoring Laboratory}. The data for PER, SNF, and RFD are obtained from \href{https://www.weather.gov/wrh/Climate?wfo=oun\%20}{National Weather Service}, while the data for the rest of the variables are obtained from \href{https://www.wunderground.com/}{Weather Underground}.

We gather data for 120 months from January (2010) to December (2019). The main reason why we did not take any more data beyond December 2019 is due to the arrival of the global lockdown for COVID-19. During the lockdown, there were changes in the climatic variations due to a decrease in pollution, and hence, they were not suitable for analysis in this context. To remove the seasonal effects and make the observations independent, we take 3 3-month averages and then take first-order differencing. Then we normalize the data. Hence, our dataset has 39 observations over 16 variables from two different populations, Oklahoma and Colorado.

We do not know the real precision matrices of the populations. However, we used the whole data and estimated the precision matrices of the populations to be used as a proxy for the real precision matrices. We use the group graphical lasso parameters to be proportional to $\sqrt{\frac{log\ p}{n}}$. The common network for both populations is given in Figure \ref{Figure 20}.
\begin{figure}[t]
\caption{Predicted common network}

\begin{centering}
\includegraphics[scale=0.4]{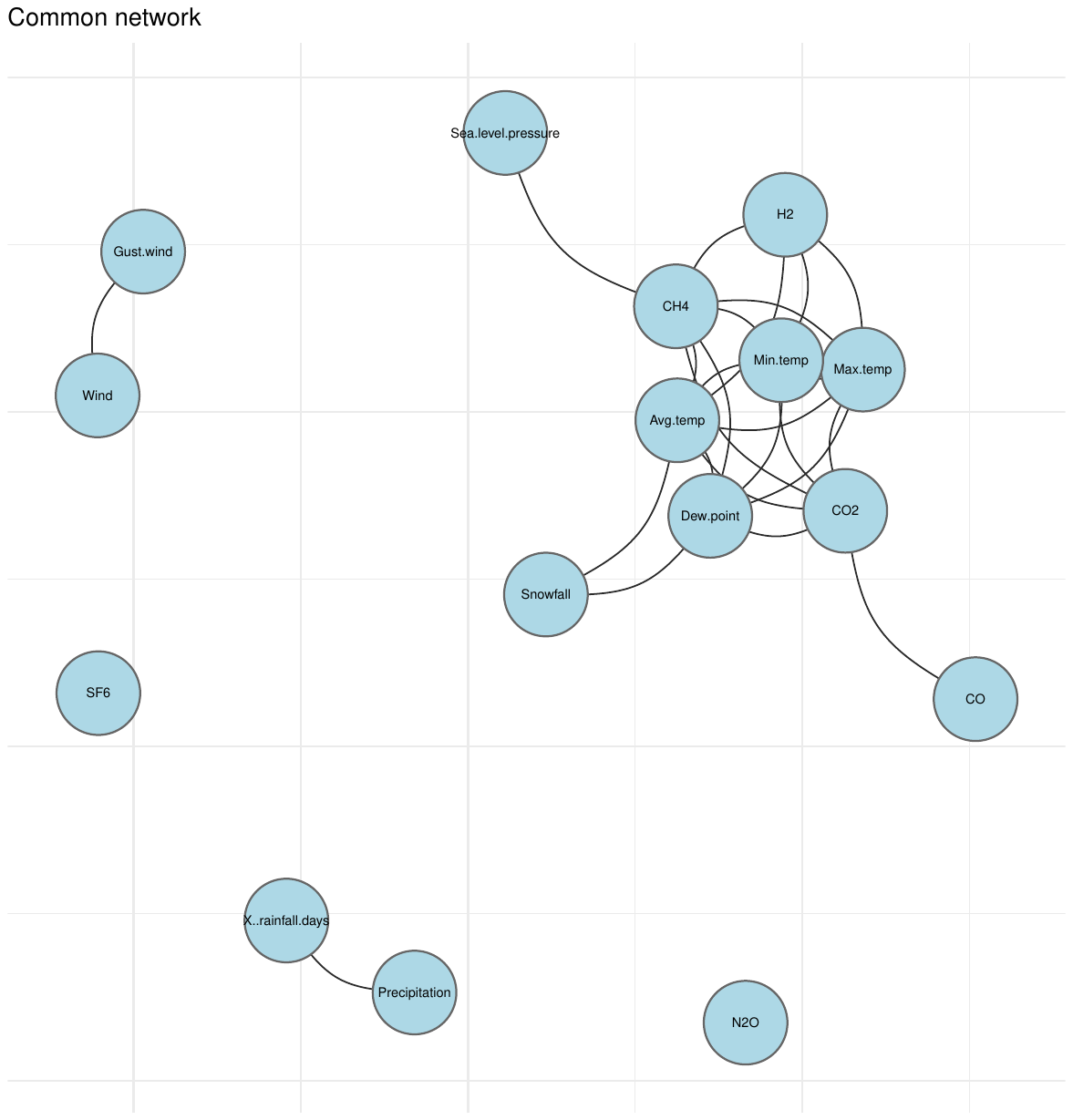}
\par\end{centering}
\label{Figure 20}
\end{figure}

We shall mainly focus on the performance of the estimates over the edges in the common network. We shall analyse the performance of the proportion of the common edges that detects the correct sign in both populations, the average supremum norm distance of the predicted matrices, and the power of the test statistic, if we had done a hypothesis test over some of the edges of the networks. To do this, we first select $N=16,20,25,36$ samples from both populations. Then, for a particular value of $N$, we randomly select $N$ observations from each population without replacement and obtain the precision matrix estimates for each population. Then we check the proportion of the common edges from each precision matrix of the sampled setup that matches the same sign as the proxy predicted matrices, considering the whole population. We repeat this process for $B=100$
times and report the average of this proportion. We also report the average of the supremum norm distance of the sampled precision matrices and the proxy precision matrices. Finally, we report the empirical power of the test at the edges {(Precipitation- Sea level pressure), (Gust wind- Snowfall)}. Also, the power of the test $H_{0}:\Omega_{0,ij}^{Oklahoma}-\Omega_{0,ij}^{Colorado}=0$ increases as the sample size increases. The detailed result is reported in Table \ref{Table 4}. From Table \ref{Table 4} we can see, as the sample size increases, the proportion of the common edges that detects the correct sign increases and the average of the sup-norm distance decreases. Also, the power of the tests of the difference is edges increases as the sample size increases.

\begin{table}
\caption{Analysis of climatic data}

\centering{}%
\begin{tabular}{|c|c|c|c|c|c|}
\hline 
\multicolumn{2}{|c|}{} & $N=16$ & $N=20$ & $N=25$ & $N=33$\tabularnewline
\hline 
\hline 
\multicolumn{2}{|c|}{Correctly signed estimated edges (in percent)} & 3.32 & 15.68 & 35.92 & 70.56\tabularnewline
\hline 
\multirow{2}{*}{Average of sup-norm distance} & Oklahoma & 1.4207 & 0.7753 & 0.5134 & 0.2686\tabularnewline
\cline{2-6} \cline{3-6} \cline{4-6} \cline{5-6} \cline{6-6} 
 & Colorado & 0.7144 & 0.4716 & 0.3344 & 0.1723\tabularnewline
\hline 
\multirow{4}{*}{Power of the test} & Precip- & \multirow{2}{*}{0.89} & \multirow{2}{*}{0.96} & \multirow{2}{*}{1.00} & \multirow{2}{*}{1.00}\tabularnewline
 & SLP &  &  &  & \tabularnewline
\cline{2-6} \cline{3-6} \cline{4-6} \cline{5-6} \cline{6-6} 
 & Gust wind- & \multirow{2}{*}{0.83} & \multirow{2}{*}{0.91} & \multirow{2}{*}{1.00} & \multirow{2}{*}{1.00}\tabularnewline
 & Snowfall &  &  &  & \tabularnewline
\hline 
\end{tabular}
\label{Table 4}
\end{table}

\subsubsection{WebKb dataset}
The dataset contains webpages collected from computer science departments
at Cornell, the University of Wisconsin, the University of Texas, and the University
of Washington. The data has been pre-processed by removing short words
that are less than 3 characters long, and then stopwords were removed.
Further \href{https://tartarus.org/martin/PorterStemmer/}{Porter stemming algorithm} was applied to the remaining words. The dataset and details of pre-processing can be found \href{https://ana.cachopo.org/datasets-for-single-label-text-categorization}{here}. For more details regarding pre-processing of the data, we refer to \cite{2007:phd-Ana-Cardoso-Cachopo}. Originally, the dataset had 7 classes. The classes are student, faculty, staff, department, course, project, and other, having 1641, 1124, 137, 182, 930, 504, and 3764 webpages respectively. Here, we select 2 classes having the most data,i.e., we shall consider student and faculty as our populations. Moreover, we take a subset of the whole data for our analysis. Our data has 1396 webpages (or observations), where the classes student, faculty, course, and project have 544 and 374 observations, respectively. We first obtain the frequency of $j^{th}$ term in the $i^{th}$ webpage, $f_{ij}$,$i=1,2,\ldots,n;j=1,2,\ldots,p$. We then obtain the whole data matrix as $X=(x_{ij})$, where $x_{ij}=e_{j}log(1+f_{ij})$, $e_{j}=1+\sum_{i=1}^{n}p_{ij}log(p_{ij})/log(n)$ is the log-entropy of the $j^{th}$ term and $p_{ij}=f_{ij}/\sum_{i=1}^{n}f_{ij}$ (\cite{dumais1991improving}). This dataset was previously used in \cite{guo2011joint}, where they studied the structure of the networks considering 100 terms with the highest log-entropy weights. We select a subset of 50 terms from this set of 100 terms and carry out our analysis based on them.

The common network of the classes is as in Figure
\ref{Figure 21}.  We studied the proportion of the
common edges that detects the correct sign in both populations,
the average supremum norm distance of the predicted matrices, and
the power of the test statistic over these common edges. Here, we have
$N_{1}=544$ and $N_{2}=374$. We select samples of $n_{i}=\{N_{i}/2,2N_{i}/3,3N_{i}/4,4N_{i}/5,5N_{i}/6,9N_{i}/10\},i=1,2$
without replacement and repeat this for $B=100$ times. As previously,
we shall see that the proportion of the common edges that detects
the correct sign increases, and the average of the sup-norm distance decreases
as the sample size increases. We perform the hypothesis testing of
equality of precision matrix entries at the edge pairs of (Perform-Support),
(Comput-Science) and (Model-Process). It is expected to have disparity
among these edges for the two classes. The power of the test corresponding
to the edges shall increase as the sample size increases. The entire analysis is summarised in Table \ref{Table 5}.
\begin{figure}
\caption{Predicted common network}

\begin{centering}
\includegraphics[scale=0.4]{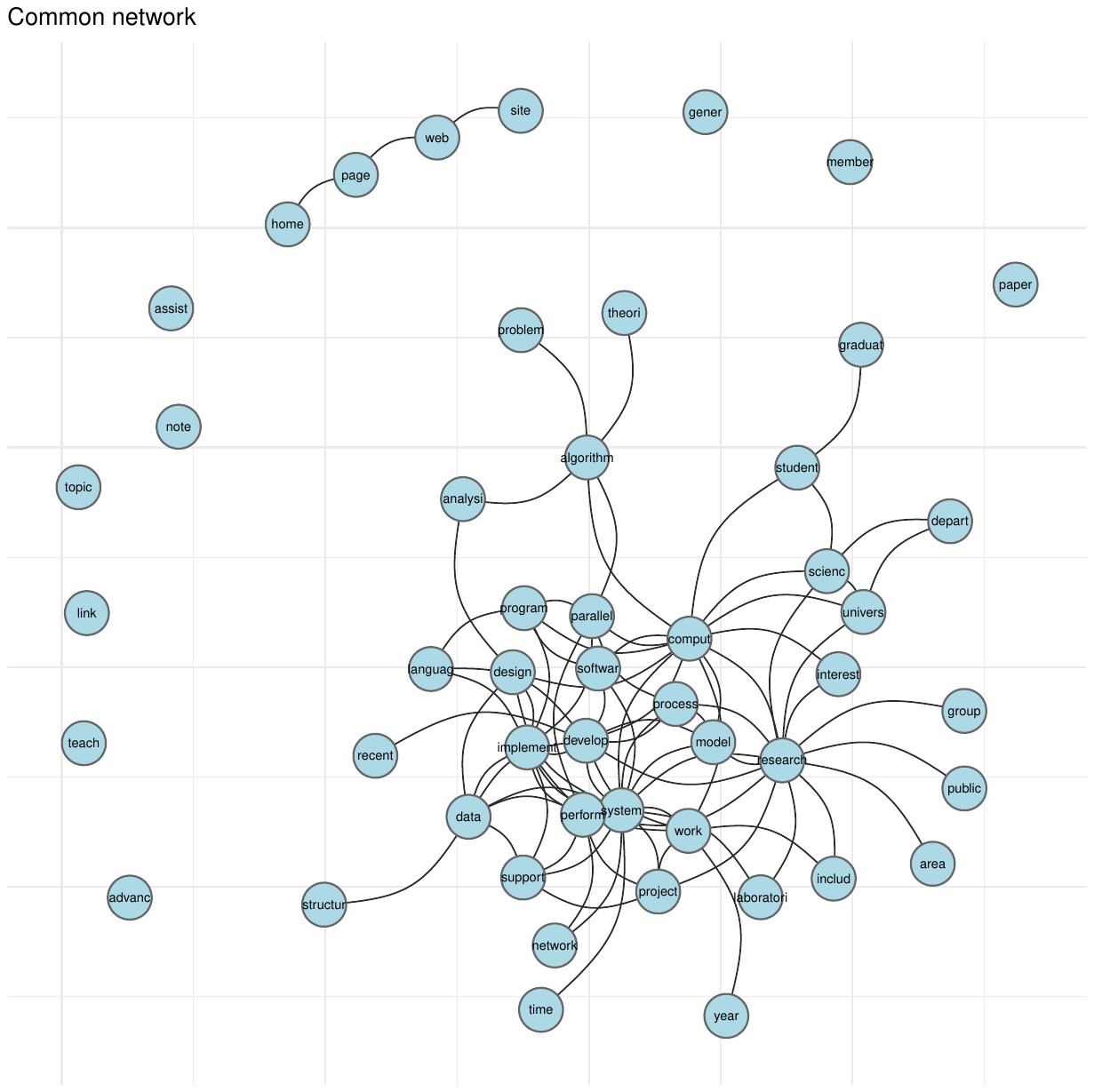}
\par\end{centering}
\label{Figure 21}
\end{figure}
\begin{center}
\begin{table}
\caption{Analysis of WebKb data}

\centering{}%
\begin{tabular}{|c|c|c|c|c|}
\hline 
\multicolumn{2}{|c|}{} & $n_{i}=N_{i}/2$ & $n_{i}=2N_{i}/3$ & $n_{i}=3N_{i}/4$\tabularnewline
\cline{3-5} \cline{4-5} \cline{5-5} 
\multicolumn{2}{|c|}{Correctly signed estimated edges (in percent)} & 25.9 & 42.28 & 55.37\tabularnewline
\cline{3-5} \cline{4-5} \cline{5-5} 
\multicolumn{2}{|c|}{} & $n_{i}=4N_{i}/5$ & $n_{i}=5N_{i}/6$ & $n_{i}=9N_{i}/10$\tabularnewline
\cline{3-5} \cline{4-5} \cline{5-5} 
\multicolumn{2}{|c|}{} & 67.6 & 69.15 & 85.38\tabularnewline
\hline 
\multirow{6}{*}{Average of sup-norm distance} &  & $n_{i}=N_{i}/2$ & $n_{i}=2N_{i}/3$ & $n_{i}=3N_{i}/4$\tabularnewline
\cline{2-5} \cline{3-5} \cline{4-5} \cline{5-5} 
 & Stu & 0.495 & 0.3377 & 0.2492\tabularnewline
\cline{2-5} \cline{3-5} \cline{4-5} \cline{5-5} 
 & Fac & 0.5768 & 0.3179 & 0.2439\tabularnewline
\cline{2-5} \cline{3-5} \cline{4-5} \cline{5-5} 
 &  & $n_{i}=4N_{i}/5$ & $n_{i}=5N_{i}/6$ & $n_{i}=9N_{i}/10$\tabularnewline
\cline{2-5} \cline{3-5} \cline{4-5} \cline{5-5} 
 & Stu & 0.2162 & 0.1923 & 0.1596\tabularnewline
\cline{2-5} \cline{3-5} \cline{4-5} \cline{5-5} 
 & Fac & 0.2159 & 0.1755 & 0.1424\tabularnewline
\hline 
\multirow{18}{*}{Power of the test} &  & $n_{i}=N_{i}/2$ & $n_{i}=2N_{i}/3$ & $n_{i}=3N_{i}/4$\tabularnewline
\cline{2-5} \cline{3-5} \cline{4-5} \cline{5-5} 
 & Perf- & \multirow{2}{*}{0.99} & \multirow{2}{*}{1} & \multirow{2}{*}{1}\tabularnewline
 & Supp &  &  & \tabularnewline
\cline{2-5} \cline{3-5} \cline{4-5} \cline{5-5} 
 &  & $n_{i}=4N_{i}/5$ & $n_{i}=5N_{i}/6$ & $n_{i}=9N_{i}/10$\tabularnewline
\cline{2-5} \cline{3-5} \cline{4-5} \cline{5-5} 
 & Perf- & \multirow{2}{*}{1} & \multirow{2}{*}{1} & \multirow{2}{*}{1}\tabularnewline
 & Supp &  &  & \tabularnewline
\cline{2-5} \cline{3-5} \cline{4-5} \cline{5-5} 
 &  & $n_{i}=N_{i}/2$ & $n_{i}=2N_{i}/3$ & $n_{i}=3N_{i}/4$\tabularnewline
\cline{2-5} \cline{3-5} \cline{4-5} \cline{5-5} 
 & Comput- & \multirow{2}{*}{0.77} & \multirow{2}{*}{0.93} & \multirow{2}{*}{0.99}\tabularnewline
 & Science &  &  & \tabularnewline
\cline{2-5} \cline{3-5} \cline{4-5} \cline{5-5} 
 &  & $n_{i}=4N_{i}/5$ & $n_{i}=5N_{i}/6$ & $n_{i}=9N_{i}/10$\tabularnewline
\cline{2-5} \cline{3-5} \cline{4-5} \cline{5-5} 
 & Comput- & \multirow{2}{*}{1} & \multirow{2}{*}{1} & \multirow{2}{*}{1}\tabularnewline
 & Science &  &  & \tabularnewline
\cline{2-5} \cline{3-5} \cline{4-5} \cline{5-5} 
 &  & $n_{i}=N_{i}/2$ & $n_{i}=2N_{i}/3$ & $n_{i}=3N_{i}/4$\tabularnewline
\cline{2-5} \cline{3-5} \cline{4-5} \cline{5-5} 
 & Model- & \multirow{2}{*}{0.59} & \multirow{2}{*}{0.75} & \multirow{2}{*}{0.85}\tabularnewline
 & Process &  &  & \tabularnewline
\cline{2-5} \cline{3-5} \cline{4-5} \cline{5-5} 
 &  & $n_{i}=4N_{i}/5$ & $n_{i}=5N_{i}/6$ & $n_{i}=9N_{i}/10$\tabularnewline
\cline{2-5} \cline{3-5} \cline{4-5} \cline{5-5} 
 & Model- & \multirow{2}{*}{0.82} & \multirow{2}{*}{0.93} & \multirow{2}{*}{0.94}\tabularnewline
 & Process &  &  & \tabularnewline
\hline 
\end{tabular}
\label{Table 5}
\end{table}
\par\end{center}
\subsubsection{Weekly log returns across companies}

In this example, we started by choosing several businesses from
various industries. The table below provides a list of companies from
several sectors:

\begin{table}[H]
\caption{List of companies from different sectors}

\centering{}%
\begin{tabular}{|c|c|}
\hline 
Sector & Companies\tabularnewline
\hline 
\hline 
Health care & Abbott, Eli Lilly and Company, Johnson and Johnson, Novartis, Pfizer\tabularnewline
\hline 
Semiconductors & AMD, Broadcom, Intel, Nvidia, Qualcomm\tabularnewline
\hline 
\multirow{2}{*}{Financial sector} & Bank of America, Barclays, Citibank, Goldman Sachs, \tabularnewline
\cline{2-2} 
 & HSBC, JP Morgan, Morgan Stanley, Wells Fargo\tabularnewline
\hline 
Automobiles & Ford, General Motors, Tesla\tabularnewline
\hline 
Tech & Amazon, Google, Meta, Yelp\tabularnewline
\hline 
\end{tabular}
\label{Table 6}
\end{table}

For the years 2016, 2019, 2022, and 2024, we gathered the weekly stock
market prices of these firms from \href{https://finance.yahoo.com/}{Yahoo Finance}
. However, we use log returns of the prices, i.e., $log(p_{t}/p_{t-1})$,
because the prices will be dependent over the weeks.
For population 1, we choose the pre-COVID years (2016 and 2019), while
for population 2, we choose the post-COVID years (2022 and 2024).
After that, we standardize the data for both groups. Note that the
companies within sectors should have a significant impact on one another,
but not much on one another. It is anticipated that this behavior
will not change between the two populations. Therefore, we can expect
to acquire a similar sparsity pattern throughout the populations when
working with group graphical lasso. Figures \ref{Figure 22}- \ref{Figure 24} depict the common
and individual networks. The percentage of common edges that correctly
identify the sign in both populations, the average supremum norm distance,
and the test power for edge differences between populations are all
examined as before. For $B=100$, we repeat this approach with samples
of $n_{i}=\{40,60,70,80\},i=1,2$ without replacement. The detailed
result is given in Table \ref{Table 7}. As the sample size grows, the empirical power
of the Google-Meta and Google-Intel edge declines, indicating that
there are no appreciable differences between the populations. Nonetheless,
the empirical power of the edge between Johnson and Johnson-Meta and
AMD-Qualcomm rises, suggesting that the conditional relation of these
edges varies significantly over the population. This suggests that
the conditional dependencies of these pairs of companies significantly
changed due to the COVID pandemic. 

\begin{table}
\caption{Analysis of weekly log returns}

\centering{}%
\begin{tabular}{|c|c|c|c|c|c|}
\hline 
\multicolumn{2}{|c|}{} & $N=40$ & $N=60$ & $N=70$ & $N=80$\tabularnewline
\hline 
\hline 
\multicolumn{2}{|c|}{Correctly signed estimated edges (in percent)} & 32.55 & 55.21 & 65.86 & 76.27\tabularnewline
\hline 
\multirow{2}{*}{Average of sup-norm distance} & \multicolumn{1}{c|}{Pre covid} & 1.0517 & 0.6343 & 0.4908 & 0.3916\tabularnewline
\cline{2-6} \cline{3-6} \cline{4-6} \cline{5-6} \cline{6-6} 
 & \multicolumn{1}{c|}{Post covid} & 0.8195 & 0.497 & 0.3866 & 0.3125\tabularnewline
\hline 
\multirow{4}{*}{Power of the test} & AMD- Qualcomm & 0.29 & 0.46 & 0.64 & 0.65\tabularnewline
\cline{2-6} \cline{3-6} \cline{4-6} \cline{5-6} \cline{6-6} 
 & Google-META & 0.11 & 0.15 & 0.11 & 0.08\tabularnewline
\cline{2-6} \cline{3-6} \cline{4-6} \cline{5-6} \cline{6-6} 
 & Johnson and Johnson-META & 0.71 & 0.97 & 0.97 & 1.00\tabularnewline
\cline{2-6} \cline{3-6} \cline{4-6} \cline{5-6} \cline{6-6} 
 & Google-Intel & 0.1 & 0.07 & 0.08 & 0.06\tabularnewline
\hline 
\end{tabular}
\label{Table 7}
\end{table}
\begin{figure}[H]
\caption{Common network}

\begin{centering}
\includegraphics[scale=0.4]{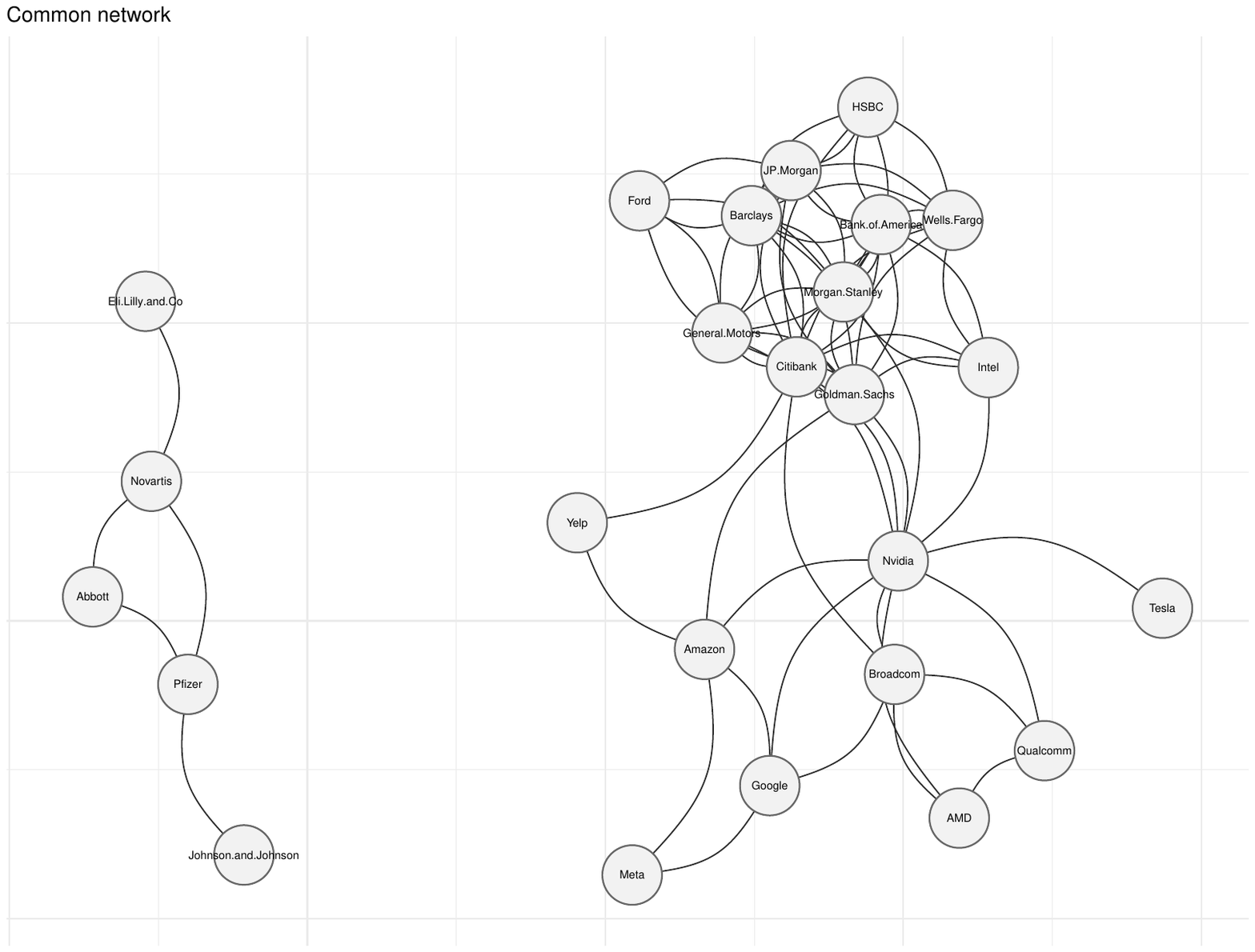}
\par\end{centering}
\label{Figure 22}
\end{figure}
\begin{figure}[H]
\caption{Predicted population 1 network}

\begin{centering}
\includegraphics[scale=0.4]{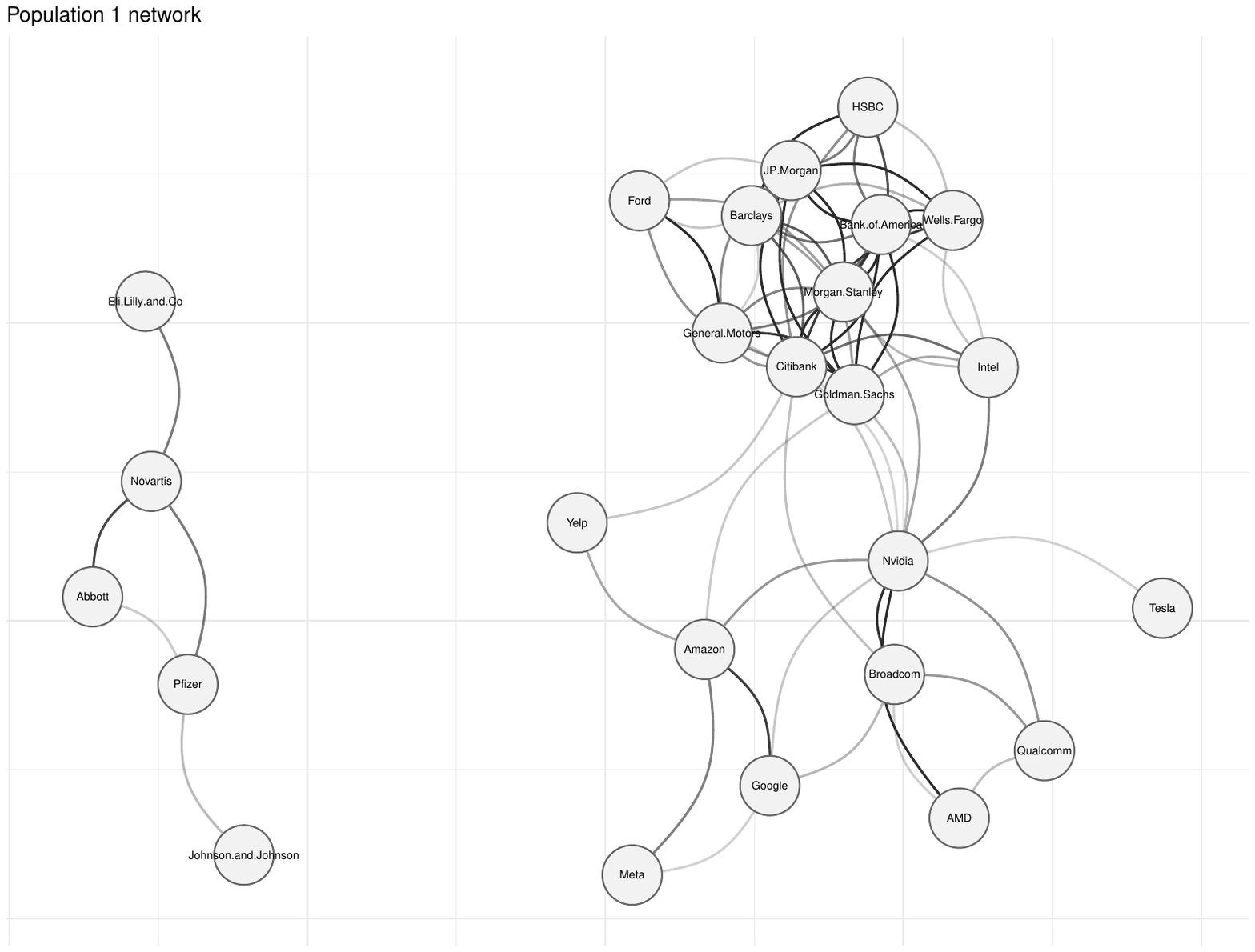}
\par\end{centering}
\label{Figure 23}
\end{figure}
\begin{figure}
\caption{Predicted population 2 network}

\begin{centering}
\includegraphics[scale=0.4]{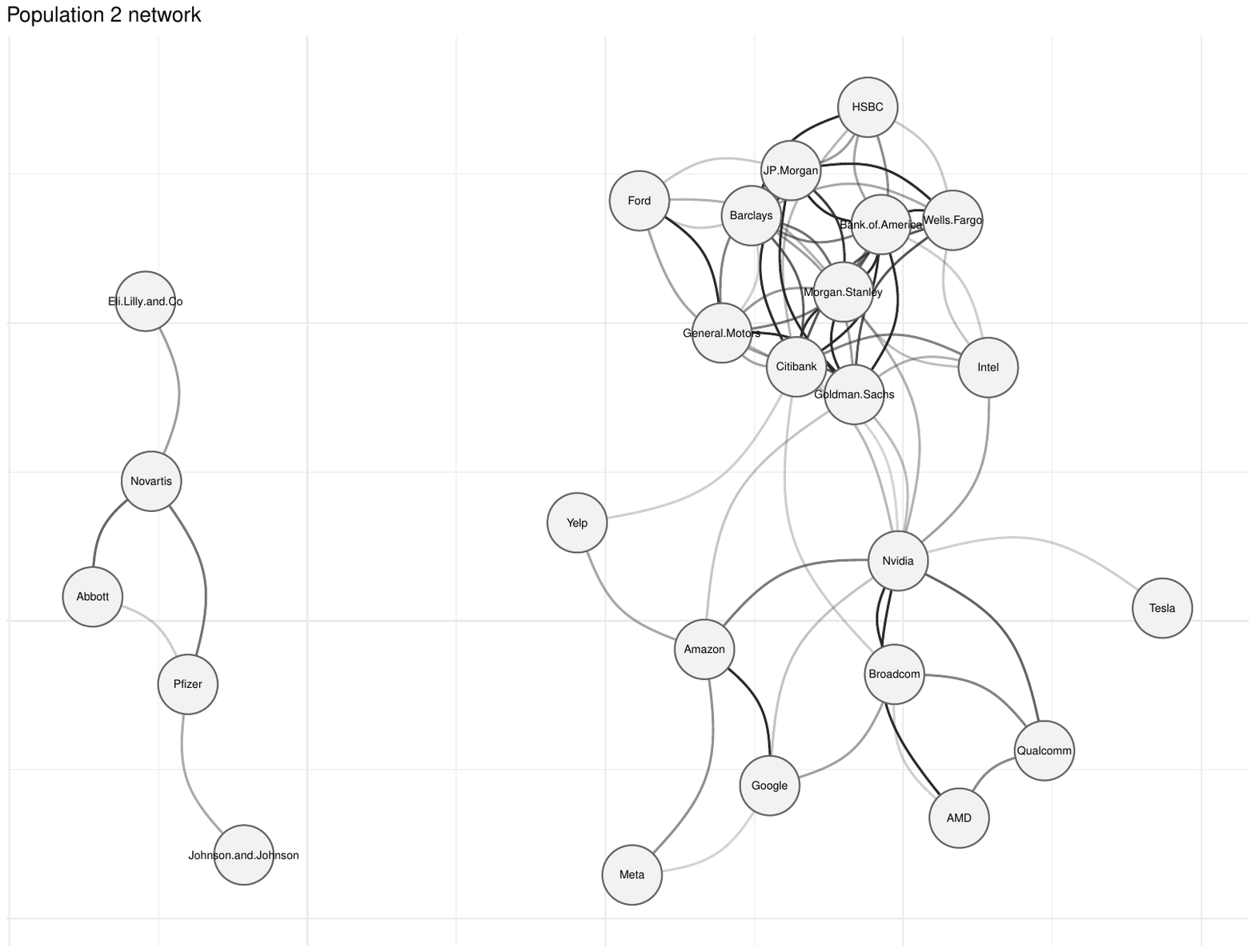}
\par\end{centering}
\label{Figure 24}
\end{figure}
\section{Conclusion}
In this paper, we have introduced a debiased framework for the group graphical lasso estimators to enable statistical inference across multiple population models under a shared sparsity structure. By establishing convergence rates and model consistency under irrepresentability conditions and between-group irrepresentability conditions, we have demonstrated that the debiased estimators achieve asymptotic Gaussianity, facilitating valid hypothesis testing for linear combinations of precision matrix entries across populations, specifically for edges shared with the graphs. We have extended our analysis to moderately high-dimensional regimes where irrepresentability conditions may not hold, showing consistency can still be attained when dimensions grow but at a slower rate than the sample sizes. Simulation studies validated the theoretical results, confirming the debiased estimators' superior performance in finite samples, while applications to real-life datasets highlighted the method's practical utility in uncovering shared dependencies across heterogeneous populations.

In summary, while \cite{danaher2014joint} introduced group graphical lasso for efficient estimation of multiple precision matrices assuming similar sparsity patterns, our work advances this by debiasing the estimates to facilitate rigorous statistical inference. This debiasing approach provides essential tools for hypothesis testing and confidence intervals in high-dimensional settings where sparsity is consistent across populations. By addressing the bias inherent in penalization, we bridge a key gap, enabling practitioners to draw reliable conclusions from multiple Gaussian graphical models beyond mere estimation.
\bibliographystyle{apalike}
\bibliography{References}
\end{document}